\documentclass[twoside,12pt]{amsart}

\usepackage{amsmath,latexsym,amssymb,mathptm, times,verbatim, enumerate,mathcomp}
    \usepackage{longtable}\usepackage{tikz,stmaryrd}
\input amssym.def
\input amssym
\input xypic
\input xy
\xyoption{all}
\newcommand{\gf}[2]{\genfrac{}{}{0pt}{}{#1}{#2}}

\def\grade{\operatorname{grade}}
\def\Pf{\operatorname{Pf}}

\newtheorem{theorem}{Theorem}[section]
\newtheorem{lemma}[theorem]{Lemma}

\newtheorem{proposition}[theorem]{Proposition}
\newtheorem{claim}[theorem]{Claim}
\newtheorem{observation}[theorem]{Observation}

\theoremstyle{definition}
\newtheorem{definition}[theorem]{Definition}
\newtheorem{remark}[theorem]{Remark}
\newtheorem{remarks}[theorem]{Remarks}
\newtheorem{data}[theorem]{Data}
\newtheorem{example}[theorem]{Example}
\newtheorem{examples}[theorem]{Examples}
\newtheorem{chunk}[theorem]{}
\newtheorem{hypotheses}[theorem]{Hypotheses}

\newtheorem*{Remark}{Remark}
\newtheorem*{Remarks}{Remarks}
\newtheorem*{proof2}{Proof of Lemma \ref{87.15}}
\newtheorem*{proof3}{The proof of  Theorem \ref{main} (\ref{main-cx}, \ref{main.a}, \ref{main.new-b})}
\newtheorem*{proof3.5}{The proof of  Theorem \ref{main} (\ref{main.b}, \ref{PS}, \ref{main.new-e}, \ref{main.new-f}, \ref{main.new-g})}
\newtheorem*{proof4}{Proof of Theorem \ref{main.c}  (\ref{main.c.iii})}
\newtheorem*{gunm-def}{Proof  of the assertions in \ref{gunm}}

\numberwithin{equation}{theorem}
\numberwithin{table}{theorem}

\newcounter{tempcounter}

\begin{document}

\baselineskip=16pt

\title[A matrix of linear forms which is annihilated by a vector of indeterminates]{\bf A matrix of linear forms which is annihilated by a vector of indeterminates}
\date\today 
\author[Andrew R. Kustin, Claudia Polini, and Bernd Ulrich]
{Andrew R. Kustin, Claudia Polini, and Bernd Ulrich}

\thanks{AMS 2010 {\em Mathematics Subject Classification}.
Primary 13D02; Secondary 13C40, 13A30.}

\thanks{The first author was partially supported by the Simons Foundation.
The second and third authors were partially supported by the NSF}


\address{Department of Mathematics, University of South Carolina,
Columbia, SC 29208} \email{kustin@math.sc.edu}

\address{Department of Mathematics, 
University of Notre Dame
Notre Dame, IN 46556} \email{cpolini@nd.edu}

\address{Department of Mathematics,
Purdue University,
West Lafayette, IN 47907}\email{ulrich@math.purdue.edu}

\begin{abstract} 
Let $R=k[T_1,\dots,T_{\goth f}]$ be a standard  graded polynomial ring over the field $k$ and  $\Psi$ be an $\goth f\times \goth g$ matrix of linear forms from $R$, where $1\le \goth g<\goth f$.
 Assume  
$\bmatrix T_1&\cdots & T_{\goth f}\endbmatrix\Psi$ is $0$ 
and  that $\operatorname{grade} I_{\goth g}(\Psi)$ is exactly one short of the maximum possible grade.
We  resolve $\overline{R}=R/I_{\goth g}(\Psi)$, prove that $\overline{R}$ has a $\goth g$-linear resolution, record explicit formulas for the $h$-vector and multiplicity of $\overline{R}$, and prove that if $\goth f-\goth g$ is even, then the ideal $I_{\goth g}(\Psi)$ is unmixed. Furthermore, if $\goth f-\goth g$ is odd, then we identify an explicit generating set for the unmixed part, $I_{\goth g}(\Psi)^{\text{unm}}$, of 
$I_{\goth g}(\Psi)$, resolve $R/I_{\goth g}(\Psi)^{\text{unm}}$, and record explicit formulas for the $h$-vector of $R/I_{\goth g}(\Psi)^{\text{unm}}$. (The rings $R/I_{\goth g}(\Psi)$ and $R/I_{\goth g}(\Psi)^{\text{unm}}$ automatically have the same multiplicity.)
These results have applications to the study of the blow-up algebras associated to  
linearly presented grade three Gorenstein ideals. 
\end{abstract}

\maketitle

\tableofcontents
\section{Introduction.}\label{Introduction} 

It is shown in  \cite{Ei87} that, ``Determinantal ideals associated to `sufficiently general' matrices of linear forms are  {\it resilient} in the sense that they remain of the `expected' codimension, or prime, even modulo a certain number of linear forms.'' 
The cases where the original matrix is generic, symmetric generic, 
catalecticant, or $1$-generic have been particularly well-studied; see for example, \cite{Ei88, GS, K08, R, SS}. We study a family of resilient determinantal ideals of matrices of linear forms. The ideals we consider  are described by a structure condition and a grade condition. Each of these ideals has a linear resolution and defines a ring which is Cohen-Macaulay on the punctured spectrum, but is  not Cohen-Macaulay. The previously studied resilient ideals associated to  a matrix of linear forms have all defined rings which are Cohen-Macaulay.

For the time being, let $R$ be a  Noetherian ring, $T_1, \ldots , T_{\goth f}$ an $R$-regular sequence, and $\Psi$ an $\goth f\times \goth g$ matrix with entries in $R$, 
where $2\le \goth g<\goth f$. (Throughout most of the paper, $\goth g$ is allowed to take the value $1$; but to simplify the exposition, in the beginning of the introduction, we insist that $\goth g$ be at least 2.) Assume that the matrix $\Psi$ has two properties. First of all, assume that
$$\bmatrix T_1&\cdots & T_{\goth f}\endbmatrix\Psi=0 \, .$$ 
It is clear from this property that the ideal $I_{\goth g}(\Psi)$ generated by the maximal minors of $\Psi$ does not have the maximum possible grade, $ \goth f-\goth g +1 .$
The second property of $\Psi$ is that $\operatorname{grade}  I_{\goth g}(\Psi)$ is only one short of the maximum possible grade, that is, $\operatorname{grade}  I_{\goth g}(\Psi)= \goth f-\goth g .$
Let $\overline{R}$ be the $R$-algebra $R/I_{\goth g} (\Psi)$ and let $\delta= \goth f-\goth g$.

For each integer $\epsilon$ in the set $\{\frac{\delta-1}2,\frac{\delta}2,\frac{\delta+1}2\}$,
we construct a complex ${\mathbb M^{\epsilon}}$.
When $\epsilon$ is either $\frac{\delta}2$ or $\frac{\delta+1}2$, that is when $\epsilon =\lceil\frac{\delta}2\rceil$, then
 ${\mathbb M^{\epsilon}}$ is a resolution of $\overline{R}$ by free $R$-modules. Though the complex $\mathbb M^{\epsilon}$ may not be minimal, 
it can be used to see that the projective dimension of $\overline{R}$ over $R$ is $\goth f -1$ or ${\goth f}$ depending on whether $\delta$ is 
even or odd. We deduce that $\overline R$ is never perfect as an $R$-module, but is grade unmixed if $\delta$ is even, indeed 
${\rm depth} \ R_{\mathfrak p} = \delta$ for every associated prime ${\mathfrak p}$ of $\overline R.$ 
In particular, $I_{\goth g}(\Psi)$ is unmixed whenever $\delta$ is even and $R$ is Cohen-Macaulay. Despite the failure of perfection, 
the resolution $\mathbb M^\epsilon$ specializes, in the sense that $S\otimes_R \mathbb M^\epsilon$ is a resolution of $S\otimes_R{\overline R}$ by free $S$-modules
whenever $S$ is a Noetherian $R$-algebra, $T_1, \ldots, T_{\goth f}$ forms an $S$-regular sequence, and $\delta\le \operatorname{grade}  S  I_{\goth g}(\Psi)$.

We can say more if the ring $R$ is non-negatively graded with $R_0$ a field and $T_1, \ldots, T_{\goth f}$ as well as the entries of 
$\Psi$ are linear forms. In this case, the minimal homogeneous free $R$-resolution of $\overline{R}$ is $\goth g$-linear. The ring $\overline{R}$ is not Cohen-Macaulay; so  the Betti numbers in the minimal resolution of $\overline{R}$ and the multiplicity of  $\overline{R}$ are not an obvious consequence of the fact that the minimal homogeneous free $R$-resolution of $\overline{R}$ is $\goth g$-linear. Nonetheless,
 we provide explicit formulas for the $h$-vector and the multiplicity of $\overline{R};$ these results do not depend on the parity of $\delta$.

The multiplicity of $\overline{R}$ and the unmixedness of $I_{\goth g} (\Psi)$ play  critical roles 
in  \cite{KPU-BA}, where we describe explicitly, in terms of generators and relations, the special fiber ring and the Rees algebra of  
linearly presented grade three Gorenstein ideals in polynomial rings of odd Krull dimension.

On the other hand, if $\delta$ is odd, then the ideal $I_{\goth g}(\Psi)$ is mixed. We identify the grade unmixed part, $I_{\goth g}(\Psi)^{\text{gunm}}$,  of $I_{\goth g}(\Psi)$ (see (\ref{gunm})),  prove that $\mathbb M^{\epsilon}$, with $\epsilon=\frac{\delta-1}2$, is a resolution of $R/I_{\goth g}(\Psi)^{\text{gunm}}$, and provide explicit formulas for the $h$-vector of $R/I_{\goth g}(\Psi)^{\text{gunm}}$. (The rings $\overline{R}=R/I_{\goth g}(\Psi)$ and $R/I_{\goth g}(\Psi)^{\text{gunm}}$ automatically have the same multiplicity.)  Ultimately, the parameter $\epsilon$ takes the value 
$\lceil\frac{\delta}2\rceil$ or
$\lceil\frac{\delta-1}2\rceil$. If $\epsilon =\lceil\frac{\delta}2\rceil$, then $\mathbb M^\epsilon$ resolves $\overline{R}$; and if 
$\epsilon=\lceil\frac{\delta-1}2\rceil$, then $\mathbb M^\epsilon$ resolves $R/I_{\goth g}(\Psi)^{\text{gunm}}$. Of course, if $\delta$ is even, then  $\lceil\frac{\delta-1}2\rceil$ and $\lceil\frac{\delta}2\rceil$ are equal as are $I_{\goth g}(\Psi)$ and $I_{\goth g}(\Psi)^{\text{gunm}}$.

\bigskip

Take $\epsilon=\lceil\frac{\delta}2\rceil$. The resolution ${\mathbb M^\epsilon}$ was obtained in stages. First we identified a fairly standard complex $\operatorname{Tot}({\mathbb B})$ with $\operatorname{H}_0(\operatorname{Tot}({\mathbb B}))=\overline{R}$. The double complex ${\mathbb B}$ is a quotient of 
of a larger double complex $\mathbb V$: the columns of $\mathbb V$ are Koszul complexes and the rows of $\mathbb V$ are truncations of generalized Eagon-Northcott complexes.
Alas, $\operatorname{Tot}({\mathbb B})$ has higher homology. Indeed, when $\delta$ is even, $\operatorname{H}_1(\operatorname{Tot}({\mathbb B}))\neq 0$. A long look at $\operatorname{H}_1(\operatorname{Tot}({\mathbb B}))$ told us that this homology came from Pfaffians. Macaulay2 experimentation lead us to a complex $\operatorname{Tot}({\mathbb T^{\epsilon}})$ whose homology is equal to $\operatorname{H}_{1\le}(\operatorname{Tot}({\mathbb B}))$. Curiously, ${\mathbb T^\epsilon}$ is also a quotient of $\mathbb V$. We became convinced that there exists a map of complexes
$$\xymatrix{\operatorname{Tot}({\mathbb T^\epsilon})\ar[d]^{\xi^{\epsilon}}\\\operatorname{Tot}({\mathbb B})}$$ which induces an isomorphism on all of the higher homology. Eventually, we were able to record a formula for $\xi^{\epsilon}$. The mapping cone, ${\mathbb L^\epsilon}$, of $\xi^{\epsilon}$ is an infinite resolution of $\overline{R}$. The resolution ${\mathbb M^\epsilon}$ is a finite subcomplex of ${\mathbb L^\epsilon}$ of the proper length which has the same homology as ${\mathbb L^\epsilon}$.

After we had found the resolution $\mathbb M^{\lceil\frac{\delta}2\rceil}$ of $\overline{R}$, and realized that $I_{\goth g} (\Psi)$ is mixed when $\delta$ is odd, we looked for the generators of $I_{\goth g}(\Psi)^{\text{gunm}}/I_{\goth g}(\Psi)$. In this search, we were inspired by the work of Mark Johnson \cite{J} and Susan Morey \cite{M} in the case $\goth f=\goth g+1$. 
We illustrate our answer in the context of the  
motivating situation from \cite{KPU-BA}. Let $R_0$ be a field, $R_1$ be the polynomial ring $R_0[T_1,\dots,T_\goth f]$, and $R$ be the bi-graded polynomial ring $R=R_0[X_1,\dots, X_\goth g,T_1,\dots,T_\goth f]$ where each $X_i$ has bi-degree $(1,0)$ and each $T_i$ has bi-degree $(0,1)$.  Let $\phi$ be an $\goth f\times \goth f$ alternating matrix 
with bi-homogeneous entries of degree $(1,0)$ and $\Psi$ be an $\goth f\times \goth g$ matrix with bi-homogeneous entries of degree $(0,1)$. Assume that $\goth f+\goth g$ is odd and 
$$\phi\bmatrix T_1\\\vdots\\T_\goth f\endbmatrix=\Psi\bmatrix X_1\\\vdots \\X_\goth g\endbmatrix.$$  Consider the $(\goth f+\goth g)\times(\goth f+\goth g)$ alternating matrix $$B=\bmatrix \phi&\Psi\\-\Psi^{\rm t}&0\endbmatrix.$$ Let $B_i$ be $(-1)^{i+1}$ times the Pfaffian of $B$ with row and column $i$ removed.  
View $B_{\goth f+\goth g}$ as a polynomial in $R_1[X_1,\dots,X_\goth g]$ and  
let $C=c_{R_1}(B_{\goth f+\goth g})$ be
 the content of $B_{\goth f+\goth g}$. It follows that $C$ is a homogeneous ideal in $R_1=R_0[T_1,\dots,T_\goth f]$ generated by forms  of degree $\goth g-1$. At this point, we use elementary techniques to show that 
\begin{equation}\label{at this point}C+I_\goth g(\Psi)\subseteq  I_\goth g(\Psi):(T_1,\dots,T_\goth f).\end{equation}
In this paper we prove that if $\delta\le \operatorname{grade} I_{\goth g}(\Psi)$, then the ideals
$$ C+I_\goth g(\Psi),\quad I_\goth g(\Psi):(T_1,\dots,T_\goth f),\quad I_\goth g(\Psi):(T_1,\dots,T_\goth f)^\infty, \quad\text{and}\quad I_{\goth g}(\Psi)^{\text{gunm}}$$  of $R_1$  all are equal. These facts are used in \cite{KPU-BA} 
where we describe explicitly, in terms of generators and relations, the special fiber ring and the Rees algebra of  
linearly presented grade three Gorenstein ideals in polynomial rings of even Krull dimension.

To prove (\ref{at this point}), first observe first that the column vector 
$$\bmatrix T_1&\cdots&T_\goth f&-X_1&\cdots&-X_\goth g\endbmatrix^{\rm t}$$ 
is in the null space of $B$; then use lower order Pfaffians of $B$ to see that for each index $i$, with $1\le i\le \goth f$,
the row vector
$$[0,\dots,0,B_{\goth f+\goth g},0,\dots,0,-B_i],$$
with $B_{\goth f+\goth g}$ in position $i$, is in the row space of $B$.
It follows that $B_{\goth f+\goth g}T_i+B_iX_\goth g=0$ for $1\le i\le \goth f$. 
The definition of   $B_i$ shows that $B_i$ is in the ideal $I_\goth g(\Psi)$; hence $T_i B_{\goth f+\goth g}\in I_\goth g(\Psi)$ and 
$$T_ic_{R_1}(B_{\goth f+\goth g})=c_{R_1}(T_i B_{\goth f+\goth g})\subseteq I_{\goth g}(\Psi).  $$ This completes the proof of (\ref{at this point}).

The generating set for $I_{\goth g}(\Psi)^{\text{gunm}}/I_{\goth g}(\Psi)$ that we use throughout most of  the paper is almost coordinate-free; consequently, at first glance, it looks much different than ``the content of $B_{\goth f+\goth g}$''; however the generating set  that we use has the advantage that only small modifications of the resolution $\mathbb M^{\lceil\frac{\delta}2\rceil}$ of $\overline{R}$ produces the resolution $\mathbb M^{\lceil\frac{\delta-1}2\rceil}$ of $R/I_{\goth g}(\Psi)^{\text{gunm}}$. We return to a coordinate-dependent matrix version of the results in Section~\ref{BA}.

There are precedents for the resolution of $J^{\text{gunm}}$ to be obtained from the resolution of $J$ using only small modifications: the modules are changed slightly, the maps are changed slightly, but the form of the resolution is not really altered. Recall, for example, the type two almost complete intersection ideals and the deviation two Gorenstein ideals of Huneke-Ulrich \cite{HU85}. The Gorenstein ideals are the unmixed part of ideals which have the same form as the almost complete intersection ideals. Furthermore, the resolutions of the Gorenstein ideals and the almost complete intersections have the same form; see \cite{K86,K92,K95}. Indeed, these ideals are the case $\goth g=1$ of the ideals $I_{\goth g}(\Psi)$ and $I_{\goth g}(\Psi)^{\text{gunm}}$ which are resolved in the present paper.

The main results of the paper, Theorems \ref{main} and \ref{main.c}, are stated in Section~\ref{results}. Section~\ref{not-con} consists of conventions, notation, and preliminary results. The maps and modules of ${\mathbb M^\epsilon}$, all of the numerical information about ${\mathbb M^\epsilon}$, and examples of ${\mathbb M^\epsilon}$ are given in Section~\ref{Med}. The double complexes ${\mathbb B}$, ${\mathbb T^\epsilon}$, and $\mathbb V$ are introduced in Section~\ref{Top-and-Bot}, where we also compute the homology of $\operatorname{Tot}({\mathbb B})$ and $\operatorname{Tot}({\mathbb T^\epsilon})$. The final piece of the puzzle, $\xi^{\epsilon}$, is defined in Section~\ref{xi}.  The proofs of  Theorem~\ref{main} and part (\ref{main.c.iii}) of Theorem~\ref{main.c} appear near the end of Section~\ref{xi}. The Hilbert series part of 
Theorem~\ref{main.c}, assertion (\ref{'23.19.c.ii}), is established in Section~\ref{h-vect}. In Section~\ref{BA} we return to the situation of blowup algebras and organize our conclusions in the exact language of \cite{KPU-BA}. We also give the complete proof that equality holds in (\ref{at this point}).

\bigskip

\section{Main  results.}\label{results}

The word ``matrix'' does not appear in the official data of the paper (Data~\ref{87data6}). We explain the transition. 
 Each column of $\Psi$, from Section~\ref{Introduction}, is annihilated by the row vector 
\begin{equation}\label{87the-Tees}\bmatrix T_1&\dots&T_{\goth f}\endbmatrix.\notag\end{equation} It follows that each column  of $\Psi$ is equal to
$$ \text{ an alternating matrix times}\bmatrix T_1\\\vdots \\ T_{\goth f}\endbmatrix.$$
In order to understand $\Psi$ one must consider the $\goth g$ relevant alternating matrices. 
We deal with Pfaffian identities involving multiple alternating matrices by making use of the divided power structure on the subalgebra $\bigoplus_n \bigwedge^{2n}F$ of an exterior algebra $\bigwedge^{\bullet}F$. 
 In other words, 
we work with an $R$-module homomorphism ${\mu: G^* \to \bigwedge^2 F}$, where  
$G$ and $F$ are  free $R$-module of rank $\goth g$ and $\goth f$, respectively.    (Notice that if one picks bases for $G^*$ and $F$, then $\mu$ selects $\goth g$ alternating matrices, one for each column of $\Psi$.) 

If $\Psi$ only has one column, then  the situation  has been  studied extensively; see, for example \cite{HU85,K86,S,K92,K95}.  Curiously, many of the ideas involved in the present project are already present in the case where $\Psi$ only has  one column.

Every construction in sections \ref{results} through \ref{h-vect} 
is built using Data~\ref{87data6}
or Data~\ref{unm-data}. 
(We return to a coordinate-dependent matrix version of the results in Section~\ref{BA}.)
The interesting results are established after 
we impose the  hypotheses of \ref{87hyp6}.

\begin{data}\label{87data6}Let $R$ be a commutative Noetherian ring,  $F$ and $G$ be free $R$-modules of rank $\goth f$ and $\goth g$, respectively, with $1\le \goth g< \goth f$, $\tau$ be an element of $F^*$, and $\mu:G^*\to \bigwedge^2F$ be an $R$-module homomorphism. Define $\Psi:G^*\to F$ by $\Psi(\gamma)=\tau(\mu(\gamma))$ for $\gamma$ in $G^*$. Let $\delta$ represent $\goth f-\goth g$, $\epsilon$ represent an integer with $\lceil\frac{\delta-1}2\rceil\le \epsilon\le \lceil \frac{\delta}2\rceil$, and  $\overline{R}$ be the $R$-algebra $R/I_{\goth g} (\Psi)$.
\end{data}

\begin{hypotheses}\label{87hyp6} Adopt Data~\ref{87data6}. 
Assume that \begin{enumerate}[(a)]\item\label{87hyp6-a} $\delta\le \operatorname{grade} I_{\goth g}(\Psi)$, (this bound is one less than the maximal possible grade for $I_{\goth g}(\Psi)$), and
\item\label{87hyp6-b} $\goth f\le \operatorname{grade} I_1(\tau)$. \end{enumerate}
\end{hypotheses}

\medskip

Data~\ref{unm-data} is more complicated than Data~\ref{87data6}, is meaningful only when  $\delta$ is odd, and pertains only to the generating set for $I_\goth g(\Psi)^{\text{gunm}}$ and the first differential in $\mathbb M^\epsilon$ when $\epsilon=\frac{\delta-1}2$. (Recall that $I_\goth g(\Psi)^{\text{gunm}}$ is the grade unmixed part of $I_\goth g(\Psi)$; see \ref{gunm} for more details.) It is possible to ignore Data~\ref{unm-data} and still read most of the paper.
 The data of \ref{87data6} is coordinate-free; however the data of \ref{unm-data} is  not completely coordinate-free; it requires that one pair of dual basis elements  $(Y_1,X_1)$, in $G\times G^*$ be distinguished. 
The ideal $\goth C$ of $R$, as described in Data~\ref{unm-data}, does not depend on the choice of  $(Y_1,X_1)$. (See the discussion surrounding (\ref{at this point}) for a completely different description of $I_{\goth g}(\Psi)^{\text{gunm}}/I_{\goth g}(\Psi)$ in a special case; but do notice that in (\ref{at this point}) one column of $\Psi$ is treated in a distinguished manner and any other column of $\Psi$ would work just as well.)

\begin{data}\label{unm-data} Adopt Data~{\rm\ref{87data6}}. 

\smallskip\noindent(a) Assume  $\delta$ odd. Decompose $G$ and $G^*$ as 
\begin{equation}\label{decomp}G=RY_1\oplus G_+\quad\text{and}\quad G^*=RX_1\oplus {G_+}^*\end{equation} with $Y_1(X_1)=1$, $Y_1({G_+}^*)=0$, and $G_+(X_1)=0$, where $Y_1\in G$, $X_1\in G^*$,  $G_+$ and ${G_+}^*$  are free submodules of $G$ and $G^*$, respectively, of rank $\goth g-1$.  Define an $R$-module (integration) homomorphism $\int_{X_1}:D_iG^*\to D_{i+1}G^*$ by way of the decomposition (\ref{decomp}); that is, 
$$\textstyle\int_{X_1}\Big(\sum\limits_{j=0}^iX_1^{(j)}\otimes \gamma_{i-j,+}\Big)=\sum\limits_{j=0}^iX_1^{(j+1)}\otimes \gamma_{i-j,+},$$ with $\gamma_{i-j,+}\in D_{i-j}({G_+}^*)$.  
Notice that \begin{equation}\label{NT} \textstyle Y_1(\int_{X_1} \gamma_i)=\gamma_i \quad\text{for all $\gamma_i\in D_i(G^*)$.}\end{equation} 
 Define \begin{equation}\label{c-2.3}\textstyle  c: D_{\frac{\delta-1}2}(G^*)\to \bigwedge^{\goth f}F\otimes \bigwedge^{\goth g}G\end{equation} to be the $R$-module homomorphism
\begin{equation}\label{2.3.3.5}\textstyle c(\gamma_{\frac{\delta-1}2}) =[D(\mu)](\int_{X_1} \gamma_{\frac{\delta-1}2})\wedge  (\bigwedge^{\goth g-1}\Psi)(\omega_{{G_+}^*})
\otimes \omega_{G_+}\wedge  Y_1.\end{equation} 

\smallskip \noindent (b) Define $\goth C$ to be the following ideal of $R$
 \begin{equation}\label{unm-data-b}\goth C=\begin{cases}\operatorname{ann}_R (\operatorname {coker} c)&\text{if $\delta$ is odd}\\
0&\text{if $\delta$ is even}\end{cases}\end{equation} and define $\widetilde{R}=R/(I_{\goth g}(\Psi)+\goth C)$.
 \end{data}
\begin{remarks} \begin{enumerate}[\rm(a)]
\item The homomorphism $c$ of (\ref{c-2.3}) and is defined only when $\delta$ is odd. 
\item If $\delta$ is odd, then the target of $c$ is a free $R$-module of rank one; so our definition of $\goth C$ is a coordinate-free way of saying that $\goth C$ is ``the image of $c$''.
\item We show in (\ref{much-promised}) that  $\goth C\subseteq I_\goth g(\Psi):I_1(\tau)$.
\item 
Some hints about
why the generators for $I_\goth g(\Psi)^{\text{gunm}}/I_\goth g(\Psi)$ as calculated in (\ref{at this point}) are the same as the generators of $\goth C$ are given in (\ref{no-Z}) and (\ref{perfect-gens}).  A complete proof may be found in Theorem~\ref{unmpart}.\ref{8.3.a}.
\end{enumerate}\end{remarks}

Theorems~\ref{main} and \ref{main.c} are the main results of the paper. 
The short version of these theorems is that we resolve $R/I_\goth g(\Psi)$ and $R/I_\goth g(\Psi)^{\text{gunm}}$, where $I_\goth g(\Psi)^{\text{gunm}}$ is the grade unmixed part of the ideal $I_\goth g(\Psi)$ (see, for example, \ref{gunm}), and we record the Hilbert series, multiplicity, and $h$-vectors of these rings.

\begin{theorem}\label{main} Adopt Data~{\rm{\ref{87data6}}} and 
{\rm{\ref{unm-data}}}.  
The following statements hold.

\begin{enumerate}[\rm(a)]\item\label{main-cx}
The maps and modules of ${\mathbb M^\epsilon}$, given in Definition~{\rm\ref{doodles'}}, form a complex. 
\setcounter{tempcounter}{\value{enumi}}
\end{enumerate}

\medskip\noindent
Assume that Hypotheses {\rm\ref{87hyp6}} are in effect for the rest of the statements.

\medskip
\begin{enumerate}[\rm(a)]
\setcounter{enumi}{\value{tempcounter}}
\item\label{main.a} If $\epsilon=\lceil \frac \delta 2\rceil$, then the complex  ${\mathbb M^{\epsilon}}$ is  a resolution of $\overline{R}$ by free $R$-modules.

\item\label{main.new-b} If $\epsilon=\lceil \frac {\delta -1}2\rceil$, then the complex  ${\mathbb M^{\epsilon}}$ is  a resolution of $\widetilde{R}$ by free $R$-modules.
\setcounter{tempcounter}{\value{enumi}}
\end{enumerate}

\medskip \noindent
Assume that  $I_1(\tau)$ is a proper ideal of $R$ for the rest of the statements.

\medskip
\begin{enumerate}[\rm(a)]
 \setcounter{enumi}{\value{tempcounter}}
\item\label{main.b} The projective dimension of  the $R$-module $\overline{R}$ is equal to
$$\begin {cases} \goth f-1, &\text{if $\delta$ is even, and}\\
\goth f,&\text{if $\delta$ is odd,}\end{cases}$$
and the  
projective dimension of  the $R$-module $\widetilde{R}$ is equal to
$$\begin {cases} \goth f-1, &\text{if $\delta$ is even,}\\
\goth f-2,&\text{if $\delta$ is odd, $3\le \delta$, and $2\le \goth g$,}\\
\goth f-1,&\text{if $\delta$ is odd,  $\goth g=1$, and $\operatorname{grade} I_\goth g(\Psi)\le \delta$, and}\\
1,&\text{if $\delta=1$ and $2\le \goth g$}.
\end{cases}$$

\item\label{PS} If $\delta$ is even, then $\operatorname{depth} R_{\mathfrak p} =\delta$ for every $\mathfrak p \in \operatorname{Ass}_R(\overline{R})$.

\item \label{main.new-e}Assume that one of the following three hypotheses is in effect:
\begin{enumerate}[\rm(i)]\item\label{main.new-e.i} $\delta$ is even, or \item\label{main.new-e.ii} $\delta$ is odd and $2\le \goth g$, or \item\label{main.new-e.iii} $\delta$ is odd, $\goth g=1$, and $\operatorname{grade} I_{\goth g}(\Psi)\le \delta$.\end{enumerate}
Then 
$\operatorname{depth}R_{\mathfrak p} =\delta$ for every $\mathfrak p \in \operatorname{Ass}_R(\widetilde{R})$.
\item \label{main.new-f}If one of the three hypotheses of \rm{(\ref{main.new-e})} is in effect, then the ideals
$$\goth C+I_{g}(\Psi),\quad I_{g}(\Psi):I_{1}(\tau),\quad I_{g}(\Psi):I_{1}(\tau)^\infty,\quad\text{and}\quad I_{g}(\Psi)^{\rm{gunm}}$$ of $R$ are equal;  and in particular, $\widetilde{R}=R/I_{g}(\Psi)^{\text{\rm gunm}}$.
\item \label{main.new-g} If one of the three hypotheses of \rm{(\ref{main.new-e})} is in effect and 
 $\epsilon=\lceil \frac {\delta -1}2\rceil$, then the complex  ${\mathbb M^{\epsilon}}$ is  a resolution of $R/I_{\goth g}(\Psi)^{\text{\rm gunm}}$ by free $R$-modules.
\end{enumerate}\end{theorem}
\begin{remark}\label{promise}We comment on Hypothesis \ref{main}.\ref{main.new-e.iii}.
We note that if $\goth g=1$ and $\delta$ is odd, then it is possible for $I_\goth g(\Psi)$ to have grade $\delta+1$; for example if
$$\Psi=\bmatrix 0&I_{\frac {\goth f}2}\\-I_{\frac {\goth f}2}&0\endbmatrix \bmatrix T_1\\\vdots\\ T_{\goth f}\endbmatrix.$$In this situation, $\widetilde{R}$ is not of any interest. A complete characterization  of this situation is given in Observation~\ref{6.rmk}. \end{remark}

\bigskip Theorem~\ref{main.c} gives the Hilbert series of three families of rings because Theorem~\ref{main} shows that
$$\operatorname{H}_0(\mathbb M^\epsilon)= \begin{cases} R/I_{\goth g}(\Psi)\text{ in the case that $I_{\goth g}(\Psi)$ is not a grade unmixed ideal,}&\text{if $\epsilon=\frac{\delta+1}2$,}\\
R/I_{\goth g}(\Psi)\text{ in the case that $I_{\goth g}(\Psi)$ is a grade unmixed ideal,}&\text{if $\epsilon=\frac{\delta}2$,}\\
R/I_{\goth g}(\Psi)^{\text{gunm}}\text{ in the case that $I_{\goth g}(\Psi)$ is not a grade unmixed ideal,}&\text{if $\epsilon=\frac{\delta-1}2$.}
\end{cases}$$

\begin{theorem}\label{main.c} Adopt Data~{\rm{\ref{87data6}}} and 
{\rm{\ref{unm-data}}}.   Assume that Hypotheses {\rm\ref{87hyp6}} are in effect, that $R$ is a non-negatively graded ring, and that  $\tau:R(-1)^{\goth f}\to R$ and $\mu:R^{\goth g}\to R^{\binom{\goth f}2}$ are homogeneous {\rm(}degree preserving{\rm)} $R$-module homomorphisms.  Then the following statements also hold. \begin{enumerate}[\rm(a)]
\item\label{'23.19.c.ii} Assume that $2\le \goth g$, or else, that $\goth g=1$ and $\goth f$ is odd. Then the Hilbert series of $\operatorname{H}_0(\mathbb M^\epsilon)$ is equal to $$\operatorname{HS}_{\operatorname{H}_0(\mathbb M^\epsilon)}(s)=\operatorname{HS}_{R}(s)\cdot \operatorname{HN}_{\operatorname{H}_0(\mathbb M^\epsilon)}(s)$$ for $\operatorname{HN}_{\operatorname{H}_0(\mathbb M^\epsilon)}(s)=(1-s)^{\goth f-\goth g}\cdot \operatorname{hn}_{\operatorname{H}_0(\mathbb M^\epsilon)}(s)$ and
$$\begin{array}{ll}\operatorname{hn}_{\operatorname{H}_0(\mathbb M^\epsilon)}(s)&
=\begin{cases}
\phantom{+}(1-s)^{\goth g}\sum\limits_{j\le \epsilon-1}
(-1)^{\delta+1}\binom{\goth g+j-1}{j}s^{2j+2\goth g-\goth f}
\\
+\sum\limits_{\ell=0}^{\goth g-2}\binom{\ell+\goth f-\goth g-1}{\ell}s^\ell\\
+\sum\limits_{\ell=0}^{\goth f-\goth g-2}(-1)^{\ell+\delta}\binom{\goth g+\ell-1}\ell s^{2\goth g-\goth f+\ell}
\end{cases}\\ 

&=\begin{cases}\textstyle \phantom{+}\sum\limits_{\ell=0}^{\goth g-1}\binom{\ell+\goth f-\goth g-1}\ell s^\ell\\-\chi(\epsilon=\frac{\delta-1}2)\binom{\goth g+\epsilon-1}{\epsilon
}s^{\goth g-1}\\+\sum\limits_{\ell=\goth g}^{q(\goth g,\goth f)} 
\sum\limits_{j\le
\epsilon-1}(-1)^{\ell+\goth g+1}\binom{\goth g}{\ell-2\goth g+\goth f-2j}
\binom{\goth g+j-1}{j}s^{\ell},\end{cases}
\end{array}$$ where $$q(\goth g,\goth f)=\begin{cases}
2\goth g-3,&\text{if $\epsilon=\frac{\delta-1}2$,}\\
2\goth g-2,&\text{if $\epsilon=\frac{\delta}2$, and}\\2\goth g-1,&\text{if $\epsilon=\frac{\delta+1}2$}.\end{cases}$$
In particular, if $R$ is a standard graded polynomial ring of Krull dimension $\dim R$ over a field, then
\begin{enumerate}[\rm(i)]
 \item\label{1.3.ii.1} $$\operatorname{HS}_{\operatorname{H}_0(\mathbb M^\epsilon)}(s)=\frac{\operatorname{hn}_{\operatorname{H}_0(\mathbb M^\epsilon)}(s)}{(1-s)^{\dim R-\goth f+\goth g}},$$
\item\label{1.3.ii.1+} the $h$-vector of $\operatorname{H}_0(\mathbb M^\epsilon)$ is
$$\operatorname{hv}(\operatorname{H}_0(\mathbb M^\epsilon))=(h_0,\dots,h_{q(\goth g,\goth f)}),$$
with $$h_\ell=\begin{cases} \binom{\ell+\goth f-\goth g-1}\ell,&
\text{if $0\le \ell \le \goth g-2$,}\vspace{5pt}\\
\binom{\ell+\goth f-\goth g-1}\ell
-\chi(\epsilon=\frac{\delta-1}2)\binom{\goth g+\epsilon-1}{\epsilon
},&\text{if $\ell=\goth g-1$, and}\vspace{5pt}\\
\sum\limits_{j\le\epsilon -1}(-1)^{\ell+\goth g+1}\binom{\goth g}{\ell-2\goth g+\goth f-2j}
\binom{\goth g+j-1}{j},&\text{if $\goth g\le \ell\le q(\goth g,\goth f)$,}
\end{cases}$$
  and 
\item\label{1.3.ii.2} the multiplicity of $\operatorname{H}_0(\mathbb M^\epsilon)$ is $$e(\operatorname{H}_0(\mathbb M^\epsilon))=\operatorname{hn}_{\operatorname{H}_0(\mathbb M^\epsilon)}(1)= 
\sum\limits_{i=0}^{\lfloor \delta/2\rfloor}\binom{\goth f-2-2i}{\delta-2i},
$$
which is equal  to 
$$\begin{cases}
\text{the number of monomials of even degree at most $\delta$ in $\goth g-1$ variables,} &\text{if $\delta$ is even, or}\\
\text{the number of monomials of odd degree at most $\delta$ in $\goth g-1$ variables}, &\text{if $\delta$ is odd}.\\
\end{cases}$$
 \end{enumerate}
 \item\label{main.c.iii} If $R$ is a standard graded polynomial ring over a field, then the minimal resolution of    $\overline{R}$ by free $R$-modules is $\goth g$-linear. \end{enumerate}  
\end{theorem}

\begin{remarks}\label{rmk} \begin{enumerate}[\rm(a)]
 
\item\label{2.8.a} In item (\ref{'23.19.c.ii}) of Theorem~\ref{main.c}, we use the notation  of \cite[5.4.1]{KR} to denote  the Hilbert series $\operatorname{HS}_{\operatorname{H}_0(\mathbb M^\epsilon)}(s)$, the numerator of the Hilbert series $\operatorname{HN}_{\operatorname{H}_0(\mathbb M^\epsilon)}(s)$,  
the simplified Hilbert numerator $\operatorname{hn}_{\operatorname{H}_0(\mathbb M^\epsilon)}(s)$, and the 
$h$-vector $\operatorname{hv}(\operatorname{H}_0(\mathbb M^\epsilon))$, of $\operatorname{H}_0(\mathbb M^\epsilon)$. We gave two formulations for the simplified Hilbert numerator $\operatorname{hn}_{\operatorname{H}_0(\mathbb M^\epsilon)}(s)$: one yields the 
$h$-vector $\operatorname{hv}(\operatorname{H}_0(\mathbb M^\epsilon))$  quickly and the other yields the multiplicity $e(\operatorname{H}_0(\mathbb M^\epsilon))$ quickly.  Recall that the Hilbert series of a Noetherian graded ring $S=\bigoplus_{0\le i}S_i$, (with $S_0$ an Artinian local ring) is the formal power series  
$$\textstyle\operatorname{HS}_S(z)=\sum_i \lambda_{S_0}(S_i) z^i,$$
where $\lambda_{S_0}(\underline{\phantom{x}})$ represents the length of an $S_0$-module, and the multiplicity of $S$
is $$e(S)=(\dim S)!
\lim_{i\to\infty}\frac{\lambda_S(S/\mathfrak m^i S)}{i^{\dim S}},$$
where $\mathfrak m$ is the maximal homogeneous ideal of $S$ and  ``$\dim$'' represents  Krull dimension.

\item\label{2.8.b} If \begin{equation}\label{RMK}\goth g=1 \quad\text{and}\quad \goth f\text{ is even,}\end{equation} then the correct statement which is  analogous to (\ref{'23.19.c.ii}) of Theorem~\ref{main.c} is $$\operatorname{HN}_{\overline{R}}(s)=(1-s)^{\goth f} \quad\text{and}\quad \operatorname{hn}_{\overline{R}}(s)=1.$$ Indeed, the form for $\frac{\operatorname{HN}_{\overline{R}}(S)}{(1-s)^{\goth f-\goth g}}$, as given in (\ref{'23.19.c.ii}), continues to be correct, even in the presence of the hypotheses (\ref{RMK}); however, in the presence of (\ref{RMK}),
 $$\frac{\operatorname{HN}_{\overline{R}}(S)}{(1-s)^{\goth f-\goth g}}
=1-s;$$consequently, it is not appropriate to call the quotient $\operatorname{hn}_{\overline{R}}(s)$. See Remark~\ref{6.3} for more details.

On the other hand, when the hypothesis of (\ref{RMK}) are in effect, then there is no meaningful statement about $\operatorname{HN}_{\operatorname{H}_0(\mathbb M^{\epsilon})}(s)$, with  $\epsilon=\frac{\delta-1}2$, which is analogous to (\ref{'23.19.c.ii}) of Theorem~\ref{main.c}. In this situation, $I_{\goth g}(\Psi)^{\text{gunm}}=(I_{\goth g}(\Psi),\alpha)$ for some homogeneous element $\alpha$ in $R$ of degree $0$; see, for example, Example~\ref{exs-Med}.\ref{4.4.f}. Thus, the Hilbert series of $\operatorname{H}_0(\mathbb M^{\epsilon})$, computed using this grading, is the same as the Hilbert series of $0$.

\item\label{2.8.c} Assertion \ref{main.c}.\ref{main.c.iii} is not true, in general, for $R/I_{\goth g}(\Psi)^{\text{gunm}}$, when $I_{\goth g}(\Psi)^{\text{gunm}}\neq I_{\goth g}(\Psi)$ because, for example, in general,  $I_{\goth g}(\Psi)^{\text{gunm}}$ has generators of different degrees. See  Example~\ref{exs-Med}.\ref{4.4.c} or Example~\ref{exs-Med}.\ref{4.4.c-may-7}.
\end{enumerate}\end{remarks}

\bigskip

\section{Conventions,  notation, and preliminary results.}\label{not-con} Data~\ref{87data6} and \ref{unm-data} are in effect throughout this section.

\begin{chunk}Unless otherwise noted, all functors will be functors of $R$-modules; that is, $\otimes$, $\operatorname{Hom}$, $(\underline{\phantom{X}})^*$, $\operatorname{Sym}_i$, $D_i$, $\bigwedge^i$, and $:$ mean
$\otimes_R$, $\operatorname{Hom}_R$, $\operatorname{Hom}_R(\underline{\phantom{X}},R)$, $\operatorname{Sym}_i^R$, $D_i^R$,  $\bigwedge^i_R$, and $:_R$ respectively.\end{chunk} 
  
\begin{chunk}If $I$ and $J$ are ideals in a ring $R$, then the {\it saturation} of $J$ by $I$ in $R$ is 
$$J:I^\infty=\bigcup\limits_{n=1}^\infty J:I^n=\{r\in R\mid rI^n\subseteq J\text{ for some $n$}\}.$$ 
\end{chunk}

\begin{chunk} We denote the total complex of the double complex $X$ by $\operatorname{Tot}(X)$. \end{chunk} 

\begin{chunk} If $z$ is a cycle in a complex, then we denote the corresponding element of homology by $\llbracket z\rrbracket$.\end{chunk} 

 \begin{chunk} \label{87Not1} 
If $\delta$ is odd, then the decomposition of (\ref{decomp}) is used in the description of $c$ from (\ref{c-2.3}) and the description of $\goth C$ from (\ref{unm-data-b}); otherwise, 
our complexes are described in a coordinate-free manner. 
We make much use of the divided power structures on the algebras $D_{\bullet}(G^*)$ and $\bigwedge^{\bullet}F$; in particular, the $R$-module homomorphism $\mu:G^*\to \bigwedge^2 F$ automatically induces 
a homomorphism  $$D(\mu):D_{\bullet}(G^*)\to \textstyle\bigwedge^{2\bullet} F$$ of divided power algebras and the composition
$$D_{\bullet}(G^*)\xrightarrow{D(\mu)}\textstyle\bigwedge^{2\bullet} F\xrightarrow{\text{inclusion}}\textstyle\bigwedge^{\bullet} F,$$ which we also denote by $D(\mu)$, is used extensively in our calculations. 

 Let $Y_1,\dots,Y_{\goth g}$ be a basis for $G$ and $X_1,\dots,X_{\goth g}$ be the corresponding dual basis for $G^*$. 
\begin{equation}\label{rep}\text{Let $\binom{Y}{i}$ represent the set of monomials of degree $i$ in $Y_1,\dots,Y_{\goth g}$.}\end{equation} If ${m=Y_1^{a_1}\cdots Y_{\goth g}^{a_{\goth g}}}$ is in $\binom{Y}{i}$, then let $m^*$ represent the element $X_1^{(a_1)}\cdots X_{\goth g}^{(a_{\goth g})}$ of $D_i(G^*)$. Observe that $\{m^*\mid m\in \binom{Y}{i}\}$ is the basis for $D_i(G^*)$ which is dual to the basis  $\binom{Y}{i}$ of $\operatorname{Sym}_iG$.
Consider the evaluation map $\operatorname{ev}: \operatorname{Sym}_iG\otimes D_i(G^*)\to R$ and let 
$$\operatorname{ev}^*:R\to D_i(G^*)\otimes \operatorname{Sym}_iG$$ be the dual of $\operatorname{ev}$. Both of these $R$-module homomorphisms are completely independent of coordinates; and therefore the element
$$\operatorname{ev}^*(1)=\sum\limits_{m\in \binom{Y}{i}} m^*\otimes m\in D_i(G^*)\otimes \operatorname{Sym}_iG$$
is completely independent of coordinates; this element will also  be used extensively in our calculations.\end{chunk}

\begin{chunk}\label{omega} In a similar manner, if
  $\omega_{G^*}$ is a basis  for $\bigwedge^{\goth g}(G^*)$ and $\omega_G$ is the corresponding dual basis for $\bigwedge^{\goth g}G$, then the element
$\omega_{G^*}\otimes\omega_G$ is a canonical element of
$\bigwedge^{\goth g}(G^*)\otimes\bigwedge^{\goth g}G$. This element is also used in our calculations.\end{chunk}

\begin{chunk} We recall some of the properties of the divided power structure on the subalgebra $\bigwedge^{2\bullet}F$ of the exterior algebra $\bigwedge^{\bullet}F$. Suppose that $e_1,\dots,e_{\goth f}$ is a basis for the free $R$-module $F$ and $$f_2=\sum_{1\le i_1<i_2\le \goth f}a_{i_1,i_2} e_{i_1}\wedge e_{i_2}$$ is an element of $\bigwedge^2F$, for some $a_{i_1,i_2}$ in $R$. 
Let $A$ be the $\goth f\times \goth f$ alternating matrix with
$$A_{i,j}=\begin{cases} a_{i,j},&\text{if $i<j$,}\\0,&\text{if $i=j$, and}\\-a_{i,j},&\text{if $j<i$}.\end{cases}$$
For each positive integer $\ell$, the $\ell$-th divided power of $f_2$ is 
$$f_2^{(\ell)}=\sum\limits_I A_I e_I\in\textstyle\bigwedge^{2\ell}F,$$ where the $2\ell$-tuple $I=(i_1,\dots,i_{2\ell})$ roams over all increasing sequences  of integers with $1\le i_1$ and $i_{2\ell}\le \goth f$, $e_I=e_{i_1}\wedge \dots \wedge e_{i_{2\ell}}$, and $A_I$ is the Pfaffian of the submatrix of $A$ which consists of rows and columns $\{i_1,\dots,i_{2\ell}\}$, in the given order.  Furthermore, $\bigwedge^{2\bullet}F$ is a DG$\Gamma$-module over $\bigwedge^{\bullet}F^*$. In particular, if $\tau\in F^*$ and $v_1,\dots,v_s$ are homogeneous elements of  $\bigwedge^{2\bullet}F$, then 
$$\tau\left(v_1^{(\ell_1)}\wedge \dots\wedge v_{s}^{(\ell_s)}\right)=
\sum\limits_{j=1}^s\tau(v_j)\wedge v_1^{(\ell_1)}\wedge \dots\wedge
v_j^{(\ell_j-1)}
\wedge \dots\wedge v_{s}^{(\ell_s)}.$$ For more details see, for example, 
\cite[Appendix~A2.4]{Ei95} or \cite[Appendix and Sect. 2]{BE77}.
\end{chunk}

\begin{chunk}\label{V-bullet} If $I$ and $J$ are integers, then $V_{I,J}$, $V_{I,J}^{{\mathbb T}}$ and  $V_{I,J}^{{\mathbb B}}$
all represent the free $R$-module $$\textstyle \bigwedge^IF\otimes D_J(G^*).$$ 
Of course the rank of $V_{I,J}^\bullet$ is $\binom{\goth f}I\binom{\goth g+J-1}J$ for any choice of $\bullet$; that is $\bullet$ might be $\mathbb T$,  $\mathbb B$, or empty.
\end{chunk} 

\begin{chunk}\label{2.7} We always use $f_i$ for an arbitrary element of $\bigwedge^iF$ and $\gamma_j$ for an arbitrary element of $D_j(G^*)$.\end{chunk}  
\begin{chunk}\label{2.8}The notation
$$\bigoplus\limits_{\gf{\epsilon\le J}{I+J\le \delta-1}}
\qquad\text{ means }\qquad \bigoplus\limits_{
{\textstyle \{(I,J)\mid \epsilon\le J\text{ and }I+J\le \delta-1\}}.}$$ \end{chunk} 

\begin{chunk}\label{chi}If $S$ is a statement then 
$$\chi(S)=\begin{cases} 1,&\text{if $S$ is true,}\vspace{5pt}\\0,&\text{if $S$ is false.}\end{cases}$$\end{chunk}

\begin{chunk} If $M$ is a matrix (or a homomorphism of free $R$-modules), then $I_r(M)$ is the ideal generated by the
$r\times r$ minors of $M$ (or any matrix representation of $M$). We denote the transpose of a matrix $M$ by $M^{\rm t}$. \end{chunk}
\begin{chunk}\label{gunm-jr}The {\it grade} of a proper ideal  $I$ in a Noetherian ring $R$  is the length of a maximal  $R$-regular sequence in $I$. The unit ideal $R$  of $R$ is regarded as an ideal  of infinite grade.
\end{chunk}

\begin{chunk}\label{gunm} Let $I$ be a proper ideal in a Noetherian ring $R$.
 The ideal $I$ is {\it grade unmixed} if $\operatorname{grade} {\mathfrak p}=\operatorname{grade} I$ for all associated prime ideals ${\mathfrak p}$ of $R/I$. The {\it grade unmixed part} of $I$ is the ideal $I^{\text{gunm}}$ which satisfies either of the following two equivalent conditions:
\begin{enumerate}[\rm(a)]
\item \label{17.15.a} $I^{\text{gunm}}$ is the smallest ideal $K$ with $I\subseteq K$, $\operatorname{grade} K=\operatorname{grade} I$, and $K$ is grade unmixed, or
\item\label{17.15.b} $I^{\text{gunm}}$ is the largest ideal $K$ with $I\subseteq K$ and $\operatorname{grade} I<\operatorname{grade} (I:K)$.
\end{enumerate}

\noindent
Furthermore, if $K$ is any grade unmixed ideal of $R$ with $I\subseteq K$ and $$\operatorname{grade} K=\operatorname{grade} I<\operatorname{grade} (I:K),$$ then $K=I^{\text{gunm}}$. In particular, if $I=\cap_i Q_i$ is a primary decomposition of $I$, with each $Q_i$  a ${\mathfrak p}_i$-primary ideal of $R$, 
 then $I^{\text{gunm}}$ is the intersection of the primary components $Q_i$ of $I$ which correspond to prime ideals ${\mathfrak p}_i$ with $\operatorname{grade} {\mathfrak p}_i=\operatorname{grade} I$. 

 Of course, if $R$ is Cohen-Macaulay, then $I^{\text{gunm}}$ is the usual unmixed part of $I$.

We recall that if $I\subseteq {\mathfrak p}$ are ideals with ${\mathfrak p}$ prime, then $$\operatorname{grade} I\le \operatorname{grade} {\mathfrak p}\le \operatorname{depth} R_{\mathfrak p};$$consequently, if $\operatorname{depth} R_{\mathfrak p}=\operatorname{grade} I$ for all ${\mathfrak p}\in \operatorname{Ass} \frac RI$, then $I$ is automatically grade unmixed.
\end{chunk}

\begin{gunm-def} If $K_1$ and $K_2$ are ideals which both satisfy $$\operatorname{grade}  I<\operatorname{grade}  (I:K_i),$$ then $\operatorname{grade}  I<\operatorname{grade}  (I:(K_1+K_2))$; hence the hypothesis that $R$ is Noetherian guarantees that an ideal $K$ which satisfies (\ref{17.15.b}) exists. Fix this $K$. 
We show that  $K$  also has the property  of (\ref{17.15.a}).

Let $\underline x$ be a  maximal $R$-regular sequence  in $I$. All of the calculations  may be made in $R/(\underline{x})$. So no harm is done if we prove the statement when $\operatorname{grade}  I=0$.

We first show $K$ is grade unmixed of grade $0$. Use the primary decomposition of $K$ to write ${K=A_1\cap A_2}$, where every associated prime ideal of $\frac R{A_1}$ has grade $0$, and, either $A_2=R$, or $A_2$ is a proper ideal and every associated prime ideal of $\frac R{A_2}$ has positive grade. Observe that $$(I:K)A_1A_2\subseteq (I:K)K\subseteq I;$$ hence,
$(I:K)A_2\subseteq (I:A_1)$. Observe that $(I:K)A_2$ has positive grade. It follows that $K\subseteq A_1$ and $0<\operatorname{grade}  (I:A_1)$. The defining property of $K$ now guarantees that $K=A_1$ and therefore, $K$ is grade unmixed. 

Now we show that $K$ has property (\ref{17.15.a}). We have already shown that  $K$ has grade $0$ and  is grade unmixed. We prove that $K$ is the smallest such ideal. Let $J$ be any grade unmixed ideal of $R$ with $I\subseteq J$ and $\operatorname{grade}  J=0$. We  prove that $K\subseteq J$. It suffices to show $K_\mathfrak p\subseteq J_\mathfrak p$ for all $\mathfrak p$ in $\operatorname{Ass}  (\frac RJ)$. Let $\mathfrak p$ be in $\operatorname{Ass}  (\frac RJ)$. The hypotheses on $J$ guarantees that $\operatorname{grade}  \mathfrak p=0$ and therefore, $(I:K)\not\subseteq \mathfrak p$.
On the other hand, $(I:K)K\subseteq I\subseteq J$; hence, 
$$K_\mathfrak p=(I:K)_\mathfrak pK_\mathfrak p\subset I_\mathfrak p\subseteq J_\mathfrak p.$$

With respect to the ``furthermore'' assertion, $I^{\text{gunm}}\subseteq K$ by
(\ref{17.15.a}) because $K$ is grade unmixed and $K\subseteq I^{\text{gunm}}$
by
(\ref{17.15.b}) because $\operatorname{grade}  I<\operatorname{grade} (I:K)$. The assertion about primary decomposition is now obvious.
\qed\end{gunm-def}

\begin{chunk}\label{pd} Let $\operatorname{pd}_R(M)$ represent the projective dimension of an $R$-module $M$. \end{chunk}

\begin{chunk}\label{perfect} Let   $I$ be a proper ideal  in a Noetherian ring $R$. Since one can compute $\operatorname{Ext}^\bullet_R(R/I,R)$ from a projective resolution of $R/I$, one obviously has 
\begin{equation}\label{above}\textstyle \operatorname{grade} I\le \operatorname{pd}_R R/I;\end{equation}
if equality holds, then $I$ is called a {\it perfect ideal}.  
Recall, for example, that if $I$ is a proper homogeneous ideal in  a polynomial ring  $R$  over a field, then $I$ is a perfect ideal if and only if $R/I$ is a Cohen-Macaulay ring. (This is not the full story. For more information, see, for example, \cite[Prop.~16.19]{BV} or \cite[Thm.~2.1.5]{BH}.) A perfect ideal $I$ of grade $g$ is a {\it Gorenstein ideal}
if
$\operatorname{Ext}^g_R(R/I,R)$ is a cyclic R-module.\end{chunk}

\begin{lemma}\label{Gamma}Adopt the notation of {\rm \ref{87Not1}}. If $A$ and $B$ are non-negative integers and $\Gamma$ is an element of $D_{B}(G^*)$, then 
$$\sum\limits_{m\in \binom{Y}{A}}m^*\cdot m(\Gamma)=\binom BA \Gamma.$$
\end{lemma}

\begin{proof}It suffices to prove the result for $\Gamma=X_1^{(b_1)}\cdots X_{\goth g}^{(b_{\goth g})}$ where each $b_i$ is a  nonnegative integer and $\sum_i b_i=B$. Let $m=Y_1^{a_1}\cdots Y_{\goth g}^{a_{\goth g}}$, where each $a_i$ is a   non-negative integer and $\sum_i a_i=A$.  Observe that $$m^*\cdot m(\Gamma)=\binom{b_1}{a_1}\cdots\binom {b_{\goth g}}{a_{\goth g}} \Gamma,$$ because $X_i^{(a_i)}\cdot X_i^{(b_i-a_i)}=\binom{b_i}{a_i}X_i^{(b_i)}$. (The most recent equation holds for all non-negative integers $a_i$ and $b_i$.) At this point, we have shown that
$$\sum\limits_{m\in \binom{Y}{A}}m^*\cdot m(\Gamma)=\sum\limits_{a_1+\dots+a_n=A}\binom{b_1}{a_1}\cdots\binom {b_{\goth g}}{a_{\goth g}}\Gamma.$$
On the other hand, \begin{equation}\label{cute}\sum\limits_{a_1+\dots+a_{\goth g}=A}\binom{b_1}{a_1}\cdots\binom {b_{\goth g}}{a_{\goth g}}=\binom BA.\end{equation} Indeed the left of (\ref{cute}) is the coefficient of $x^Ay^{B-A}$ in the left side  
of the polynomial \begin{equation}\label{duh}(x+y)^{b_1}\cdots (x+y)^{b_{\goth g}}=(x+y)^{B},\end{equation} and the right side of (\ref{cute}) is the coefficient of $x^Ay^{B-A}$ in the right side of (\ref{duh}).
\end{proof}

\section{The maps and modules of ${\mathbb M^\epsilon}$.}\label{Med}

The object ${\mathbb M^\epsilon}$ is the focal point of this paper. We introduce the maps and modules of ${\mathbb M^\epsilon}$ in Definition \ref{doodles'}; all of the numerical information about ${\mathbb M^\epsilon}$ is contained in Remarks~\ref{num-info-Med}; and some examples of ${\mathbb M^\epsilon}$ are given in Examples \ref{CompleteExample} and \ref{exs-Med}. We prove in Remark~\ref{rmk5} that ${\mathbb M^\epsilon}$ is a complex and in Lemma~\ref{resol} that ${\mathbb M^\epsilon}$ is a resolution when Hypotheses \ref{87hyp6} are in effect. 
Ultimately, ${\mathbb M^\epsilon}$ is a subcomplex of ${\mathbb L^\epsilon}$,  ${\mathbb L^\epsilon}$ is the mapping cone of
$$\xymatrix{\operatorname{Tot}({\mathbb T^\epsilon})\ar[d]^{\xi^{\epsilon}}\\\operatorname{Tot}({\mathbb B}),}$$ and ${\mathbb T^\epsilon}$ and ${\mathbb B}$ are quotient sub-double complexes of the double complex $\mathbb V$. We use ${\mathbb L^\epsilon}$, ${\mathbb T^\epsilon}$, ${\mathbb B}$, and $\mathbb V$ to prove Lemma~\ref{resol}; however, ${\mathbb M^\epsilon}$ is the object of interest in this paper; and therefore, we introduce it, in complete detail, first.
 
The  conventions and notation of Section~\ref{not-con} are used throughout this section; in particular, the modules $V^\bullet_{i,j}$ are 
defined in \ref{V-bullet}, the symbols $X_{\ell}$, $Y_{\ell}$ and $\binom{Y}{i}$ are defined in \ref{87Not1},  the notation $$\textstyle \bigoplus\limits_{\gf{\epsilon\le J}{I+J\le \delta-1}}$$ is explained in \ref{2.8}, the function
$\chi$ is explained in \ref{chi}, and the bases $\omega_{G^*},\omega_G$ are explained in \ref{omega}. As one reads Definition~\ref{doodles'}, it  might be helpful to simultaneously follow  Example~\ref{CompleteExample} where we record ${\mathbb M^\epsilon}$, all of its constituent pieces, and all of its forms, when $(\goth g,\goth f)=(3,9)$ and $\epsilon=\frac \delta2$. 

\begin{definition}\label{doodles'}
Adopt Data~{\rm{\ref{87data6}}} and 
{\rm{\ref{unm-data}}}. The maps and modules ${\mathbb M^\epsilon}$ are described as follows:
\begin{enumerate}[\rm(a)]
\item\label{doodles'.a} As a graded $R$-module $${\mathbb M^\epsilon}=(\bigoplus\limits_{\gf{\epsilon\le J}{I+J\le \delta-1}} V_{I,J}^{{\mathbb T^\epsilon}}) \oplus (\bigoplus\limits_{\gf{j\le \epsilon -1}{\delta\le i+j}} V_{i,j}^{{\mathbb B}}) \oplus {\textstyle(\bigwedge^{\goth f}F\otimes \bigwedge^{\goth g}G)},$$
with \begin{enumerate}[\rm(i)]
\item $V_{I,J}^{{\mathbb T^\epsilon}}$ in position $I+2J-\delta+2$,
\item $V_{i,j}^{{\mathbb B}}$ in position $i+2j-\delta+1$, and 
\item $\textstyle(\bigwedge^{\goth f}F\otimes \bigwedge^{\goth g}G)$ in position $0$.
 \end{enumerate}

\item\label{doodles'.b} The $R$-module homomorphisms  of ${\mathbb M^\epsilon}$ are  described below.
\begin{enumerate}[\rm(i)] \item\label{med-top-to-bot'} If $I+2J-\delta+2=N$, $2\le N$, $0\le I$, $\epsilon\le J$, and $I+J\le \delta-1$, then $V_{I,J}^{{\mathbb T^\epsilon}}$ is a summand of ${\mathbb M^{\epsilon}}_N$ and $$d(f_I\otimes \gamma_J\in V_{I,J}^{{\mathbb T^\epsilon}})
=\begin{cases} 
\phantom{+}\chi(\epsilon\le J-1)\sum\limits_{\ell=1}^{\goth g}\Psi(X_\ell)\wedge f_I\otimes Y_\ell(\gamma_J)\in V_{I+1,J-1}^{{\mathbb T^{\epsilon}}}\subseteq {\mathbb M^\epsilon}_{N-1}\\
+\tau(f_I)\otimes \gamma_J\in V_{I-1,J}^{{\mathbb T^\epsilon}}\subseteq {\mathbb M^\epsilon}_{N-1}\\
+\sum\limits_{\gf{i+2j=I+2J}{\delta\le i+j}}(-1)^{I+j}\binom{J-1-j}{J-\epsilon} \sum\limits_{m\in \binom{Y}{J-j}} [D(\mu)](m^*)\wedge f_I\otimes m(\gamma_J)\\\hskip1in\in V_{i,j}^{{\mathbb B}}\subseteq {\mathbb M^\epsilon}_{N-1}.
\end{cases}$$

\item If $i+2j-\delta+1=N$, $2\le N$, $\delta\le i+j$,  $0\le i,j$, and $j\le \epsilon-1$, then $V_{i,j}^{{\mathbb B}}$ is a summand of 
${\mathbb M^\epsilon}_N$ and 
 $$d(f_i\otimes \gamma_j\in V_{i,j}^{{\mathbb B}}) 
=\begin{cases} 
\phantom{+}\sum\limits_{\ell=1}^{\goth g}\Psi(X_\ell)\wedge f_i\otimes Y_\ell(\gamma_j)\in V_{i+1,j-1}^{{\mathbb B}}\subseteq {\mathbb M^\epsilon}_{N-1}\\
+\chi(\delta\le i+j-1)\tau (f_i)\otimes \gamma_j\in V_{i-1,j}^{{\mathbb B}}\subseteq {\mathbb M^\epsilon}_{N-1}.
\end{cases}$$
\item The $R$-module homomorphism  
\begin{align}\textstyle d:{\mathbb M^\epsilon}_1&=\textstyle V_{\delta,0}^{{\mathbb B}}\oplus  \chi(\epsilon=\frac{\delta-1}2)V_{0,\epsilon}^{\mathbb T^\epsilon}\to{\mathbb M^\epsilon}_0={\textstyle(\bigwedge^{\goth f}F\otimes \bigwedge^{\goth g}G)}\label{blip}\\\intertext{is}
\textstyle d(f_\delta\in V_{\delta,0}^{{\mathbb B}})&=\textstyle f_\delta\wedge (\bigwedge^{\goth g}\Psi)(\omega_{G^*})\otimes \omega_G, \notag\\\intertext{and, if $\epsilon= \frac{\delta-1}2$, then} 
d(\gamma_\epsilon\in V_{0,\epsilon}^{\mathbb T^{\epsilon}})&\textstyle =
c(\gamma_\epsilon)\text{ for $c$ as defined in {\rm(\ref{c-2.3})}.}
\label{blop}
\end{align}
\end{enumerate}
\end{enumerate}
\end{definition}
 \begin{remarks}\label{num-info-Med}\begin{enumerate}[\rm(a)]

\item\label{num-info-Med.new-a} If $\delta=1$ and $\epsilon=\frac{\delta-1}2$, then $\mathbb M^\epsilon$ does not have any non-zero summands of the form $V_{i,j}^{\mathbb B}$. Otherwise, the module $V_{\goth f,\epsilon-1}^{\mathbb B}$ is a non-zero summand of $\mathbb M^\epsilon_{\goth g+2\epsilon-1}$; furthermore, if ${\goth g+2\epsilon-1<N}$,  then ${\mathbb M^\epsilon}_N$ does not contain any non-zero summands of the form $V_{i,j}^{{\mathbb B}}$. The value of  $\goth g+2\epsilon-1$ is 
$$\begin{cases} \goth f-2,&\text{if $\epsilon=\frac{\delta-1}2$,}\\
\goth f-1,&\text{if $\epsilon=\frac \delta2$, and }\\\goth f,&\text{if $\epsilon=\frac{\delta+1}2$}.\end{cases}$$

\item\label{num-info-Med.b} 
If $\delta=1$ and $\epsilon=\frac{\delta+1}2$, then $\mathbb M^\epsilon$ does not have any non-zero summands of the form $V_{I,J}^{\mathbb T^\epsilon}$. Otherwise, 
$V_{0,\delta-1}^{{\mathbb T^\epsilon}}$ is a non-zero summand of ${\mathbb M^\epsilon}_{\delta}$; furthermore, if $\delta+1\le N$, then ${\mathbb M^\epsilon}_N$ does not contain any non-zero summands of the form $V_{I,J}^{{\mathbb T^\epsilon}}$. 

\item\label{num-info-Med.a} The largest index $N$ with ${\mathbb M^\epsilon}_N\neq 0$ is
\begin{equation}\label{last}N_{\max}=
\begin{cases} 
1,&\text{if $\epsilon=\frac{\delta-1}2$ and $\delta=1$},\\
\goth f-1,&
\text{if $\epsilon=\frac{\delta-1}2$ and $1= \goth g$,}\\
\goth f-2,&
\text{if $\epsilon=\frac{\delta-1}2$, $2\le \delta$, and $2\le \goth g$,}\\
\goth f-1,&\text{if $\epsilon=\frac \delta2$, and }\\
\goth f,&\text{if $\epsilon=\frac{\delta+1}2$};\end{cases}\end{equation} furthermore, if $N_{\max}$ is the parameter of (\ref{last}), then ${\mathbb M^\epsilon}_{N_{\max}}$ is equal to
$$\begin{cases}
V_{0,\delta-1}^{\mathbb T^\epsilon},&\begin{array}{l}\text{if $\epsilon=\frac{\delta-1}2$
and $1=\goth g$}\\\text{or  $\epsilon=\frac{\delta-1}2$
and $1=\delta$,
}\end{array}\vspace{8pt}\\
\chi(\epsilon=\frac{\delta-1}2)\chi(2=\goth g)V_{0,\delta-1}^{\mathbb T}\oplus
\chi(\epsilon=\frac{\delta}2)\chi(1=\goth g)V_{0,\delta-1}^{\mathbb T}\oplus
V_{\goth f,\epsilon-1}^{\mathbb B},
&\begin{array}{l}\text{otherwise.}\end{array}\end{cases}$$

\item As noted in (\ref{num-info-Med.b}), $\mathbb M^\epsilon$ does not have any non-zero summands of the form $V_{I,J}^{\mathbb T^\epsilon}$ if $\delta=1$ and $\epsilon=\frac{\delta+1}2$. Otherwise, the module $V_{0,\epsilon}^{\mathbb T^\epsilon}$ is a non-zero summand of ${\mathbb M^\epsilon}_{N_0}$ for
\begin{equation} N_0=2\epsilon -\delta+2=\begin{cases}
1,&\text{if $\epsilon=\frac{\delta-1}2$,}\\
2,&\text{if $\epsilon=\frac{\delta}2$, and }\\
3,&\text{if $\epsilon=\frac{\delta+1}2$ and $2\le \delta$;}\end{cases}\end{equation} furthermore, if $N<N_0$, then ${\mathbb M^\epsilon}_N$ does not contain any non-zero summands of the form $V_{I,J}^{{\mathbb T^\epsilon}}$. 

\item\label{num-info-Med.c} If $R$ is a bi-graded ring and  $$\tau:R(-1,0)^{\goth f}\to R\quad\text{and}\quad \mu: R(0,-1)^{\goth g}\to R^{\binom{\goth f}2}$$ are bi-homogeneous $R$-module homomorphisms, then the maps and modules of ${\mathbb M^\epsilon}$ are bi-homogeneous with 
\begin{align}V_{I,J}^{{\mathbb T^\epsilon}}&\simeq R(\goth f-2\goth g-I-2J,-\goth g-J)^{\operatorname{rank} V_{I,J}^{{\mathbb T^\epsilon}}},\notag\\ V_{i,j}^{{\mathbb B}}
&\simeq R(\goth f-2\goth g-i-2j,-\goth g-j)^{\operatorname{rank} V_{i,j}^{{\mathbb B}}},\quad\text{and}\notag\\ \textstyle(\bigwedge^{\goth f}F\otimes \bigwedge^{\goth g}G)&\simeq R.\notag\end{align}
Indeed, under the given hypotheses,
$$\Psi:R(-1,-1)^{\goth g}\to R^{\goth f}\quad\text{and}\quad  D(\mu):R(0,-J)^{\binom{J+\goth g-1}{J}}\to R^{\binom{\goth f}{2J}}$$ are also  
bi-homogeneous $R$-module homomorphisms.
\item The hypotheses of (\ref{num-info-Med.c}) are in effect in the generic case where $$R=R_0[T_1,\dots,T_{\goth f},\{A_{i,j,k}\mid 1\le i<j\le \goth f\text{ and } 1\le k\le \goth g\}],$$$\deg T_i=(1,0)$, $\deg A_{i,j,k}=(0,1)$,  $\tau(e_i)=T_i$, for $1\le i\le \goth f$, and $$\mu(X_k)=\sum_{1\le i<j\le \goth f}A_{i,j,k} e_i\wedge e_j,$$ for $1\le k\le \goth g$, with 
$e_1,\dots,e_{\goth f}$ a basis for $F$ and  $X_1,\dots,X_{\goth g}$  a basis for $G^*$. 
\item\label{num-info-Med.d} If $R$ is a graded ring and $\tau:R(-1)^{\goth f}\to R$ and $\mu:R^{\goth g}\to R^{\binom{\goth f}2}$ are homogeneous $R$-module homomorphisms, then the maps and modules of ${\mathbb M^\epsilon}$ are homogeneous with 
$${\mathbb M^\epsilon}_N\simeq 
\begin{cases}
R&\text{if $N=0$}\\
R(-\goth g +2-N)^{\beta'_N}\oplus R(-\goth g +1-N)^{\beta_N}&\text{if $1\le N\le 
N_{\max}$, as given in (\ref{last}) },\end{cases}$$
where
\begin{equation}\label{beta-N}\beta'_N=\sum\limits_{\gf{I+2J-\delta+2=N}{\gf{\epsilon\le J}{I+J\le \delta-1}}} \operatorname{rank} V_{I,J}^{{\mathbb T^\epsilon}}\quad\text{and}\quad
\beta_N=\sum\limits_{\gf{i+2j-\delta+1=N}{\gf{j\le \epsilon -1}{\delta\le i+j}}} \operatorname{rank} V_{i,j}^{{\mathbb B}},\end{equation}for $1\le N\le N_{\max}$.
\item\label{special hypotheses} The hypotheses of  (\ref{num-info-Med.d}) are in effect in the special case where $R$ is the polynomial ring 
$$R=R_0[T_1,\dots,T_{\goth f}],$$ $\deg T_i=1$, $\deg \alpha_{i,j,k}=0$,  $\tau(e_i)=T_i$, for $1\le i\le \goth f$, and $$\mu(X_k)=\sum_{1\le i<j\le \goth f}\alpha_{i,j,k} e_i\wedge e_j,$$ for $1\le k\le \goth g$, with 
$e_1,\dots,e_{\goth f}$ a basis for $F$, and  $X_1,\dots,X_{\goth g}$  a basis for $G^*$, and $\alpha_{i,j,k}\in R_0$. 

\item\label{num-info-Med.e} If $\epsilon$ is equal to $\frac\delta2$ or $\frac{\delta+1}2$,  ${\mathbb M^\epsilon}$ is a resolution, and the hypotheses of (\ref{special hypotheses}) are in effect with $R_0$ a field,   then there is a quasi-isomorphism from ${\mathbb M^\epsilon}$ to the $\goth g$-linear minimal resolution
$$ 0\to R(-\goth g-N_{\max}+1)^{b_{N_{\max}}}\to\dots \to R(-\goth g-1)^{b_{2}}\to R(-\goth g)^{b_1}\to R, $$ 
with $b_N=\beta_N-\beta_{N+1}'$ for $\beta_N$ and $\beta'_N$ as defined in (\rm{\ref{beta-N}}). (This remark is a consequence of three facts. First of all, 
 the generators of ${\mathbb M^{\epsilon}}_1$ all have degree $\goth g$; secondly, every homogeneous $R$-module homomorphism $$R(-a)\to R(-a-1)$$ is necessarily zero; and thirdly, the image of a homogeneous $R$-module homomorphism $R(-a)^A\to R(-a)^{B}$ is a free summand of the target.)
 
However, the analogous statement is not true when $\epsilon=\frac{\delta-1}{2}$. See Examples~\ref{exs-Med}.\ref{4.4.c} and \ref{exs-Med}.\ref{4.4.c-may-7}.   
\item\label{num-info-Med.i} In Definition \ref{doodles'}.\ref{med-top-to-bot'}, it is not necessary to impose the condition $j\le \epsilon-1$ in the $V_{i,j}^{{\mathbb B}}$-component of $d(V_{I,J}^{{\mathbb T^\epsilon}})$; because, if the expression
\begin{equation}\label{exp'}\sum\limits_{\gf{i+2j=I+2J}{\delta\le i+j}}(-1)^{I+j}\binom{J-1-j}{J-\epsilon} \sum\limits_{m\in \binom{Y}{J-j}} [D(\mu)](m^*)\wedge f_I\otimes m(\gamma_J)\in V_{i,j}^{{\mathbb B}}\end{equation} is non-zero, then the inequality $j\le \epsilon-1$ is automatically satisfied. Indeed, if (\ref{exp'}) is non-zero, then $j\le J$. However $j$ can not equal $J$; because, if $j=J$, then $i=I$ and $\delta\le i+j=I+J\le \delta-1$; which is impossible. Thus, $0\le J-j-1$. On the other hand, the binomial coefficient $\binom{J-1-j}{J-\epsilon}$ is not zero; so $0\le J-\epsilon\le J-1-j$, and $j\le \epsilon-1$, as claimed. (This remark, which looks technical, is actually the proof of the assertion that ${\mathbb M^\epsilon}$ is a subcomplex of ${\mathbb L^\epsilon}$; see Remark~\ref{rmk5}.)
\end{enumerate}
\end{remarks}

\begin{example}\label{CompleteExample}Let $(\goth g,\goth f)=(3,9)$. The parameter $\delta$ (which equals $\goth f-\goth g$) is even, and therefore the constraint $\lceil\frac{\delta-1}2\rceil \le \epsilon\le \lceil\frac{\delta}2\rceil$ forces $\epsilon$ to equal $\frac\delta2=3$. In this example, we record ${\mathbb M^\epsilon}$, all of its constituent pieces, and all of its forms. 
Definition~\ref{doodles'}.\ref{doodles'.a} says that 
\begin{align}{\mathbb M^\epsilon}& = \bigoplus_{(\text{\ref{I,J}})} V_{I,J}^{{\mathbb T^\epsilon}} \oplus \bigoplus_{(\text{\ref{i,j}})} V_{i,j}^{{\mathbb B}} \oplus \textstyle (\bigwedge^9F \otimes \bigwedge^3G),\notag\\\intertext{with} (I,J)\in & \{(I,3)\mid 0\le I\le 2\}\cup \{(I,4)\mid 0\le I\le 1\}\cup \{(0,5)\}\label{I,J}\\\intertext{and}
(i,j)\in & \{(i,0)\mid 6\le i\le 9\}\cup \{(i,1)\mid 5\le i\le 9\}\cup \{(i,2)\mid 4\le i\le 9\}.\label{i,j}\end{align}
\begin{table}
\begin{center}
\begingroup\allowdisplaybreaks
\begin{align}\label{39top}\xymatrix{V_{0,5}^{{\mathbb T^\epsilon}}\ar[r]^{\Psi}&V_{1,4}^{{\mathbb T^\epsilon}}\ar[r]^{\Psi}\ar[d]^{\tau}&
V_{2,3}^{{\mathbb T^\epsilon}}\ar[d]^{\tau}\\
&V_{0,4}^{{\mathbb T^\epsilon}}\ar[r]^{\Psi}&
V_{1,3}^{{\mathbb T^\epsilon}}\ar[d]^{\tau}\\&&V_{0,3}^{{\mathbb T^\epsilon}}}\\
\intertext{and}
\label{39bot}
\xymatrix{
V^{\mathbb B}_{9,2}\ar[d]^\tau\\
V^{\mathbb B}_{8,2}\ar[d]^\tau\ar[r]^\Psi&V^{\mathbb B}_{9,1}\ar[d]^\tau\\
V^{\mathbb B}_{7,2}\ar[d]^\tau\ar[r]^\Psi&V^{\mathbb B}_{8,1}\ar[d]^\tau\ar[r]^\Psi&V^{\mathbb B}_{9,0}\ar[d]^\tau\\
V^{\mathbb B}_{6,2}\ar[d]^\tau\ar[r]^\Psi&V^{\mathbb B}_{7,1}\ar[d]^\tau\ar[r]^\Psi&V^{\mathbb B}_{8,0}\ar[d]^\tau\\
V^{\mathbb B}_{5,2}\ar[d]^\tau\ar[r]^\Psi&V^{\mathbb B}_{6,1}\ar[d]^\tau\ar[r]^\Psi&V^{\mathbb B}_{7,0}\ar[d]^\tau\\
V^{\mathbb B}_{4,2}\ar[r]^\Psi&V^{\mathbb B}_{5,1}\ar[r]^\Psi&V^{\mathbb B}_{6,0}\ar[rr]^{\wedge^3\Psi\hskip20pt}&&\bigwedge^9F\otimes \bigwedge^3G.}\end{align}\endgroup
\caption{{ Double complexes which are used in the construction of the first version (\ref{first version}) of $\mathbb M^\epsilon$ in Example~\ref{CompleteExample}.}}\label{WE ARE NOT}
\end{center}\end{table}

\noindent It is useful to consider the double complexes of Table~\ref{WE ARE NOT}. (The double complexes $\mathbb T^\epsilon$ and $\mathbb B$ are officially introduced in Definition~\ref{87.11}; the complex $\operatorname{Tot}(\text{\ref{39top}})$ is a shift (see (\ref{new})) of a subcomplex of $\operatorname{Tot} (\mathbb T^\epsilon)$ 
and the complex $\operatorname{Tot}(\text{\ref{39bot}})$ is a subcomplex of $\operatorname{Tot} (\mathbb B)$.) It is shown in Lemma~\ref{MOC} that there is a map of complexes $\xi$ from a shift of $\operatorname{Tot}(\text{\ref{39top}})$ to $\operatorname{Tot}(\text{\ref{39bot}})$, with ${\xi(V_{0,3}^{\mathbb T^\epsilon})\subseteq V_{6,0}^{\mathbb B}}$, so that ${\mathbb M^\epsilon}$ is the mapping cone of $\xi$. As a graded module, ${\mathbb M^\epsilon}$ is  
\begin{align}0\to V^{\mathbb B}_{9,2}\to V^{\mathbb B}_{8,2}\to 
\begin{matrix} V^{\mathbb T^\epsilon}_{0,5}
\\\oplus\\V^{\mathbb B}_{7,2}\\\oplus\\ V^{\mathbb B}_{9,1}\end{matrix} \to
\begin{matrix} V^{\mathbb T^\epsilon}_{1,4}\\\oplus \\V^{\mathbb B}_{6,2}\\\oplus\\V^{\mathbb B}_{8,1}\end{matrix}
 \to 
\begin{matrix} V^{\mathbb T^\epsilon}_{0,4}\\
\oplus\\V^{\mathbb T^\epsilon}_{2,3}\\\oplus\\ V^{\mathbb B}_{5,2}\\\oplus \\ V^{\mathbb B}_{7,1}\\\oplus
\\ V^{\mathbb B}_{9,0}
\end{matrix}
\to \begin{matrix} V^{\mathbb T^\epsilon}_{1,3}\\\oplus\\V^{\mathbb B}_{4,2}\\\oplus\\
 V^{\mathbb B}_{6,1}\\\oplus \\ V^{\mathbb B}_{8,0}
\end{matrix}
\to \begin{matrix} V^{\mathbb T^\epsilon}_{0,3}\\\oplus\\V^{\mathbb B}_{5,1}
\\\oplus\\
V^{\mathbb B}_{7,0}\end{matrix}
\to V^{\mathbb B}_{6,0}\to \textstyle\bigwedge^9F\otimes \bigwedge^3G,\label{first version}
\end{align}
with $\bigwedge^9F\otimes \bigwedge^3G$ in position zero. (One can use the formulas given in Definition~\ref{doodles'}.\ref{doodles'.a} to calculate the position of each summand $V^{\mathbb T^\epsilon}_{I,J}$ and $V^{\mathbb B}_{i,j}$ of ${\mathbb M^\epsilon}$; however, if the double complexes (\ref{39top}) and (\ref{39bot}) are available, then it is easy to read the position of each summand of ${\mathbb M^\epsilon}$ from the mapping cone construction.) If the data is bi-homogeneous, as described in Remark~\ref{num-info-Med}.\ref{num-info-Med.c}, then 
${\mathbb M^\epsilon}$ is given in Table~\ref{NEITHER}.
\begin{table}
\begin{center}
\begin{align}0\to R(-10,-5)^6\to R(-9,-5)^{54}\to \begin{matrix} R(-7,-8)^{21}\\\oplus\\R(-8,-5)^{216}\\\oplus\\ R(-8,-4)^3\end{matrix} \to
\begin{matrix} R(-6,-7)^{135}\\\oplus\\R(-7,-5)^{504}\\\oplus\\
R(-7,-4)^{27}\end{matrix}\notag\\ \notag\\ \to \begin{matrix} R(-5,-7)^{15}\\\oplus\\R(-5,-6)^{360}\\\oplus\\ R(-6,-5)^{756}\\\oplus \\ R(-6,-4)^{108}\\\oplus\\ R(-6,-3)^1
\end{matrix}\to \begin{matrix} R(-4,-6)^{90}\\\oplus\\R(-5,-5)^{756}\\\oplus\\ R(-5,-4)^{252}\\\oplus \\ R(-5,-3)^{9}
\end{matrix}\to \begin{matrix} R(-3,-6)^{10}\\\oplus\\R(-4,-4)^{378}
\\\oplus\\
R(-4,-3)^{36}\end{matrix}\to R(-3,-3)^{84}\to R.\notag\end{align}
\caption{The complex  ${\mathbb M^\epsilon}$ from Example~\ref{CompleteExample} when
the data is bi-homogeneous, as described in Remark~\ref{num-info-Med}.\ref{num-info-Med.c}.
}\label{NEITHER}
\end{center}\end{table}
The rank of $V^\bullet_{I,J}$ is given in \ref{V-bullet}. The bi-homogeneous twists in ${\mathbb M^\epsilon}$ are given in Remark~\ref{num-info-Med}.\ref{num-info-Med.c} or may be read from the double complexes (\ref{39top}) and (\ref{39bot}) as soon as one knows that $\bigwedge^9F\otimes \bigwedge^3G=R$ and $V^{{\mathbb T}}_{0,3}=R(-3,-6)^{10}$. If the hypotheses of Remark~\ref{num-info-Med}.\ref{num-info-Med.d} are in effect, then ${\mathbb M^\epsilon}$ is 
\begin{align}0\to R(-10)^6\to R(-9)^{54}\to \begin{matrix} R(-7)^{21}\\\oplus\\ R(-8)^{219}\end{matrix}\to \begin{matrix} R(-6)^{135}\\\oplus\\R(-7)^{531}\end{matrix} \to
\begin{matrix} R(-5)^{375}\\\oplus\\ R(-6)^{865}\end{matrix}\notag\\ \notag\\
\to \begin{matrix} R(-4)^{90}\\\oplus \\R(-5)^{1017}\end{matrix}
\to \begin{matrix} R(-3)^{10}\\\oplus \\R(-4)^{414}\end{matrix}\to R(-3)^{84}\to R.\notag\end{align}
If the hypotheses of Remark~\ref{num-info-Med}.\ref{num-info-Med.e} are in effect, then ${\mathbb M^\epsilon}$ is quasi-isomorphic to
\begin{align}0\to R(-10)^6\to R(-9)^{54}\to R(-8)^{219}\to R(-7)^{510} \to
R(-6)^{730}\notag\\ \label{c}\to R(-5)^{642}\to R(-4)^{324}\to R(-3)^{74}\to R.\end{align}
 \end{example}

\begin{examples}\label{exs-Med} These examples are presented more quickly than Example~\ref{CompleteExample}. \begin{enumerate}[\rm(a)]

\item\label{4.4.a} Let $(\goth g,\goth f)=(2,6)$ and $\epsilon=\frac \delta 2=2$. 
The modules of $\mathbb M^\epsilon$ are
$$0\to V_{6,1}^{\mathbb B}\to \begin{matrix}V_{0,3}^{\mathbb T^\epsilon}\\\oplus\\V_{5,1}^{\mathbb B}\end{matrix} \to
\begin{matrix} V_{1,2}^{\mathbb T^\epsilon}\\\oplus\\V_{4,1}^{\mathbb B}\\\oplus\\
V_{6,0}^{\mathbb B}\end{matrix}\to \begin{matrix} V_{0,2}^{\mathbb T^\epsilon}\\\oplus\\V_{3,1}^{\mathbb B}\\\oplus\\
V_{5,0}^{\mathbb B}\end{matrix}\to V_{4,0}^{\mathbb B}\to \textstyle \bigwedge^6F\otimes \bigwedge^2G.$$
If the hypotheses of Remark~\ref{num-info-Med}.\ref{num-info-Med.c} are in effect, then ${\mathbb M^\epsilon}$ is 
$$0\to R(-6,-3)^2\to \begin{matrix} R(-4,-5)^4\\\oplus\\R(-5,-3)^{12}\end{matrix} \to
\begin{matrix} R(-3,-4)^{18}\\\oplus\\R(-4,-3)^{30}\\\oplus\\
R(-4,-2)^1\end{matrix}\to \begin{matrix} R(-2,-4)^{3}\\\oplus\\R(-3,-3)^{40}\\\oplus\\
R(-3,-2)^6\end{matrix}\to R(-2,-2)^{15}\to R.$$
If the hypotheses of Remark~\ref{num-info-Med}.\ref{num-info-Med.d} are in effect, then ${\mathbb M^\epsilon}$ is 
$$0\to R(-6)^2\to \begin{matrix} R(-4)^4\\\oplus\\R(-5)^{12}\end{matrix} \to
\begin{matrix} R(-3)^{18}\\\oplus\\R(-4)^{31}\end{matrix}\to \begin{matrix} R(-2)^{3}\\\oplus\\R(-3)^{46}\end{matrix}\to R(-2)^{15}\to R.$$
If the hypotheses of Remark~\ref{num-info-Med}.\ref{num-info-Med.e} are in effect, then ${\mathbb M^\epsilon}$ is quasi-isomorphic to
\begin{equation}\label{a}0\to R(-6)^2\to R(-5)^{12}\to R(-4)^{27}\to R(-3)^{28}\to R(-2)^{12}\to R.\end{equation}

\item\label{4.4.b} Let $(\goth g,\goth f)=(3,6)$ and $\epsilon=\frac{\delta+1}2=2$. The modules of  $\mathbb M^\epsilon$ are 
$$0\to V_{6,1}^{\mathbb B}\to V_{5,1}^{\mathbb B}\to \begin{matrix} V_{4,1}^{\mathbb B}\\\oplus\\V_{6,0}^{\mathbb B}\end{matrix} \to
\begin{matrix} V_{0,2}^{\mathbb T^\epsilon}\\\oplus\\V_{3,1}^{\mathbb B}\\\oplus\\
V_{5,0}^{\mathbb B}\end{matrix}\to \begin{matrix} V_{2,1}^{\mathbb B}\\\oplus\\V_{4,0}^{\mathbb B}
\end{matrix}\to V_{3,0}^{\mathbb B}\to \textstyle \bigwedge^6F\otimes \bigwedge^3G.$$
  If the hypotheses of Remark~\ref{num-info-Med}.\ref{num-info-Med.c} are in effect, then ${\mathbb M^\epsilon}$ is 
\begin{align}0\to R(-8,-4)^3\to R(-7,-4)^{18}\to \begin{matrix} R(-6,-4)^{45}\\\oplus\\R(-6,-3)^{1}\end{matrix} \to
\begin{matrix} R(-4,-5)^{6}\\\oplus\\R(-5,-4)^{60}\\\oplus\\
R(-5,-3)^6\end{matrix}\notag\\ \notag\\\to \begin{matrix} R(-4,-4)^{45}\\\oplus\\R(-4,-3)^{15}
\end{matrix}\to R(-3,-3)^{20}\to R.\end{align}
If the hypotheses of Remark~\ref{num-info-Med}.\ref{num-info-Med.d} are in effect, then ${\mathbb M^\epsilon}$ is 
$$0\to R(-8)^3\to R(-7)^{18}\to R(-6)^{46}\to \begin{matrix} R(-4)^6\\\oplus\\R(-5)^{66}\end{matrix} \to
R(-4)^{60}\to R(-3)^{20}\to R.$$
If the hypotheses of Remark~\ref{num-info-Med}.\ref{num-info-Med.e} are in effect, then ${\mathbb M^\epsilon}$ is quasi-isomorphic to
\begin{equation}\label{b}0\to R(-8)^3\to R(-7)^{18}\to R(-6)^{46}\to R(-5)^{66} \to
R(-4)^{54}\to R(-3)^{20}\to R.\end{equation}

\item\label{4.4.c} Let $(\goth g,\goth f)=(3,6)$ and $\epsilon=\frac{\delta-1}2=1$. The modules of $\mathbb M^\epsilon$ are $$0\to
V_{6,0}^{{\mathbb B}}
\to
\begin{matrix} V_{0,2}^{{\mathbb T^1}}
\\
\oplus
\\
V_{5,0}^{{\mathbb B}}
\end{matrix}
\to
\begin{matrix}
V_{1,1}^{{\mathbb T^1}}
\\
\oplus
\\
V_{4,0}^{{\mathbb B}}
\end{matrix}
\to
\begin{matrix}
V_{0,1}^{{\mathbb T^1}}
\\
\oplus\\
V_{3,0}^{{\mathbb B}}
\end{matrix}
\to
{\textstyle\bigwedge^{6}F\otimes \bigwedge^{3}G}.$$
If the hypotheses of Remark~\ref{num-info-Med}.\ref{num-info-Med.c} are in effect, then ${\mathbb M^\epsilon}$ is 
\begin{equation}\label{b'}0\to
R(-6,-3)^1
\to
\begin{matrix} R(-4,-5)^6
\\
\oplus
\\
R(-5,-3)^6
\end{matrix}
\to
\begin{matrix}
R(-3,-4)^{18}
\\
\oplus
\\
R(-4,-3)^{15}
\end{matrix}
\to
\begin{matrix}
R(-2,-4)^{3}
\\
\oplus\\
R(-3,-3)^{20}
\end{matrix}
\to
R.\end{equation}
If the hypotheses of Remark~\ref{num-info-Med}.\ref{num-info-Med.d} are in effect, then ${\mathbb M^\epsilon}$ is 
$$0\to
R(-6)^1
\to
\begin{matrix} R(-4)^6
\\
\oplus
\\
R(-5)^6
\end{matrix}
\to
\begin{matrix}
R(-3)^{18}
\\
\oplus
\\
R(-4)^{15}
\end{matrix}
\to
\begin{matrix}
R(-2)^{3}
\\
\oplus\\
R(-3)^{20}
\end{matrix}
\to
R.$$
In the present example, $\epsilon=\lceil \frac {\delta-1} 2\rceil$ and $\epsilon\neq \lceil \frac {\delta} 2\rceil$. If the rest of the hypotheses of Remark~\ref{num-info-Med}.\ref{num-info-Med.e}, other than the hypothesis $\epsilon= \lceil \frac {\delta} 2\rceil$, are in effect, then we do not know the graded Betti numbers in a minimal homogeneous resolution of $\operatorname{H}^0(\mathbb M^\epsilon)$; indeed, we do not  know if these Betti numbers can be determined from the data $(\epsilon,\goth f,\goth g)$ or if more information about the $R$-module $\mu$ (of Data \ref{87data6}) is required.

\item\label{4.4.c-may-7} Let $(\goth g,\goth f)=(4,7)$ and $\epsilon=\frac{\delta-1}2=1$. The modules of $\mathbb M^\epsilon$ are $$0\to
V_{7,0}^{{\mathbb B}}\to
V_{6,0}^{{\mathbb B}}
\to
\begin{matrix} V_{0,2}^{{\mathbb T^1}}
\\
\oplus
\\
V_{5,0}^{{\mathbb B}}
\end{matrix}
\to
\begin{matrix}
V_{1,1}^{{\mathbb T^1}}
\\
\oplus
\\
V_{4,0}^{{\mathbb B}}
\end{matrix}
\to
\begin{matrix}
V_{0,1}^{{\mathbb T^1}}
\\
\oplus\\
V_{3,0}^{{\mathbb B}}
\end{matrix}
\to
{\textstyle\bigwedge^{7}F\otimes \bigwedge^{4}G}.$$
If the hypotheses of Remark~\ref{num-info-Med}.\ref{num-info-Med.c} are in effect, then ${\mathbb M^\epsilon}$ is 
\begin{equation}\notag 0\to R(-8,-4)^1\to
R(-7,-4)^7
\to
\begin{matrix} R(-5,-6)^{10}
\\
\oplus
\\
R(-6,-4)^{21}
\end{matrix}
\to
\begin{matrix}
R(-4,-5)^{28}
\\
\oplus
\\
R(-5,-4)^{35}
\end{matrix}
\to
\begin{matrix}
R(-3,-5)^{4}
\\
\oplus\\
R(-4,-4)^{35}
\end{matrix}
\to
R.\end{equation}
If the hypotheses of Remark~\ref{num-info-Med}.\ref{num-info-Med.d} are in effect, then ${\mathbb M^\epsilon}$ is 
\begin{equation}\label{notlinear}0\to R(-8)^1\to
R(-7)^7
\to
\begin{matrix} R(-5)^{10}
\\
\oplus
\\
R(-6)^{21}
\end{matrix}
\to
\begin{matrix}
R(-4)^{28}
\\
\oplus
\\
R(-5)^{35}
\end{matrix}
\to
\begin{matrix}
R(-3)^{4}
\\
\oplus\\
R(-4)^{35}
\end{matrix}
\xrightarrow{d_1}
R.\end{equation}
Notice that the graded Betti numbers show that it is not possible for (\ref{notlinear}) to be quasi-isomorphic to a pure complex; indeed, the image of
$d_1$ is generated in two different degrees.

\item\label{4.4.d} One can use Macaulay2 \cite{M2} to verify the  graded Betti numbers of (\ref{c}), (\ref{a}), (\ref{b}),  and  (\ref{b'}), when the Hypotheses~\ref{87hyp6} are in effect.

\item\label{4.4.e} When $\goth f =2n+1$ is odd, $\goth g=1$, and $\epsilon=\frac{\delta}2$, then ${\mathbb M^\epsilon}$ is isomorphic to the complex 
\begin{equation}\label{aci}\mathbb M'=\bigoplus\limits_{\gf{i+j\le n}{0\le j}} {\textstyle\bigwedge^iF^*h^{(j)}}\oplus \bigoplus\limits_{\gf{p+q\le n-1}{0\le q}} {\textstyle\bigwedge^pF\lambda^{(q)}}\end{equation} of \cite[Def.~2.15]{K95}. In $\mathbb M'$, 
${\textstyle\bigwedge^iF^*h^{(j)}}$ is in position $i+2j$ and ${\textstyle\bigwedge^pF\lambda^{(q)}}$ in position ${p+2q+2}$.
The module $\bigwedge^\goth f F\otimes \bigwedge^\goth g G$ in ${\mathbb M^\epsilon}$ corresponds to $\textstyle\bigwedge^{0}F^*h^{(0)}$ 
in $\mathbb M'$, the module $V_{i,j}^{{\mathbb B}}$ in ${\mathbb M^\epsilon}$ corresponds to 
$\textstyle\bigwedge^{\goth f-i}F^*h^{(i+j-\delta)}$ 
in $\mathbb M'$, and the module $V_{I,J}^{{\mathbb T^\epsilon}}$ in ${\mathbb M^\epsilon}$ corresponds to ${\textstyle\bigwedge^IF\lambda^{(J-n)}}$ in $\mathbb M'$.

\item\label{4.4.f} When $\goth f =2n$ is even, $\goth g=1$, and $\epsilon=\frac{\delta-1}2$, then ${\mathbb M^\epsilon}$ is isomorphic to the complex 
\begin{equation}\label{gor}\mathbb M'=\bigoplus\limits_{\gf{i+j\le n-1}{0\le j}} {\textstyle\bigwedge^iF^*h^{(j)}}\oplus \bigoplus\limits_{\gf{p+q\le n-1}{0\le q}} {\textstyle\bigwedge^pF\lambda^{(q)}}\end{equation} of \cite[Thm.~2.4]{K92}. (Earlier versions of the complex may be found in \cite{K86,S}.)  In $\mathbb M'$, 
${\textstyle\bigwedge^iF^*h^{(j)}}$ is in position $i+2j$ and ${\textstyle\bigwedge^pF\lambda^{(q)}}$ in position ${p+2q+1}$.
The module $\bigwedge^\goth f F\otimes \bigwedge^\goth g G$ in ${\mathbb M^\epsilon}$ corresponds to $\textstyle\bigwedge^{0}F^*h^{(0)}$ 
in $\mathbb M'$, the module $V_{i,j}^{{\mathbb B}}$ in ${\mathbb M^\epsilon}$ corresponds to 
$\textstyle\bigwedge^{\goth f-i}F^*h^{(i+j-\delta)}$ 
in $\mathbb M'$, and the module $V_{I,J}^{{\mathbb T^\epsilon}}$ in ${\mathbb M^\epsilon}$ corresponds to ${\textstyle\bigwedge^IF\lambda^{(J+1-n)}}$ in $\mathbb M'$.

\item\label{4.4.g} Let $\delta=1$. If  $\epsilon=\frac{\delta+1}2=1$, then $\mathbb M^\epsilon$ is  
\begin{equation}\label{*star}0\to V_{\goth f,0}^{\mathbb B}\xrightarrow{\tau} V_{\goth f-1,0}^{\mathbb B}\xrightarrow{\tau} \dots \to V_{2,0}^{\mathbb B}\xrightarrow{\tau} V_{1,0}^{\mathbb B}\xrightarrow{\rm(\ref{blip})}\textstyle \bigwedge^\goth fF\otimes \bigwedge^\goth gG.\end{equation}
The subcomplex $0\to {\mathbb M^\epsilon}_\goth f \to\dots\to{\mathbb M^\epsilon}_1$ of $\mathbb M^\epsilon$ is a truncation of the Koszul complex associated to $\tau$.
If  $\epsilon=\frac{\delta-1}2=1$, then $\mathbb M^\epsilon$ is  
\begin{equation}\label{**}0\to V_{0,0}^{\mathbb T^\epsilon} \xrightarrow{\rm(\ref{blop})}\textstyle \bigwedge^\goth fF\otimes \bigwedge^\goth gG.\end{equation}
Once we prove Theorem~\ref{main}, then the exactness of (\ref{*star}) gives that $I_g(\Psi)=\alpha I_1(\tau)$ for some element $\alpha$ in $R$ and the exactness of (\ref{**}) identifies $\alpha$ as a generator of the image of
\begin{equation}\label{more**}\textstyle c(1)=\mu(X_1)\wedge (\bigwedge^{\goth g-1}\Psi)(X_2\wedge\dots \wedge X_{\goth g})\otimes Y_{\goth g}\wedge \dots\wedge Y_1\end{equation}
under any isomorphism from $\bigwedge^\goth f F\otimes \bigwedge^\goth gG$ to $R$. (We use the notation of \ref{87Not1}.) These ideas, but not this phrasing, are already known by Mark Johnson \cite{J} and Susan Morey \cite{M}.

A more concrete version of (\ref{more**}) is obtained as follows. Let $X_1,\dots, X_\goth g$ be a basis for $G^*$, $e_1,\dots,e_{\goth f}$ a basis for $F$, $$\mu(X_k)=\sum\limits_{1\le i<j\le \goth f} a_{i,j,k}\, e_i\wedge e_j\in \textstyle \bigwedge^2F,\quad \text{  $T$ be the column vector $[\tau(e_1),\dots,\tau(e_{\goth f})]^{\rm t}$},$$ and, for $1\le k\le \goth g$, let $A_k$ be the $\goth f\times \goth f$ alternating matrix with $a_{i,j,k}$ in position $i,j$ when $i<j$. In $R$, $c(1)$ is a unit times the Pfaffian of \begin{equation}\label{no-Z}\bmatrix A_1&\vline&A_2T&\vline&\dots&\vline&A_{\goth g}T\\\hline
 T^{\rm t}A_2&\vline\\
\vdots&\vline&&&0\\T^{\rm t}A_{\goth g}&\vline\endbmatrix.\end{equation}
It is useful to observe that  (\ref{no-Z}) is also the coefficient of $Z_1$ in the Pfaffian of $$\bmatrix Z_1A_1+Z_2A_2+\dots+Z_\goth g A_\goth g&\vline&A_2T&\vline&\dots&\vline&A_{\goth g}T\\\hline
 T^{\rm t}A_2&\vline\\
\vdots&\vline&&&0\\T^{\rm t}A_{\goth g}&\vline\endbmatrix.$$
In this discussion, $Z_1,\dots,Z_\goth g$ are indeterminates over $R$; they play the role of place holders.

\item\label{4.4.h} Take $\goth f$ to be odd and $\goth g=2$. If the hypotheses of Theorem~\ref{main} are in effect, then the ideal $I_\goth g(\Psi)^{\text{gumn}}$ is perfect of grade $\goth f-2=\delta$ and $R/I_{\goth g}(\Psi)^{\text{gumn}}$ is resolved by $\mathbb M^\epsilon$, for $\epsilon=\frac{\delta-1}2$. To describe a generating set for this ideal, let $A_1$ and $A_2$ be $\goth f\times \goth f$ alternating matrices with entries in $R$ and let $T$ be a $\goth f\times 1$ column vector, again with entries in $R$. The ideal  $I_\goth g(\Psi)^{\text{gumn}}$ is generated by

\begin{equation}\label{perfect-gens}\begin{array}{l}(\text{the coefficients of $Z_1^2$ and $Z_1Z_2$ in the Pfaffian of }\bmatrix Z_1A_1+Z_2A_2&\vline&A_2T\\\hline T^{\text{transpose}}A_2&\vline&0\endbmatrix)\\+ I_2(\bmatrix A_1T&\vline&A_2T\endbmatrix),\end{array}\end{equation}provided 
(\ref{perfect-gens}) has grade at least $\delta$ and $I_1(\tau)$ has grade at least $\goth f$. Once again, $Z_1$ and $Z_2$ are indeterminates over $R$; they play the role of place holders.
\end{enumerate}
\end{examples}

\begin{chunk}\label{lead-in} The final non-zero map of $\mathbb M$ is of considerable interest. For most choices of $(\goth g,\goth f)$, this map is $$\tau: V^{\mathbb B}_{\goth f,\epsilon-1} \to V^{\mathbb B}_{\goth f-1,\epsilon-1}.$$ If this map is written as a matrix, then it looks like the matrix of Table~\ref{last map}, where $e_1,\dots e_{\goth f}$ is a basis for $F$  and $T_i=\tau(e_i)$. The matrix has $\goth f\binom{\goth g+\epsilon -2}{\goth g-1}$ rows and $\binom{\goth g+\epsilon -2}{\goth g-1}$ columns. For small values of $\goth g$  or $\delta$ this general form is perturbed slightly. The complete form of the last map is given in Observation~\ref{last-map}.\end{chunk}
\begin{table}
\begin{center}
$$\bmatrix 
T_1        &\vline&&\vline&&\vline\\
T_2        &\vline&&\vline&&\vline\\
\vdots     &\vline&0&\vline&0&\vline&0\\
T_{\goth f}&\vline&&\vline&&\vline\\\hline
           &\vline& T_1&\vline&&\vline&  \\
&\vline& T_2        &\vline&&\vline&  \\
0&\vline& \vdots     &\vline&0&\vline&0  \\
&\vline& T_{\goth f}&\vline&&\vline&  \\\hline
0&\vline&0            &\vline&\ddots&\vline&0 \\\hline
&\vline& &\vline&       &\vline&T_1\\
&\vline& &\vline&       &\vline&T_2\\
0&\vline&0 &\vline&     0  &\vline&\vdots\\
&\vline& &\vline&       &\vline&T_{\goth f}
\endbmatrix$$
\caption{{ The last non-zero map of $\mathbb M^\epsilon$ for most $(\goth g,\goth f)$. See \ref{lead-in}}, Observation~\ref{last-map}, and (\ref{4.6.6'}).}\label{last map}
\end{center}\end{table}

\begin{observation}\label{last-map} Adopt Data~{\rm{\ref{87data6}}} and 
{\rm{\ref{unm-data}}}. Retain the description of $\mathbb M^\epsilon$ as given in Definition~{\rm\ref{doodles'}} and Remarks~{\rm\ref{num-info-Med};} in particular, the value of  $N_{\max}$ is given in {(\rm\ref{last})}. Then the final non-zero map of $\mathbb M^\epsilon$ is $$d_{N_{\max}}:{\mathbb M^\epsilon}_{N_{\max}}\to {\mathbb M^\epsilon}_{N_{\max}-1}$$ and this map is

\begingroup\allowdisplaybreaks 
\begin{align}
&V_{0,0}^{\mathbb T^\epsilon} \xrightarrow{\rm(\ref{blop})}\textstyle \bigwedge^\goth fF\otimes \bigwedge^\goth gG,&&\begin{array}{l}\text{if $\delta=1$ and $\epsilon=\frac{\delta-1}2$,}\end{array}\label{4.6.1'}\vspace{10pt}\\
&V_{\goth f,0}^{\mathbb B} \xrightarrow{\bmatrix \tau\endbmatrix }\textstyle V_{\goth f-1,0}^{\mathbb B},&&\begin{array}{l}\text{if $\delta=1$ and $\epsilon=\frac{\delta+1}2$,}\end{array}\label{4.6.2'}\vspace{10pt}\\
&V_{0,\delta-1}^{\mathbb T^\epsilon}\xrightarrow{\bmatrix \Psi\\\xi \endbmatrix} \begin{matrix} V_{1,\delta-2}^{\mathbb T^\epsilon}\\\oplus\\ V^{\mathbb B}_{\goth f,\epsilon-1},\end{matrix}&&\begin{array}{l}\textstyle  \text{if $2\le \delta$ and $(\epsilon,\goth g)=(\frac{\delta-1}2,1)$},\end{array}\label{4.6.3'} \vspace{10pt}\\
&\begin{matrix} V_{0,\delta-1}^{\mathbb T^\epsilon}\\\oplus\\ V^{\mathbb B}_{\goth f,\epsilon-1}\end{matrix}\xrightarrow{\bmatrix \Psi&0\\\xi&\tau \endbmatrix} \begin{matrix} \chi(3\le \delta)V_{1,\delta-2}^{\mathbb T^\epsilon}\\\oplus\\ 
V^{\mathbb B}_{\goth f-1,\epsilon-1},\end{matrix}
&&\begin{array}{l}\textstyle  \text{if $2\le \delta$ and $(\epsilon,\goth g)$ equals $(\frac{\delta-1}2,2)$ or $(\frac\delta2,1)$},\end{array}\label{4.6.4'} \vspace{10pt}\\ 
&\begin{matrix}  V^{\mathbb B}_{\goth f,\epsilon-1}\end{matrix}\xrightarrow{\bmatrix 0\\\tau \endbmatrix}  \begin{matrix} V_{0,\delta-1}^{\mathbb T^\epsilon}\\\oplus\\ V^{\mathbb B}_{\goth f-1,\epsilon-1},\end{matrix}
&&\begin{array}{l}\textstyle  \text{if $2\le \delta$ and $(\epsilon,\goth g)$ equals $(\frac{\delta-1}2,3)$, or $(\frac\delta2,2)$,}\\\text{or $(\frac{\delta+1}2,1)$}\end{array}\label{4.6.5'} \vspace{10pt}\\
&V^{\mathbb B}_{\goth f,\epsilon-1}\xrightarrow{\bmatrix \tau \endbmatrix}   V^{\mathbb B}_{\goth f-1,\epsilon-1},
&& \begin{array}{l}\textstyle \text{if $2\le \delta$ and $(\epsilon,\goth g)$ satisfy one of the following:}\\\textstyle\text{$\epsilon=\frac{\delta-1}2$ and $4\le\goth g$, or $\epsilon=\frac{\delta}2$ and $3\le\goth g$,}\\\textstyle\text{or $\epsilon=\frac{\delta+1}2$ and $2\le\goth g$.}\end{array}\label{4.6.6'}\end{align}\endgroup

The map $\xi:V_{0,\delta-1}^{\mathbb T^\epsilon}\to V_{2(\delta-\epsilon),\epsilon-1}^{\mathbb B}$, which appears in {\rm(\ref{4.6.3'})} and {\rm(\ref{4.6.4'})}, is given by 
\begin{equation}\label{4.6.1}\xi(\gamma_{\delta-1})=
(-1)^{\epsilon-1} \sum\limits_{m\in \binom{Y}{\delta-\epsilon}} [D(\mu)](m^*)\otimes m(\gamma_{\delta-1})\in V_{2(\delta-\epsilon),\epsilon-1}^{{\mathbb B}}\subseteq {\mathbb M^\epsilon}_{\delta-1};\end{equation}
furthermore, $$
2(\delta-\epsilon)=\begin{cases}\goth f,&\text{if 
$(\epsilon,\goth g)=(\frac{\delta-1}2,1)$, and}\\
 \goth f-1,&\text{if 
$(\epsilon,\goth g)=(\frac{\delta-1}2,2)$ or $(\frac{\delta}2,1)$.}\end{cases}$$
\end{observation}
\begin{proof} If $\delta=1$, then $\mathbb M^\epsilon$ is recorded in (\ref{*star}) and (\ref{**}). Henceforth, we assume $2\le \delta$. The differential $d_{N_{\max}}$ is given in Definition~\ref{doodles'}.\ref{doodles'.b}. We need only deal with the modules. The module ${\mathbb M^\epsilon}_{N_{\max}}$ is given in Remark~\ref{num-info-Med}.\ref{num-info-Med.a}. One easily calculates 
the 
summands of ${\mathbb M^\epsilon}_{N_{\max}-1}$.
 For example, if $\epsilon=\frac{\delta+1}2$
and $V_{i,j}^{\mathbb B}$ is a summand of ${\mathbb M^\epsilon}_{N_{\max}-1}={\mathbb M^\epsilon}_{\goth f-1}$, then
$$i+2j-\delta+1=\goth f-1 \quad \text{and}\quad  j\le \textstyle \epsilon -1=\frac{\delta+1}2-1;$$hence,
$$\goth f-2-i+\delta=2j\le \delta-1.$$
It follows that $\goth f-1\le i$. On the other hand, $f-i$ and $\delta$ have the same parity; so $i=\goth f-1$. 
A similar argument works in the remaining cases.
\end{proof}

\section{The double complexes ${\mathbb T^\epsilon}$ and ${\mathbb B}$.}\label{Top-and-Bot}
Adopt Data~{\rm{\ref{87data6}}} and 
{\rm{\ref{unm-data}}}.
The double complexes  ${\mathbb T^\epsilon}$ and ${\mathbb B}$, which are used in the construction of ${\mathbb M^\epsilon}$, both are quotients of the double complex ${\mathbb V}$. In Definition~\ref{87.11}, we introduce all three double complexes $\mathbb{V}$,  ${\mathbb T^\epsilon}$, and ${\mathbb B}$; we also introduce the sub-double-complex $\mathbb {U}$ of $\mathbb{V}$ with $\mathbb{V}/\mathbb{U}$ equal to ${\mathbb B}$. (We have no need for the subcomplex of $\mathbb{V}$ which defines ${\mathbb T^\epsilon}$.) Please keep in mind that, for our purposes, ${\mathbb T^\epsilon}$ and ${\mathbb B}$ are the important complexes. We have introduced $\mathbb{V}$ and $\mathbb{U}$ in order to calculate the homology of ${\mathbb B}$; see Lemma~\ref{87.15-bot}. (It  is not difficult to compute the homology of ${\mathbb T^\epsilon}$; see Lemma~\ref{87.15}.) 
\begin{definition}\label{87.11} Adopt Data~{\rm{\ref{87data6}}} and 
{\rm{\ref{unm-data}}}.
\begin{table}
\begin{center}
$$\mathbb{V}:\quad\xymatrix{&\vdots\ar[d]^\tau&\vdots\ar[d]^\tau&\vdots\ar[d]\\
\cdots\ar[r]^{\Psi\ \ \ \ }&V_{\delta,1}\ar[r]^{\Psi}\ar[d]^{\tau}&V_{\delta+1,0}\ar[r]\ar[d]^{\tau}&0\ar[d]\\
\cdots\ar[r]^{\Psi\ \ \ \ }&V_{\delta-1,1}\ar[r]^{\Psi}\ar[d]^{\tau}&V_{\delta,0}\ar[r]^{\bigwedge^{\goth g}\Psi\ \ \ \ }\ar[d]^{\tau}&\bigwedge^{\goth f}F\otimes \bigwedge^{\goth g}G\ar[d]^{\tau}\\
\cdots\ar[r]^{\Psi\ \ \ \ }&V_{\delta-2,1}\ar[r]^{\Psi}\ar[d]^{\tau}&V_{\delta-1,0}\ar[r]^{\bigwedge^{\goth g}\Psi\ \ \ \ }\ar[d]^{\tau}&\bigwedge^{\goth f-1}F\otimes \bigwedge^{\goth g}G\ar[d]^{\tau}\\
&\vdots\ar[d]^\tau&\vdots\ar[d]^\tau&\vdots\ar[d]^\tau\\
\cdots\ar[r]&V_{0,1}\ar[r]^\Psi\ar[d]&V_{1,0}\ar[r]^{\bigwedge^{\goth g}\Psi\ \ \ \ }\ar[d]^\tau&\bigwedge^{\goth g+1}F\otimes \bigwedge^{\goth g}G\ar[d]^\tau\\
\cdots\ar[r]&0\ar[r]&V_{0,0}\ar[r]^{\bigwedge^{\goth g}\Psi\ \ \ \ }&\bigwedge^{\goth g}F\otimes \bigwedge^{\goth g}G\\
}$$
\hrule
$${\mathbb T^{\epsilon}}:\quad \xymatrix{&\vdots\ar[d]^\tau&\vdots\ar[d]^\tau&\vdots\ar[d]^\tau\\
\dots\ar[r]&V_{0,\epsilon+2}\ar[r]^\Psi\ar[d]&V_{1,\epsilon+1}\ar[d]^\tau\ar[r]^\Psi&V_{2,\epsilon}\ar[d]^{\tau}\\
\dots\ar[r]&0\ar[r]\ar[d]&V_{0,\epsilon+1}\ar[r]^{\Psi}\ar[d]&V_{1,\epsilon}\ar[d]^\tau\\
\dots\ar[r]&0\ar[r]&0\ar[r]&V_{0,\epsilon},}$$
\caption{{ The double complexes $\mathbb{V}$ and ${\mathbb T^{\epsilon}}$ from Definition \ref{87.11}.}}\label{8br-of-X}
\end{center}\end{table}
\begin{table}
\begin{center}
$${\mathbb B}:\quad \xymatrix{
&\vdots\ar[d]^{\tau}&\vdots\ar[d]^{\tau}&\vdots\ar[d]^{\tau}&\vdots\ar[d]\\
\dots\ar[r]^{\Psi\ \ \ \ }&V_{\delta,2}\ar[r]^{\Psi}\ar[d]^{\tau}&V_{\delta+1,1}\ar[r]^{\Psi}\ar[d]^{\tau}&V_{\delta+2,0}\ar[r]\ar[d]^{\tau}&0\ar[d]\\
\dots\ar[r]^{\Psi\ \ \ \ }&V_{\delta-1,2}\ar[r]^{\Psi}\ar[d]^{\tau}&V_{\delta,1}\ar[r]^{\Psi}\ar[d]^{\tau}&V_{\delta+1,0}\ar[r]\ar[d]^{\tau}&0\ar[d]\\
\dots\ar[r]^{\Psi\ \ \ \ }&V_{\delta-2,2}\ar[r]^{\Psi}&V_{\delta-1,1}\ar[r]^{\Psi}&V_{\delta,0}\ar[r]&\bigwedge^{\goth f}F\otimes \bigwedge^{\goth g}G,}$$
\hrule
$$\mathbb{U}:\quad \xymatrix{\cdots\ar[r]^{\Psi\ \ \ \ }&V_{\delta-2,1}\ar[r]^{\Psi}\ar[d]^{\tau}&V_{\delta-1,0}\ar[r]^{\bigwedge^{\goth g}\Psi\ \ \ \ }\ar[d]^{\tau}&\bigwedge^{\goth f-1}F\otimes \bigwedge^{\goth g}G\ar[d]^{\tau}\\
&\vdots\ar[d]^\tau&\vdots\ar[d]^\tau&\vdots\ar[d]^\tau\\
\cdots\ar[r]&V_{0,1}\ar[r]^\Psi\ar[d]&V_{1,0}\ar[r]^{\bigwedge^{\goth g}\Psi\ \ \ \ }\ar[d]^\tau&\bigwedge^{\goth g+1}F\otimes \bigwedge^{\goth g}G\ar[d]^\tau\\
\cdots\ar[r]&0\ar[r]&V_{0,0}\ar[r]^{\bigwedge^{\goth g}\Psi\ \ \ \ }&\bigwedge^{\goth g}F\otimes \bigwedge^{\goth g}G
}$$
\caption{{ The double complexes ${\mathbb B}$ and $\mathbb{U}$ from Definition \ref{87.11}.}}\label{8br-of-X2}
\end{center}\end{table}
Define $\mathbb {V}$, ${\mathbb T^{\epsilon}}$, ${\mathbb B}$ and $\mathbb {U}$ to be the double complexes which are  given in Tables~\ref{8br-of-X} and \ref{8br-of-X2}.  
The double complexes  ${\mathbb T^{\epsilon}}$ and ${\mathbb B}$ are quotients of $\mathbb{V}$ and $\mathbb {U}$ is the subcomplex of $\mathbb {V}$ with $\mathbb{V}/\mathbb{U}={\mathbb B}$.
The modules are indexed in such a way that  $$\textstyle\operatorname{Tot}(\mathbb{V})_0=(\bigwedge^{\goth f}F\otimes \bigwedge^{\goth g}G)\oplus \bigoplus\limits_{i+2j=\delta-1}V_{i,j}.$$

\begin{enumerate} [\rm(a)]
\item The module $V_{i,j}$ is the free $R$-module  $\bigwedge^i F\otimes D_j(G^*)$, see \ref{V-bullet}.
\item \label{87.11.d}The $R$-module homomorphism $$\xymatrix{V_{i,j}\ar[d]^{\tau}\\V_{i-1,j}}
$$   sends $f_i\otimes \gamma_j$ to $\tau(f_i)\otimes\gamma_j$. (Recall the convention \ref{2.7}.) 
\item\label{87.11.e} The $R$-module homomorphism  $\xymatrix{V_{i,j}\ar[r]^{\Psi\ \ \ \ \ }&V_{i+1,j-1}}$  sends $f_i\otimes \gamma_j$ to
$\sum\limits_\ell \Psi(X_\ell)\wedge f_i\otimes Y_{\ell}(\gamma_j)$. (The element $\sum\limits_{\ell=1}^{\goth g} X_\ell\otimes Y_{\ell}$ of $G^*\otimes G$ is explained in \ref{87Not1}.) 
\item\label{87.11.f} The $R$-module homomorphism  $V_{i,0}\xrightarrow{\bigwedge^{\goth g}\Psi} \bigwedge^{\goth g+i}F\otimes \bigwedge^{\goth g}G$ 
is $$f_{i}\mapsto f_{i}\wedge ({\textstyle\bigwedge}^{\goth g}\Psi)(\omega_{G^*})\otimes \omega_G.$$ (The symbols $\omega_{G^*}$ and $\omega_G$ are explained in {\rm\ref{omega}}.) \end{enumerate}
\end{definition}

\begin{remarks}\begin{enumerate}[\rm(a)]\item We always use $f_i$ for an arbitrary element of $\bigwedge^iF$ and $\gamma_j$ for an arbitrary element of $D_j(G^*)$; see \ref{2.7}.
\item Recall from the definition of $\Psi$ in Data~\ref{87data6} that $\tau\circ \Psi=0$. Observe that  the squares 
$$\xymatrix{V_{i,j}\ar[r]\ar[d]&V_{i+1,j-1}\ar[d]\\V_{i-1,j}\ar[r]&V_{i,j-1}}$$ from the complexes of Definition~\ref{87.11} anti-commute  and the squares 
$$\xymatrix{V_{i,0}\ar[r]\ar[d]&\bigwedge^{i+\goth g}F\otimes \bigwedge^{\goth g}G\ar[d]\\V_{i-1,0}\ar[r]&\bigwedge^{i-1+\goth g}F\otimes \bigwedge^{\goth g}G}$$
commute on-the-nose.

\item Let $N$ be an integer.
\begin{enumerate}[\rm(i)]
\item The module $\operatorname{Tot}(\mathbb{V})_N$ is equal to
$$\begin{cases} 
 (\textstyle{\bigwedge}^{\goth f+N}F\otimes \bigwedge^gG)
\oplus \sum\limits_{(\text{\ref{87dis3}})}V_{i,j},
&\text{if $-\delta\le N$, and}\\
0,&\text{if $N<-\delta$,}
\end{cases}$$
where the sum is taken over
\begin{equation}\label{87dis3} \left\{(i,j)\mid i+2j-\delta+1=N,\quad \text{and}\quad0\le i, j\right\}.\end{equation}

\item The module $\operatorname{Tot}({\mathbb T^{\epsilon}})_N$ is equal to
$$\begin{cases} 0,&\text{if $N\le 0$, and}\\
\sum\limits_{(\text{\ref{87dis1}})}V_{i,j},&\text{if $1\le N$,}
\end{cases}$$where the sum is taken over 
\begin{equation}\label{87dis1}\{(i,j)\mid i+2j-\delta+1=N,\quad 0\le i,\quad\text{and}\quad \epsilon\le j\}.\end{equation}

\item The module $\operatorname{Tot}({\mathbb B})_N$ is equal to
$$\begin{cases} 
0,&\text{if $N\le -1$,}\\
\bigwedge^{\goth f}F\otimes\bigwedge^{\goth g}G,&\text{if $N=0$, and}\\
\sum\limits_{(\text{\ref{87dis2}})}V_{i,j},&\text{if $1\le N$,}
\end{cases}$$
where the sum is taken over
\begin{equation}\label{87dis2} \left\{(i,j)\mid i+2j-\delta+1=N,\quad \delta\le i+j,\quad \text{and}\quad0\le i, j\right\}.\end{equation}

\item   The module $\operatorname{Tot}(\mathbb{U})_N$ is equal to
$$\begin{cases} 
\sum\limits_{(\text{\ref{87dis4}})}V_{i,j},
&\text{if $0\le N$,}\\
 (\textstyle{\bigwedge}^{\goth f+N}F\otimes \bigwedge^gG)
\oplus \sum\limits_{(\text{\ref{87dis4}})}V_{i,j},
&\text{if $-\delta \le N\le -1$, and}\\
0,&\text{if $N<-\delta$.}
\end{cases}$$
where the sum is taken over
\begin{equation}\label{87dis4} \left\{(i,j)\mid i+2j-\delta+1=N,\quad i+j\le\delta-1\quad \text{and}\quad0\le i, j\right\}.\end{equation}
\end{enumerate}
\end{enumerate}
\end{remarks}

\begin{lemma}\label{87.15}Adopt Data~{\rm\ref{87data6}}. Assume Hypothesis~{\rm\ref{87hyp6}.\ref{87hyp6-b}} holds.    Recall the double complexes ${\mathbb T}^{\epsilon}$  and $\mathbb{V}$ of Definition~{\rm \ref{87.11}}. \begin{enumerate}[\rm(a)] \item\label{87.15.a}
If $N$ is  an  integer, then 
$$\operatorname{H}_N(\operatorname{Tot}({\mathbb T^{\epsilon}}))\simeq \begin{cases} \frac R{I_1(\tau)} \otimes D_{\frac{\delta+N-1}2}(G^*),&\text{if $\delta+N-1$ is even and $2\epsilon-\delta+1\le N$, and}\\
0,&\text{otherwise}.\end{cases}$$
Furthermore, if $\delta+N-1$ is even and $2\epsilon-\delta+1\le N$, then 
\begin{equation}\label{8basis}\textstyle\{ \llbracket z_{m^*}\rrbracket \mid m\in \binom{Y}{\frac{\delta+N-1}2}\}\end{equation} is a basis for the   
 free $R/I_1(\tau)$-module $\operatorname{H}_N(\operatorname{Tot}({\mathbb T^{\epsilon}}))$. $($See Remark~{\rm\ref{8.16}} and {\rm(\ref{rep})}.$)$

\item\label{87.21.d} If $1\le \delta$ and $1\le N$, then$$\operatorname{H}_N(\operatorname{Tot}(\mathbb{V}))\simeq\begin{cases}
R/I_1(\tau)\otimes D_{\frac{\delta+N-1}2}(G^*),&\text{if $\delta+N-1$ is even, and}\\
0,&\text{otherwise.}\end{cases}
$$ Furthermore, if $\delta+N-1$ is even and $1\le N$, then 
\begin{equation}\label{8basis.b}\textstyle\{ \llbracket Z_{m^*}\rrbracket \mid m\in \binom{Y}{\frac{\delta+N-1}2}\}\end{equation} is a basis for the   
 free $R/I_1(\tau)$-module $\operatorname{H}_N(\operatorname{Tot}(\mathbb{V}))$. $($See Remark~{\rm\ref{8.16}} and {\rm(\ref{rep})}.$)$
\end{enumerate}
 \end{lemma}
\begin{remarks}\label{8.16} (a) We explain the notation in (\ref{8basis}) and (\ref{8basis.b}). For each element $\gamma$  of $D_{\frac{\delta+N-1}2}(G^*)$,  cycles $z^\epsilon_{\gamma}$ in $\operatorname{Tot}({\mathbb T^{\epsilon}})_{\frac{\delta+N-1}2}$ and $Z_{\gamma}$ in $\operatorname{Tot}(\mathbb{V})_{\frac{\delta+N-1}2}$ are defined in Lemma~\ref{8z-gamma}. In particular, the homology class $\llbracket z_{m^*}\rrbracket$ of  (\ref{8basis}) refers to the homology class of the cycle $z_{m^*}$ in $\operatorname{Tot}({\mathbb T^{\epsilon}})_{\frac{\delta+N-1}2}$ which corresponds to the element $m^*$ of $D_{\frac{\delta+N-1}2}(G^*)$ as described in Lemma~\ref{8z-gamma}. 

\medskip\noindent (b) A less technical statement of Lemma~\ref{87.15}.\ref{87.15.a} is that the non-zero homology of $\operatorname{Tot}({\mathbb T^{\epsilon}})$ is ``caused by'' the external corners $V_{0,j_0}$ in ${\mathbb T^{\epsilon}}$. Furthermore, the corner $V_{0,j_0}$ contributes $$R/I_1(\tau)\otimes D_{j_0}(G^*)$$ to $\operatorname{H}_{2j_0-\delta+1}(\operatorname{Tot}({\mathbb T^{\epsilon}}))$. We use (\ref{87dis1}) to read that $V_{0,j_0}$ sits in position $2j_0-\delta+1$ of $\operatorname{Tot}({\mathbb T^{\epsilon}})$. Also, it is immediate that  if $N$ is $2j_0-\delta+1$, then $j_0=(N+\delta-1)/2$.
\end{remarks}

\begin{proof2} (\ref{87.15.a}) Each column of ${\mathbb T^{\epsilon}}$ is the tensor product of a free $R$-module with the Koszul complex associated to a regular sequence. Consequently, the homology of each column of ${\mathbb T^{\epsilon}}$ is well understood. We view $\operatorname{Tot}({\mathbb T^{\epsilon}})$ as the limit of a sequence of mapping cones; each mapping cone adjoins one more column from ${\mathbb T^{\epsilon}}$. The columns are arranged in such a way that the homology of the new column does not interfere with the homology of the total complex of the previous columns. We provide some details. Let $j_0$ be an integer with $\epsilon\le j_0$; and $\mathbb A$ and $\mathbb P$  be the subcomplexes $$\mathbb {A}=  \operatorname{Tot}(\bigoplus_{}V_{*,\le j_0-1})
\quad\text{and}\quad 
\mathbb {P}=  \operatorname{Tot}(\bigoplus_{}V_{*,\le j_0})
$$ of $\operatorname{Tot}({\mathbb T^{\epsilon}})$.
We assume by induction that $$\text{$\operatorname{H}_N(\mathbb A)\neq 0$ 
only if $\delta+N-1$ is even, and $2\epsilon-\delta+1\le N \le 2j_0-\delta-1$.}$$
We see that $\mathbb{P}/\mathbb{A}$ is the column $\bigoplus_{i}V_{i,j_0}$ of ${\mathbb T^\epsilon}$; thus $\mathbb{P}/\mathbb{A}$ is the tensor product of the Koszul complex associated to $\tau$ with $D_{j_0}(G^*)$. Hypothesis~{\rm\ref{87hyp6}.\ref{87hyp6-b}} ensures that 
$$\operatorname{H}_N(\mathbb{P}/\mathbb{A})\simeq \begin{cases}
R/I_1(\tau)\otimes D_{j_0}(G^*),&\text{if $N=2j_0-\delta+1$, and}\\
0,&\text{otherwise.}
\end{cases}$$
The long exact sequence of homology which is induced by the short exact sequence of complexes
$$0\to \mathbb{A}\to \mathbb{P}\to \mathbb{P}/\mathbb{A}\to 0$$
yields that 
$$\operatorname{H}_N(\mathbb{P})\simeq \begin{cases} 
0,&\text{if $2j_0-\delta+2\le N$},\\ 
R/I_1(\tau)\otimes D_{j_0}(G^*),&\text{if $N=2j_0-\delta+1$},\\ 
\operatorname{H}_N(\mathbb{A}),&\text{if $N\le 2j_0-\delta$}.\end{cases}$$
Furthermore, the cycles that represent $\operatorname{H}_N(\mathbb{A})$ continue to represent $\operatorname{H}_N(\mathbb{P})$ for $N\le 2j_0-\delta$. Also, it is not difficult to identify a basis for the free $R/I_1(\tau)$-module  $\operatorname{H}_{2j_0-\delta+1}(\mathbb{P})$. One takes a basis for $D_{j_0}(G^*)$.  Each element of this basis is a cycle in the complex  $\mathbb{P}/\mathbb{A}$. One lifts the cycle in $\mathbb{P}/\mathbb{A}$ to a cycle in $\mathbb{P}$ and then takes the homology class. In \ref{87Not1}, we identified the basis  $\{m^*\mid m\in \binom{Y}{j_0}\}$ for $D_{j_0}(G^*)$. In Lemma~\ref{8z-gamma} we lift $m^*$ to the cycle $z_{m^*}$ in $\mathbb P\subseteq \operatorname{Tot}({\mathbb T^{\epsilon}})$.

\medskip\noindent(\ref{87.21.d}) The proof is similar to the proof of (\ref{87.15.a}). 
The only difference is that the stair-case pattern involving exterior lower left-hand corners starts with the second column instead of the first column. That is, $\operatorname{H}_{-\delta+1}(\operatorname{Tot}(\mathbb{V}))$ and $\operatorname{H}_{-\delta}(\operatorname{Tot}(\mathbb{V}))$ are not calculated using the method of 
the proof of (\ref{87.15.a}). On the other hand, the statement of the result makes no claim about these homologies because  $-\delta<-\delta+1\le 0<N$. Once $1\le N$, then 
the module $\operatorname{H}_N(\operatorname{Tot}(\mathbb{V}))$ is non-zero if and only if $\operatorname{Tot}(\mathbb{V})_N$ contains one of the corners $V_{0,j_0}$ of $\mathbb{V}$.  Furthermore, (\ref{87dis3}) shows that  $V_{0,j_0}$ is a summand of $\operatorname{Tot}(\mathbb{V})_N$ if and only if $2j_0-\delta+1=N$. The rest of the proof is identical to the proof of (\ref{87.15.a}).
 \end{proof2}

\begin{lemma}\label{8z-gamma}Adopt Data~{\rm\ref{87data6}}  and  recall the double complexes $\mathbb{V}$, ${\mathbb T^{\epsilon}}$, and ${\mathbb B}$ of Definition~{\rm \ref{87.11}}.
Let $N$ be a positive integer with  $\delta+N-1$ even and let $\gamma$ be an element of $D_{\frac{\delta+N-1}2}(G^*)$. Define $Z_{\gamma}$ to be the following element of $\mathbb{V}:$
\begin{equation}\label{Z-gamma}
Z_{\gamma}=\sum\limits_{j=0}^{\frac{\delta+N-1}2}\sum\limits_{m\in \binom{Y}j}(-1)^{j}
[D(\mu)](m(\gamma))\otimes m^*\in V_{\delta+N-1-2j,j}\subset \mathbb{V}\end{equation}and define $z^\epsilon_{\gamma}$ and $\zeta_{\gamma}$ to be the images of $Z_{\gamma}$ in the quotient  double complexes ${\mathbb T^{\epsilon}}$ and ${\mathbb B}$, respectively.
Then 
\begin{enumerate}[\rm(a)]
\item $Z_\gamma$, $z^\epsilon_{\gamma}$, and $\zeta_{\gamma}$ are elements $\operatorname{Tot}(\mathbb {V})_{N}$, $\operatorname{Tot}({\mathbb T^{\epsilon}})_{N}$, and $\operatorname{Tot}({\mathbb B})_{N}$, respectively,
\item $Z_\gamma$, $z^\epsilon_{\gamma}$, and $\zeta_{\gamma}$ are cycles in $\operatorname{Tot}(\mathbb {V})$, $\operatorname{Tot}({\mathbb T^{\epsilon}})$, and $\operatorname{Tot}({\mathbb B})$, respectively,
\item $z^\epsilon_{\gamma}=\sum\limits_{J=\epsilon}^{\frac{\delta+N-1}2}\sum\limits_{m_1\in \binom{Y}{\frac{\delta+N-1}2-J}}(-1)^{J}
[D(\mu)](m_1^*)\otimes m_1(\gamma)\in V_{\delta+N-1-2J,J}\subseteq \operatorname{Tot}({\mathbb T^{\epsilon}})_N,\text{ and}$
\item $\zeta_{\gamma}=\sum\limits_{j=0}^{N-1}\sum\limits_{m_1\in \binom{Y}{\frac{\delta+N-1}2-j}}(-1)^{j}
[D(\mu)](m_1^*)\otimes m_1(\gamma)\in V_{\delta+N-1-2j,j}\subseteq \operatorname{Tot}({\mathbb B})_N$.

\end{enumerate}
\end{lemma}
\begin{proof} ${}$

\noindent(a)  Observe  that  $[D(\mu)](m(\gamma))\otimes m^*$, in the notation of (\ref{Z-gamma}), is an element of $V_{\delta+N-1-2j,j}$, which, according to (\ref{87dis3}), is a summand of $\operatorname{Tot}(\mathbb{V})_N$. 

\medskip\noindent(c,d) According to (\ref{87dis1}) and (\ref{87dis2}), $V_{\delta+N-1-2j,j}$ is a non-zero summand of ${\mathbb T^{\epsilon}}$ if and only if $\epsilon\le j$; and $V_{\delta+N-1-2j,j}$ is a non-zero summand of ${\mathbb B}$ if and only if $j\le N-1$. Observe that 
$$\sum\limits_{m\in \binom{Y}J}m(\gamma)\otimes m^*=
\sum\limits_{m_1\in \binom{Y}{\frac{\delta+N-1}2-J}} m_1^*\otimes m_1(\gamma)
\in D_{\frac{\delta+N-1}2-J}(G^*)\otimes D_{J}(G^*).$$

 \medskip\noindent(b) It suffices to show that $Z_\gamma$ is a cycle in $\operatorname{Tot}(\mathbb{V})$. The differential $d$ of $\operatorname{Tot}(\mathbb{V})$ takes $\operatorname{Tot}(\mathbb{V})_N$ to $$\operatorname{Tot}(\mathbb{V})_{N-1}=
\sum\limits_{j=0}^{\frac{\delta+N-3}2}V_{\delta+N-2-2j,j}.$$ Fix $j_0$ with  $0\le j_0\le \frac{\delta+N-3}2$. Observe that the component of $d(Z_\gamma)$ in $V_{\delta+N-2-2j_0,j_0}$ is $A+B$, where $A$ is the horizontal contribution and $B$ is the vertical contribution. In other words, 
$$\begin{cases}A=\sum\limits_{\ell=1}^{\goth g} \sum\limits_{m\in \binom{Y}{j_0+1}}(-1)^{j_0+1}
\Psi(X_\ell)\wedge [D(\mu)](m(\gamma))\otimes Y_{\ell}(m^*) \quad\text{and}\\ B= \sum\limits_{m\in \binom{Y}{j_0}}(-1)^{j_0}
\tau([D(\mu)](m(\gamma)))\otimes m^*.\end{cases}$$
The elements
\begin{equation}\label{trick1}\sum\limits_{\ell=1}^{\goth g}\sum\limits_{m\in \binom{Y}{j_0+1}}X_\ell\otimes m\otimes Y_{\ell}(m^*)\quad\text{ and }\quad
\sum\limits_{\ell=1}^{\goth g}\sum\limits_{m\in \binom{Y}{j_0}}X_\ell\otimes Y_{\ell}m\otimes m^*\end{equation}
of $G^*\otimes \operatorname{Sym}_{j_0+1}G\otimes D_{j_0}(G^*)$ are equal; so,
$$A=\sum\limits_{\ell=1}^{\goth g} \sum\limits_{m\in \binom{Y}{j_0}}(-1)^{j_0+1}
\Psi(X_\ell)\wedge [D(\mu)]((Y_{\ell}m)(\gamma))\otimes m^*.$$ If $\gamma'$ is an arbitrary element of $D_\bullet(G^*)$, then
\begin{equation}\label{trick2}\tau([D(\mu)](\gamma'))= \sum\limits_{\ell=1}^{\goth g} \tau(\mu(X_{\ell}))\wedge [D(\mu)](Y_\ell(\gamma')).\end{equation} (It suffices to test this assertion when $\gamma'=X_1^{(a_1)}X_2^{(a_2)}\cdots X_\ell^{(a_\ell)}$.) It follows that
$$A=\sum\limits_{m\in \binom{Y}{j_0}}(-1)^{j_0+1}
\tau( [D(\mu)](m(\gamma)))\otimes m^*=-B.$$ 
\end{proof}

\begin{proposition}\label{87.21} Adopt Data~{\rm\ref{87data6}}. Assume   Hypotheses~{\rm\ref{87hyp6}} holds. Recall the double complex $\mathbb{U}$ of Definition~{\rm\ref{87.11}}. 
Then  $\operatorname{H}_i(\operatorname{Tot}(\mathbb{U}))=0$, for $0\le i$.
\end{proposition}

\begin{proof} 
 Each row of  $\mathbb {U}$ is isomorphic to a truncation of a generalized Eagon-Northcott complex. Recall that, for each $R$-module homomorphism $$\Phi: G^*\to F$$(or equivalently ${\Phi^*:F^*\to G}$), there is a family of generalized Eagon-Northcott complexes $\{\mathcal C_{\Phi^*}^{i}\}$ (see, for example, \cite[Appendix A.2]{Ei95}). If \begin{equation}\label{8acyclic-range}-1\le i\le \delta+1,\end{equation} then  $\mathcal C_{\Phi^*}^{i}$ has length $\delta+1$; and, if $\delta+1\le \operatorname{grade} I_{\goth g}(\Phi)$, then   $\mathcal C_{\Phi^*}^{i}$  is acyclic for $i$ satisfying (\ref{8acyclic-range}). Furthermore, the complexes  $\mathcal C_{\Phi^*}^{i}$, for $i$ satisfying (\ref{8acyclic-range}), exhibit depth-sensitivity. In particular, if $\delta\le \operatorname{grade} I_{\goth g}(\Phi)$, then $\operatorname{H}_j(\mathcal C_{\Phi^*}^i)=0$ for $2\le j$ and $i$ satisfying  (\ref{8acyclic-range}). Hypothesis~{\rm\ref{87hyp6}.\ref{87hyp6-a}} guarantees that $\delta\le \operatorname{grade} I_{\goth g}(\Psi)$; thus, $\operatorname{H}_j(\mathcal C_{\Psi^*}^i)=0$ for $2\le j$ and $i$ satisfying  (\ref{8acyclic-range}). The rows of $\mathbb {U}$ are isomorphic to  truncations of the complexes $\mathcal C_{\Psi^*}^i$, for $1\le i\le \delta$.
In particular, the top row of $\mathbb{U}$, which is also the longest row of $\mathbb{U}$, is isomorphic to the truncation of 
$$\textstyle \mathcal C^1_{\Psi^*}:\quad  0\to \bigwedge^{\goth f}F^* \otimes D_{\delta-1}G \to\cdots \to\bigwedge^{\goth g+2}F^*\otimes D_1G \to\bigwedge^{\goth g+1}F^*\otimes D_0G \to F^*\xrightarrow{\Psi^*} G $$ at $F^*$. The resulting truncation is an acyclic complex of length $\delta$; so we use the full force of Hypothesis~{\rm\ref{87hyp6}.\ref{87hyp6-a}} to conclude that the top row of $\mathbb{U}$ acyclic. In the lower rows of $\mathbb{U}$, the truncation of the corresponding generalized Eagon-Northcott complex $\mathcal C^i_{\Psi^*}$ is more severe, consequently, the row is shorter than the top row of $\mathbb{U}$ and  one need not use all of Hypothesis~{\rm\ref{87hyp6}.\ref{87hyp6-a}} to conclude that the row is acyclic. We observe that the complex $\mathcal C^1$ is also called the Buchsbaum-Rim complex. 
 
At any rate, each row of $\mathbb {U}$ has non-zero homology only at the right hand boundary. 
We observe that the modules from the right hand column of $\mathbb {U}$
 all are summands of  $\bigoplus_{i\le -1}\operatorname{Tot}(\mathbb{U})_{i}$. In particular, $\bigwedge^{\goth f-1}F\otimes \bigwedge^{\goth g}G$, which is the right-most module in the top row of $\mathbb{U}$, is a summand of $\operatorname{Tot}(\mathbb{U})_{-1}$. Standard results about total complexes now yield that $\operatorname{H}_N(\operatorname{Tot}(\mathbb{U}))=0$, for $0\le N$. 
\end{proof}

\begin{lemma}\label{87.15-bot}Adopt Data~{\rm\ref{87data6}}. Assume  Hypotheses~{\rm\ref{87hyp6}} hold. Recall the double complex ${\mathbb B}$ of Definition~{\rm \ref{87.11}}.
If $N$ is  an  integer, then 
$$\operatorname{H}_N(\operatorname{Tot}({\mathbb B}))\simeq \begin{cases} \frac R{I_1(\tau)} \otimes D_{\frac{\delta+N-1}2}(G^*),&\text{if $\delta+N-1$ is even and $1\le N$,}\\
\overline{R},&\text{if $N=0$, and}\\
0,&\text{otherwise}.\end{cases}$$
Furthermore, if $\delta+N-1$ is even and $1\le N$, then 
\begin{equation}\label{8basis-bot}\textstyle\{ \llbracket \zeta_{m^*}\rrbracket \mid m\in \binom{Y}{\frac{\delta+N-1}2}\}\end{equation} is a basis for the   
 free $R/I_1(\tau)$-module $\operatorname{H}_N(\operatorname{Tot}({\mathbb B}))$, where the formula for $\zeta_{m^*}$ is given in Lemma~{\rm\ref{8z-gamma}}.
 \end{lemma}

\begin{proof}The value of $\operatorname{H}_N(\operatorname{Tot}({\mathbb B}))$ for $N\le 0$ is clear from the definition of ${\mathbb B}$. For positive $N$, the short exact sequence of complexes 
$$0\to\operatorname{Tot}(\mathbb{U})\to \operatorname{Tot}(\mathbb{V})\to \operatorname{Tot}({\mathbb B})\to 0,$$ combined with Proposition~\ref{87.21}, yields that $\operatorname{H}_N(\operatorname{Tot}(\mathbb{V}))\simeq \operatorname{H}_N(\operatorname{Tot}({\mathbb B}))$ for $1\le N$; hence, Lemma~\ref{87.15}.\ref{87.21.d} completes the calculation of the isomorphism class of $\operatorname{H}_i(\operatorname{Tot}({\mathbb B}))$. Furthermore, according to Lemma~\ref{8z-gamma}, (\ref{8basis-bot}) is the basis for $\operatorname{H}_N(\operatorname{Tot}({\mathbb B}))$ that corresponds to the basis (\ref{8basis.b})  of 
$\operatorname{H}_N(\operatorname{Tot}(\mathbb{V}))$.
\end{proof}

\section{The map of complexes $\xi^\epsilon:\operatorname{Tot}({\mathbb T^\epsilon})\to \operatorname{Tot}({\mathbb B})$.}\label{xi}

Data~\ref{87data6} and \ref{unm-data} are in effect throughout this section.    The double complexes ${\mathbb T^{\epsilon}}$ and ${\mathbb B}$ are defined in \ref{87.11} and the maps and modules of ${\mathbb M^{\epsilon}}$ are given in Definition~\ref{doodles'}.

In this section we  define a map of complexes $\xi^{\epsilon}:\operatorname{Tot}({\mathbb T^{\epsilon}})\to \operatorname{Tot}({\mathbb B})$.  Once $\xi^{\epsilon}$ is  defined, we let ${\mathbb L^{\epsilon}}$ be the mapping cone of $\xi^{\epsilon}$. We prove that when Hypotheses~\ref{87hyp6} are in effect, then  
${\mathbb L^{\epsilon}}$ is a resolution,
${\mathbb M^{\epsilon}}$ is a 
subcomplex of ${\mathbb L^{\epsilon}}$, and  ${\mathbb L^{\epsilon}}/{\mathbb M^{\epsilon}}$ is split exact; so, in particular, 
${\mathbb M^{\epsilon}}$ is a resolution of 
$$\begin{cases}\overline{R},&
\text{if $\epsilon=\lceil\frac{\delta}2\rceil$, and}\\  
\widetilde{R},&\text{if $\epsilon= \lceil\frac {\delta-1}2\rceil$}.\end{cases}$$
 The most recent statement is assertion (\ref{main.a}) of Theorem~\ref{main}. 
The section  concludes with the proof of Theorem~\ref{main} and item (\ref{main.c.iii}) of Theorem~\ref{main.c}.  Theorems~\ref{main} and \ref{main.c} are the main results of the paper.  
 
Ideas from \ref{87Not1} are used in the definition of $\xi^\epsilon$.

\begin{definition}\label{def-xi} Adopt Data~{\rm{\ref{87data6}}} and 
{\rm{\ref{unm-data}}}. Recall  the double complexes ${\mathbb T^{\epsilon}}$ and ${\mathbb B}$ from Definition~{\rm\ref{87.11}}. Define the $R$-module homomorphism ${\xi^{\epsilon}:\operatorname{Tot}({\mathbb T^{\epsilon}})\to\operatorname{Tot}({\mathbb B})}$ as follows. 
\begin{enumerate}[\rm(a)]
\item If $N=0$ and $\epsilon=\frac{\delta-1}2$, then $$\textstyle\xi^{\epsilon}_N:
V_{0,\epsilon}=\mathbb T^{\epsilon}_0\to \mathbb B_0=\textstyle \bigwedge^{\goth f}F\otimes \bigwedge^{\goth g}G$$ is given by
$\textstyle\xi^{\epsilon}_0(\gamma_{\epsilon})=c(\gamma_{\epsilon})$, where 
$c$ is the map of (\ref{c-2.3}).
\item If $1\le N$ and 
 $I$ and $J$ are parameters  with
 $I+2J-\delta+1=N$, $0\le I$,  and $\epsilon\le J$, then 
\begin{align}&\xi^{\epsilon}_N(f_I\otimes \gamma_J\in V_{I,J}\subseteq \operatorname{Tot}({\mathbb T^\epsilon})_N)
\notag\\
=& 
\textstyle
(-1)^N\sum\limits_{\gf{i+2j=I+2J}{\delta\le i+j}}(-1)^{I+j}\binom{J-1-j}{J-\epsilon} \sum\limits_{m\in \binom{Y}{J-j}} [D(\mu)](m^*)\wedge f_I\otimes m(\gamma_J)
\in V_{i,j}\subseteq \operatorname{Tot}({\mathbb B})_{N}.
\notag
\end{align}
\end{enumerate}\end{definition}

\begin{lemma}\label{MOC} Adopt Data~{\rm{\ref{87data6}}} and 
{\rm{\ref{unm-data}}}. Then the $R$-module homomorphism $${\xi^{\epsilon}:\operatorname{Tot}({\mathbb T^{\epsilon}})\to\operatorname{Tot}({\mathbb B})}$$ of Definition~{\rm\ref{def-xi}} is a map of complexes.\end{lemma}
\begin{proof}We verify that each square
\begin{equation}\label{square}\xymatrix{\operatorname{Tot}({\mathbb T^{\epsilon}})_N\ar[r]^{d^{\mathbb T^{\epsilon}}}\ar[d]^{\xi^{\epsilon}_N}&\operatorname{Tot}({\mathbb T^{\epsilon}})_{N-1}\ar[d]^{\xi^{\epsilon}_{N-1}}\\
\operatorname{Tot}({\mathbb B})_N\ar[r]^{d^{\mathbb B}}&\operatorname{Tot}({\mathbb B})_{N-1}}\end{equation}commutes. The differential $d^{{\mathbb B}}$ is special when $N=1$; so we treat $N=1$ first. In this case, 
$$
\operatorname{Tot}({\mathbb T^{\epsilon}})_{N}=\begin{cases} 
V_{1,\frac{\delta-1}2},&\text{if $\epsilon=\frac{\delta-1}2$},\\
V_{0,\frac\delta2},&\text{if $\epsilon=\frac{\delta}2$, and}\\
0,&\text{if $\epsilon=\frac{\delta+1}2$},\end{cases}\quad\text{and}\quad 
\operatorname{Tot}({\mathbb T^{\epsilon}})_{N-1}=\begin{cases}
V_{0,\frac{\delta-1}2},&\text{if $\epsilon=\frac{\delta-1}2$},\\
0,&\text{if $\epsilon=\frac{\delta}2$, and}\\
0,&\text{if $\epsilon=\frac{\delta+1}2$,}\end{cases}
$$
If $\epsilon=\frac{\delta-1}2$ and $f_1\otimes \gamma_{\epsilon}\in V_{1,\epsilon}$, then $$(d^{{\mathbb B}}\circ\xi^{\epsilon}_1)(f_1\otimes \gamma_{\epsilon})=
 d^{{\mathbb B}}\big(f_1\wedge [D(\mu)](\gamma_{\epsilon})\big)=
  f_1\wedge [D(\mu)](\gamma_{\epsilon})\wedge (\textstyle\bigwedge^{\goth g}\Psi)(\omega_{G^*})\otimes \omega_{G}
$$
and 
\begingroup\allowdisplaybreaks
\begin{align}
&\textstyle \phantom{=}(\xi^{\epsilon}_0\circ d^{{\mathbb T^{\epsilon}}})(f_1\otimes \gamma_{\epsilon})
\notag\\&\textstyle =c(\tau(f_1)\cdot \gamma_{\epsilon})\notag\\
&\textstyle =\tau(f_1)\cdot  [D(\mu)](\int_{X_1} \gamma_{\epsilon})
\wedge(\bigwedge^{\goth g-1}\Psi)(\omega_{{G_+}^*})
\otimes \omega_{G_+}\wedge   Y_1\notag\\
&\textstyle = f_1\wedge \tau\Big( [D(\mu)](\int_{X_1} \gamma_{\epsilon})\wedge (\bigwedge^{\goth g-1}\Psi)(\omega_{{G_+}^*}) \Big)\otimes \omega_{G_+}\wedge Y_1,\notag\\\intertext{because $\bigwedge^\bullet F$ is a DG-module over $\bigwedge^\bullet F^*$,}&\textstyle = f_1\wedge\tau\Big([D(\mu)](\int_{X_1} \gamma_{\epsilon})\Big)\wedge (\bigwedge^{\goth g-1}\Psi)(\omega_{{G_+}^*}) \otimes \omega_{G_+}\wedge   Y_1,\notag\\\intertext{because $\tau\circ \Psi=0$,}
&\textstyle = f_1\wedge \Psi(X_1)\wedge \Big([D(\mu)]( \gamma_{\epsilon})\Big)\wedge (\bigwedge^{\goth g-1}\Psi)(\omega_{{G_+}^*})\otimes \omega_{G_+}\wedge Y_1,\notag\\\intertext{because of  (\ref{trick2}), (\ref{NT}), and the fact that $(\bigwedge^\goth g\Psi)(\bigwedge^\goth g {G_+}^*)=0$,}
&\textstyle = f_1\wedge  \Big([D(\mu)]( \gamma_{\epsilon})\Big)\wedge \Psi(X_1)\wedge (\bigwedge^{\goth g-1}\Psi)(\omega_{{G_+}^*})\otimes \omega_{G_+}\wedge Y_1\notag\\
&\textstyle = f_1\wedge  \Big([D(\mu)]( \gamma_{\epsilon})\Big)\wedge  (\bigwedge^{\goth g}\Psi)(\omega_{G^*})\otimes \omega_{G},\notag\\
\intertext{because $X_1\otimes \omega_{{G_+}^*}$ and $\omega_{G_+}\wedge Y_1$ are dual bases for $\bigwedge^\goth g G^*$ and $\bigwedge^\goth g G$, see (\ref{omega}),} 
&\textstyle =(d^{{\mathbb B}}\circ\xi^{\epsilon}_1)(f_1\otimes \gamma_{\epsilon}).\notag
\end{align}\endgroup
The calculation we just completed shows that $\operatorname{im}(\tau)\cdot \operatorname{im} (c)\subseteq \operatorname{im} (d^{\mathbb B}_1)$, in other words,
\begin{equation}\label{much-promised} I_1(\tau)\cdot \goth C\subseteq I_{\goth g}(\Psi).\end{equation}

If $\epsilon=\frac{\delta}2$ and $\gamma_{\frac \delta2}\in V_{0,\frac\delta2}=\operatorname{Tot}({\mathbb T^{\epsilon}})_{1}$, then
$$(d^{{\mathbb B}}\circ\xi^{\epsilon}_1)(\gamma_{\frac \delta2})=-d^{{\mathbb B}}([D(\mu)](\gamma_{\frac \delta2}))
=-[D(\mu)](\gamma_{\frac \delta2})\wedge ({\textstyle\bigwedge^{\goth g}}\Psi)(\omega_{G^*})\otimes \omega_G.$$Observe that $[D(\mu)](\gamma_{\frac \delta2})\wedge ({\textstyle\bigwedge^{\goth g}}\Psi)(\omega_{G^*})$ is an element of $\tau(\bigwedge^{\goth f+1}F)=0$. Indeed, if $\omega_{G^*}$ is $X_1\wedge\dots \wedge X_{\goth g}$ and 
$\gamma_{\frac \delta2}=X_1^{(a_1)}\dots X_{\goth g}^{(a_{\goth g})}$, then 
\begin{align}&[D(\mu)](\gamma_{\frac \delta2})\wedge ({\textstyle\bigwedge^{\goth g}}\Psi)(\omega_{G^*})\notag\\=&\tau\left((\mu (X_1))^{(a_1+1)}\wedge (\mu (X_2))^{(a_2)}\wedge \dots \wedge(\mu (X_{\goth g}))^{(a_{\goth g})}\wedge \tau(\mu (X_2))\wedge\dots\wedge \tau(\mu (X_{\goth g}))\right)\in 
\tau({\textstyle\bigwedge^{\goth f+1}}F).\notag\end{align} This completes the proof that (\ref{square}) commutes when $N=1$.

Now we consider $2\le N$. Fix $I$ and $J$ with $I+2J-\delta+1=N$, $0\le I$, and $\epsilon\le J$ 
and fix $f_I\otimes \gamma_J\in V_{I,J}\subseteq \operatorname{Tot}({\mathbb T^{\epsilon}})_N$. Observe that 
$$(\xi^{\epsilon}_{N-1}\circ d^{{\mathbb T^{\epsilon}}})(f_I\otimes \gamma_J)=S_1+S_2\quad\text{and}\quad (d^{{\mathbb B}}\circ \xi^{\epsilon}_N)(f_I\otimes \gamma_J)=S_3+S_4$$ with

\begingroup\allowdisplaybreaks \begin{align}
S_1&=\textstyle \xi^{\epsilon}_{N-1}\left(\chi(\epsilon\le J-1)\sum\limits_{\ell=1}^{\goth g}\Psi(X_\ell)\wedge f_I\otimes Y_\ell(\gamma_J)\in V_{I+1,J-1}\subseteq \operatorname{Tot}({\mathbb T^{\epsilon}})_{N-1}\right),\notag\\
S_2&=\textstyle\xi^{\epsilon}_{N-1}\left(\tau(f_I)\otimes \gamma_J\in V_{I-1,J}\subseteq \operatorname{Tot}({\mathbb T^{\epsilon}})_{N-1}\right),\notag\\
S_3&=(-1)^N\textstyle\sum\limits_{\gf{i+2j=I+2J}{\delta\le i+j}}(-1)^{I+j}\binom{J-1-j}{J-\epsilon} \sum\limits_{m\in \binom{Y}{J-j}} \sum\limits_{\ell=1}^{\goth g}\Psi(X_\ell)\wedge [D(\mu)](m^*)\wedge f_I\otimes Y_\ell(m(\gamma_J))\notag\\&\hskip1in\in V_{i+1,j-1}\subseteq \operatorname{Tot}({\mathbb B})_{N-1}, \notag\\
S_4&=(-1)^N\textstyle\sum\limits_{\gf{i+2j=I+2J}{\delta\le i+j}}(-1)^{I+j}\binom{J-1-j}{J-\epsilon} \sum\limits_{m\in \binom{Y}{J-j}} \chi(\delta\le i+j-1)\tau ([D(\mu)](m^*)\wedge f_I)\otimes m(\gamma_J)\notag\\&\hskip1in\in V_{i-1,j}\subseteq \operatorname{Tot}({\mathbb B})_{N-1}. \notag
\end{align}\endgroup

\noindent We compute 

$$S_1=\begin{cases}\chi(\epsilon\le J-1)\sum\limits_{\ell=1}^{\goth g} 
(-1)^{N-1}\sum\limits_{\gf{i+2j=I+2J-1}{\delta\le i+j}}(-1)^{I+1+j}\binom{J-2-j}{J-1-\epsilon}\\\hskip.5in \sum\limits_{m\in \binom{Y}{J-1-j}} [D(\mu)](m^*)\wedge \Psi(X_\ell)\wedge f_I\otimes m(Y_\ell(\gamma_J))
\in V_{i,j}\subseteq \operatorname{Tot}({\mathbb B})_{N-1}.\end{cases}
$$
The binomial coefficient renders the factor $\chi(\epsilon\le J-1)$ redundant. Apply the trick of (\ref{trick1}), followed by (\ref{trick2}), to see that 

\begingroup\allowdisplaybreaks\begin{align}
S_1&=\begin{cases}\sum\limits_{\ell=1}^{\goth g} 
(-1)^{N-1}\sum\limits_{\gf{i+2j=I+2J-1}{\delta\le i+j}}(-1)^{I+1+j}\binom{J-2-j}{J-1-\epsilon}\notag\\ \hskip.5in \sum\limits_{m\in \binom{Y}{J-j}} [D(\mu)](Y_\ell(m^*))\wedge \Psi(X_\ell)\wedge f_I\otimes m(\gamma_J)
\in V_{i,j}\subseteq \operatorname{Tot}({\mathbb B})_{N-1}\end{cases}\notag\\
 &=\begin{cases} 
(-1)^{N-1}\sum\limits_{\gf{i+2j=I+2J-1}{\delta\le i+j}}(-1)^{I+1+j}\binom{J-2-j}{J-1-\epsilon} \sum\limits_{m\in \binom{Y}{J-j}}
 \tau([D(\mu)](m^*))\wedge f_I\otimes m(\gamma_J)\\\hskip1in
\in V_{i,j}\subseteq \operatorname{Tot}({\mathbb B})_{N-1}.\end{cases}\notag
\end{align}\endgroup

\noindent It is immediate that $$ S_2=\begin{cases}
(-1)^{N-1}\sum\limits_{\gf{i+2j=I-1+2J}{\delta\le i+j}}(-1)^{I-1+j}\binom{J-1-j}{J-\epsilon} \sum\limits_{m\in \binom{Y}{J-j}} [D(\mu)](m^*)\wedge \tau(f_I)\otimes m(\gamma_J)\\\hskip.5in 
\in V_{i,j}\subseteq \operatorname{Tot}({\mathbb B})_{N-1}\end{cases}
$$
Our approach to $S_3$ is similar to our approach to $S_1$. We obtain
\begingroup\allowdisplaybreaks \begin{align}
S_3&=\begin{cases}(-1)^N\sum\limits_{\gf{i+2j=I+2J}{\delta\le i+j}}(-1)^{I+j}\binom{J-1-j}{J-\epsilon} \sum\limits_{m\in \binom{Y}{J-j+1}} \sum\limits_{\ell=1}^{\goth g}\Psi(X_\ell)\wedge [D(\mu)](Y_\ell(m^*))\wedge f_I\otimes m(\gamma_J)\\\hskip.5in \in V_{i+1,j-1}\subseteq \operatorname{Tot}({\mathbb B})_{N-1}\end{cases}
\notag\\
&=\begin{cases}(-1)^N\sum\limits_{\gf{i+2j=I+2J}{\delta\le i+j}}(-1)^{I+j}\binom{J-1-j}{J-\epsilon} \sum\limits_{m\in \binom{Y}{J-j+1}}  \tau([D(\mu)](m^*))\wedge f_I\otimes m(\gamma_J)\\\hskip.5in \in V_{i+1,j-1}\subseteq \operatorname{Tot}({\mathbb B})_{N-1}\end{cases}
\notag
\end{align}\endgroup 
Replace $i$ with $i-1$ and $j$ with $j+1$ and conclude that

$$
S_3=\begin{cases}(-1)^N\sum\limits_{\gf{i+2j+1=I+2J}{\delta\le i+j}}(-1)^{I+j+1}\binom{J-2-j}{J-\epsilon} \sum\limits_{m\in 
\binom{Y}{J-j}}  \tau([D(\mu)](m^*))\wedge f_I\otimes m(\gamma_J)\\\hskip.5in\in V_{i,j}\subseteq \operatorname{Tot}({\mathbb B})_{N-1}.\end{cases}$$

Replace $i$ with $i+1$ in $S_4$ to obtain
\begingroup\allowdisplaybreaks\begin{align}
S_4&=\begin{cases}(-1)^N\sum\limits_{\gf{i+1+2j=I+2J}{\delta\le i+1+j}}(-1)^{I+j}\binom{J-1-j}{J-\epsilon} \sum\limits_{m\in \binom{Y}{J-j}} \chi(\delta\le i+j)\tau ([D(\mu)](m^*)\wedge f_I)\otimes m(\gamma_J)\notag\\\hskip1in
\in V_{i,j}\subseteq \operatorname{Tot}({\mathbb B})_{N-1}\end{cases}\notag\\
&=\begin{cases}(-1)^N\sum\limits_{\gf{i+1+2j=I+2J}{\delta\le i+j}}(-1)^{I+j}\binom{J-1-j}{J-\epsilon} \sum\limits_{m\in \binom{Y}{J-j}} 
\tau ([D(\mu)](m^*)\wedge f_I)\otimes m(\gamma_J)\notag\\\hskip1in
\in V_{i,j}\subseteq \operatorname{Tot}({\mathbb B})_{N-1}.\end{cases}\notag\end{align}\endgroup
The action of $F^*$ on $\bigwedge^\bullet F$ now gives $S_4=S_4'+S_4''$ with
\begingroup\allowdisplaybreaks\begin{align}
S_4'&=\begin{cases}(-1)^N\sum\limits_{\gf{i+1+2j=I+2J}{\delta\le i+j}}(-1)^{I+j}\binom{J-1-j}{J-\epsilon} \sum\limits_{m\in \binom{Y}{J-j}} 
\tau ([D(\mu)](m^*))\wedge f_I\otimes m(\gamma_J)\\\hskip1in
\in V_{i,j}\subseteq \operatorname{Tot}({\mathbb B})_{N-1} \text{ and}\end{cases}\notag\\
S_4''&=\begin{cases}(-1)^N\sum\limits_{\gf{i+1+2j=I+2J}{\delta\le i+j}}(-1)^{I+j}\binom{J-1-j}{J-\epsilon} \sum\limits_{m\in \binom{Y}{J-j}} 
[D(\mu)](m^*)\wedge \tau (f_I)\otimes m(\gamma_J)\\\hskip1in
\in V_{i,j}\subseteq \operatorname{Tot}({\mathbb B})_{N-1}.\end{cases}\notag\end{align}\endgroup

\noindent We see that
$S_1=S_3+S_4'$ 
 and $S_2=S_4''$.
The square (\ref{square}) commutes for all $N$ and the proof is complete.
\end{proof}

Let ${\mathbb L^\epsilon}$ be the mapping cone of $\xi^{\epsilon}$; so $d^{{\mathbb L^\epsilon}}:{\mathbb L^\epsilon}_N\to {\mathbb L^\epsilon}_{N-1}$ is
$$\begin{matrix} \operatorname{Tot}({\mathbb T^\epsilon})[-1]_{N}\\\oplus\\ \operatorname{Tot}({\mathbb B})_{N}\end{matrix}
\xrightarrow{\ \bmatrix d_{N-1}^{\mathbb T^\epsilon}&0\\ \\(-1)^{N-1}\xi^{\epsilon}_{N-1}&d_N^{{\mathbb B}}\endbmatrix\ }\begin{matrix} \operatorname{Tot}({\mathbb T^\epsilon})[-1]_{N-1}\\\oplus\\ \operatorname{Tot}({\mathbb B})_{N-1}.\end{matrix}$$ We write $V_{I,J}^{{\mathbb T^\epsilon}}$ for the image of $V_{I,J}$ in $\operatorname{Tot}({\mathbb T^\epsilon})[-1]\subseteq {\mathbb L^\epsilon}$ and 
$V_{i,j}^{{\mathbb B}}$ for the image of $V_{i,j}$ in $\operatorname{Tot}({\mathbb B})\subseteq {\mathbb L^\epsilon}$. So, in particular,
\begin{equation}\label{new}V_{I,J}^{{\mathbb T^\epsilon}}\subseteq {\mathbb L^\epsilon}_{I+2J-\delta+2}\quad\text{and}\quad V_{i,j}^{{\mathbb B}}\subseteq {\mathbb L^\epsilon}_{i+2j-\delta+1}.\end{equation} It is useful to gather the definition of ${\mathbb L^\epsilon}$ in one place. 

\begin{definition}\label{huge-def}Adopt Data~{\rm{\ref{87data6}}} and 
{\rm{\ref{unm-data}}}. The complex ${\mathbb L^{\epsilon}}$ is described as follows:
\begin{enumerate}[\rm(a)]
\item As a graded $R$-module $${\mathbb L^\epsilon}=(\bigoplus\limits_{\epsilon\le J} V_{I,J}^{{\mathbb T^\epsilon}}) \oplus (\bigoplus\limits_{\delta\le i+j} V_{i,j}^{{\mathbb B}}) \oplus {\textstyle(\bigwedge^{\goth f}F\otimes \bigwedge^{\goth g}G)},$$
with \begin{enumerate}[\rm(i)]
\item $V_{I,J}^{{\mathbb T^{\epsilon}}}$ in position $I+2J-\delta+2$,
\item $V_{i,j}^{{\mathbb B}}$ in position $i+2j-\delta+1$, and 
\item $\textstyle(\bigwedge^{\goth f}F\otimes \bigwedge^{\goth g}G)$ in position $0$.
 \end{enumerate}

\item The $R$-module homomorphisms  of ${\mathbb L^{\epsilon}}$ are  described below.
\begin{enumerate}[\rm(i)] \item If $I+2J-\delta+2=N$, $2\le N$, $0\le I$, and $\epsilon\le J$, 
then $V_{I,J}^{{\mathbb T^{\epsilon}}}$ is a summand of ${\mathbb L^{\epsilon}}_N$ and $$d(f_I\otimes \gamma_J\in V_{I,J}^{{\mathbb T^\epsilon}})
=\begin{cases} 
\phantom{+}\chi( \epsilon\le J-1)\sum\limits_{\ell=1}^{\goth g}\Psi(X_\ell)\wedge f_I\otimes Y_\ell(\gamma_J)\in V_{I+1,J-1}^{{\mathbb T^\epsilon}}\subseteq {\mathbb L^\epsilon}_{N-1}\\
+\tau(f_I)\otimes \gamma_J\in V_{I-1,J}^{{\mathbb T^\epsilon}}\subseteq {\mathbb L^\epsilon}_{N-1}\\
+\sum\limits_{\gf{i+2j=I+2J}{\delta\le i+j}}(-1)^{I+j}\binom{J-1-j}{J-\epsilon} \sum\limits_{m\in \binom{Y}{J-j}} [D(\mu)](m^*)\wedge f_I\otimes m(\gamma_J)\\\hskip1in\in V_{i,j}^{{\mathbb B}}\subseteq {\mathbb L^\epsilon}_{N-1}.
\end{cases}$$

\item If $i+2j-\delta+1=N$, $2\le N$, $\delta\le i+j$, and $0\le i,j$, 
then $V_{i,j}^{{\mathbb B}}$ is a summand of 
${\mathbb L^\epsilon}_N$ and 
 $$d(f_i\otimes \gamma_j\in V_{i,j}^{{\mathbb B}}) 
=\begin{cases} 
\phantom{+}\sum\limits_{\ell=1}^{\goth g}\Psi(X_\ell)\wedge f_i\otimes Y_\ell(\gamma_j)\in V_{i+1,j-1}^{{\mathbb B}}\subseteq {\mathbb L^\epsilon}_{N-1}\\
+\chi(\delta\le i+j-1)\tau (f_i)\otimes \gamma_j\in V_{i-1,j}^{{\mathbb B}}\subseteq {\mathbb L^\epsilon}_{N-1}.
\end{cases}$$
\item The $R$-module homomorphism  \begin{align} \textstyle d:{\mathbb L^\epsilon}_1& \textstyle =V_{\delta,0}^{{\mathbb B}}\oplus \chi(\textstyle\epsilon=\frac{\delta-1}2) V_{0,\epsilon}^{\mathbb T^{\epsilon}} \to{\mathbb L^\epsilon}_0={\textstyle(\bigwedge^{\goth f}F\otimes \bigwedge^{\goth g}G)}\notag\\ \intertext{is}
\textstyle d(f_\delta\in V_{\delta,0}^{{\mathbb B}})&=\textstyle f_\delta\wedge (\bigwedge^{\goth g}\Psi)(\omega_{G^*})\otimes \omega_G\notag\end{align}and, if $\epsilon= \frac{\delta-1}2$, then
$d(\gamma_\epsilon\in V_{0,\epsilon}^{\mathbb T^{\epsilon}}) =
c(\gamma_\epsilon)$  for $c$ as defined in {\rm(\ref{c-2.3})}.
\end{enumerate}
\end{enumerate}
\end{definition}
 \begin{remark}\label{rmk5} A quick comparison of Definitions \ref{doodles'} and \ref{huge-def} shows that ${\mathbb M^\epsilon}$ is a complex, and indeed a subcomplex of ${\mathbb L^\epsilon}$. The only tricky part of this assertion is already contained in Remark~\ref{num-info-Med}.\ref{num-info-Med.i}.\end{remark}
 
\begin{lemma}\label{real-10.8} Adopt Data~{\rm{\ref{87data6}}} and 
{\rm{\ref{unm-data}}}. Retain the complexes
${\mathbb M^\epsilon}\subseteq {\mathbb L^\epsilon}$ of Remark~{\rm\ref{rmk5}}.
Then the complex ${\mathbb L^\epsilon}/{\mathbb M^\epsilon}$ is split exact. So, in particular, $$\operatorname{H}_N({\mathbb L^\epsilon})=\operatorname{H}_N({\mathbb M^\epsilon})$$ for all $N$.\end{lemma}
\begin{proof}The complex ${\mathbb L^\epsilon}/{\mathbb M^\epsilon}$, which  is 
$$\bigoplus\limits_{\gf{\epsilon\le J}{\delta\le I+J}} V_{I,J}^{{\mathbb T^{\epsilon}}} \oplus \bigoplus\limits_{\gf{ \epsilon \le j}{\delta\le i+j}} V_{i,j}^{{\mathbb B}},$$
is the mapping cone of 
\begin{equation}\xymatrix{\bigoplus\limits_{\gf{\epsilon\le J}{\delta\le I+J}} V_{I,J}^{{\mathbb T^\epsilon}}\ar[d]^{\xi^{\epsilon\, \prime}}\\  \bigoplus\limits_{\gf{ \epsilon \le j}{\delta\le i+j}} V_{i,j}^{{\mathbb B}},}\notag\end{equation}where $\xi^{\epsilon\, \prime}$ is the map induced by $\xi^{\epsilon}$. It suffices to show that $\xi^{\epsilon\, \prime}$ is an isomorphism. We show that if $I+2J=i+2j$, then 
 $\xi^{\epsilon\, \prime}:V_{I,J}^{{\mathbb T^\epsilon}}\to V_{i,j}^{{\mathbb B}}$ is an isomorphism if $J=j$ and is the zero map if $J<j$. The second assertion is obvious because the sum 
$$\sum\limits_{m\in\binom{Y}{J-j}}$$
over the empty set (there are no monomials of negative degree) is zero.
The first assertion is also obvious because the component  $\xi^{\epsilon\, \prime}:V_{I,J}^{{\mathbb T^\epsilon}}\to V_{I,J}^{{\mathbb B}}$ of $\xi^{\epsilon\, \prime}$ 
sends $f_I\otimes \gamma_J\in V_{I,J}^{{\mathbb T^\epsilon}}$ to
$$\textstyle 
(-1)^N(-1)^{I+J}\binom{-1}{J-\epsilon} \sum\limits_{m\in \binom{Y}{0}} [D(\mu)](m^*)\wedge f_I\otimes m(\gamma_J)=
(-1)^{N+I+\epsilon} f_I\otimes \gamma_J
\in V_{I,J}^{{\mathbb B}},
$$because $\binom{-1}{b}=(-1)^b$ for all non-negative integers $b$, where $N=I+2J-\delta+1$.
 Thus, $$\xymatrix{\bigoplus\limits_{\gf{\gf{\epsilon\le J}{\delta\le I+J}}{I+2J=N}} V_{I,J}^{{\mathbb T^\epsilon}}\ar[d]^{\xi^{\epsilon\, \prime}_N}\\  \bigoplus\limits_{\gf{\gf{ \epsilon\le j}{\delta\le i+j}}{i+2j=N}} V_{i,j}^{{\mathbb B}}}$$ is represented by a square triangular matrix  with isomorphisms on the main diagonal.
\end{proof}

\begin{lemma}\label{resol}Adopt Data~{\rm{\ref{87data6}}} and 
{\rm{\ref{unm-data}}}. Recall the map of complexes $\xi^{\epsilon}:\operatorname{Tot}({\mathbb T^\epsilon})\to\operatorname{Tot}({\mathbb B})$ from Lemma~{\rm\ref{MOC}} and the complexes ${\mathbb L^\epsilon}$ and ${\mathbb M^\epsilon}$ 
from Lemma~{\rm\ref{real-10.8}}.
\begin{enumerate}[\rm(a)]
\item\label{resol.a} If $N$ is a positive integer with $\delta+N-1$ even and $\gamma$ is an element of $D_{\frac{\delta+N-1}2}(G^*)$, $$z^\epsilon_\gamma\in \operatorname{Tot}({\mathbb T^\epsilon})_N\quad \text{and}\quad \zeta_\gamma\in\operatorname{Tot}({\mathbb B})_N$$ are the cycles of Lemma~{\rm\ref{8z-gamma}}, then $\xi^{\epsilon}(z^\epsilon_\gamma)=(-1)^{\epsilon+N}\zeta_{\gamma}$.
\item\label{resol.b} If Hypotheses~{\rm\ref{87hyp6}} are in effect, then ${\mathbb L^\epsilon}$ and ${\mathbb M^\epsilon}$ are resolutions
of $$\begin{cases}\overline{R},&
\text{if $\epsilon=\frac{\delta}2$ or $\epsilon=\frac{\delta+1}2$, and}\\  
\widetilde{R},&\text{if $\epsilon= \frac {\delta-1}2$}.\end{cases}$$
 by free $R$-modules.
\end{enumerate}
\end{lemma}
\begin{proof} (\ref{resol.a}) 
We compute $$\xi^{\epsilon}(z^\epsilon_{\gamma})=
\sum\limits_{J=\epsilon}^{\frac{\delta+N-1}2}\sum\limits_{m_1\in \binom{Y}{\frac{\delta+N-1}2-J}}(-1)^{J}
\xi^{\epsilon}\left([D(\mu)](m_1^*)\otimes m_1(\gamma)\in V_{\delta+N-1-2J,J}\subseteq \operatorname{Tot}({\mathbb T^\epsilon})_N\right)$$
$$=\begin{cases}
\sum\limits_{J=\epsilon}^{\frac{\delta+N-1}2}\sum\limits_{m_1\in \binom{Y}{\frac{\delta+N-1}2-J}}\sum\limits_{\gf{i+2j=\delta+N-1}{\delta\le i+j}} \sum\limits_{m\in \binom{Y}{J-j}}(-1)^{J+N+j}\binom{J-1-j}{J-\epsilon}\\ \hskip1in\times[D(\mu)](m^*)\wedge [D(\mu)](m_1^*)\otimes m(m_1(\gamma))
\in V_{i,j}\subseteq \operatorname{Tot}({\mathbb B})_{N}
\end{cases}$$
$$=\begin{cases}
\sum\limits_{J=\epsilon}^{\frac{\delta+N-1}2}\sum\limits_{\gf{i+2j=\delta+N-1}{\delta\le i+j}} 
\sum\limits_{m_1\in \binom{Y}{\frac{\delta+N-1}2-j}}
\sum\limits_{m\in \binom{Y}{J-j}}(-1)^{J+N+j}\binom{J-1-j}{J-\epsilon}\\ \hskip1in\times[D(\mu)](m^*)\wedge [D(\mu)](m(m_1^*))\otimes m_1(\gamma)
\in V_{i,j}\subseteq \operatorname{Tot}({\mathbb B})_{N}.
\end{cases}$$
The parameters $i$ and $j$ satisfy $i+2j=\delta+N-1$ and $\delta\le i+j$ if and only if they satisfy $i=\delta+N-1-2j$ and $j\le N-1$.
It follows that
\begingroup\allowdisplaybreaks \begin{align}
\xi^{\epsilon}(z^\epsilon_{\gamma})&=\begin{cases}\sum\limits_{j=0}^{N-1} \sum\limits_{m_1\in \binom{Y}{\frac{\delta+N-1}2-j}}
(-1)^{j+N}
\sum\limits_{J=\epsilon}^{\frac{\delta+N-1}2}(-1)^{J}
\binom{J-1-j}{J-\epsilon} \\ \hskip.2in\times\sum\limits_{m\in \binom{Y}{J-j}} [D(\mu)](m^*)\wedge [D(\mu)](m(m_1^*))\otimes m_1(\gamma)\in V_{
\delta+N-1-2j,j}\subseteq \operatorname{Tot}({\mathbb B})_{N}\end{cases}
\notag\\
&=\begin{cases}\sum\limits_{j=0}^{N-1} \sum\limits_{m_1\in \binom{Y}{\frac{\delta+N-1}2-j}}
(-1)^{j+N}
\sum\limits_{J=\epsilon}^{\frac{\delta+N-1}2}(-1)^{J}
\binom{J-1-j}{J-\epsilon} \\ 
\hskip.2in\times [D(\mu)](\sum\limits_{m\in \binom{Y}{J-j}}m^*\cdot m(m_1^*))\otimes m_1(\gamma)\in V_{
\delta+N-1-2j,j}\subseteq \operatorname{Tot}({\mathbb B})_{N}.\end{cases}
\notag\end{align}\endgroup Apply Lemma~\ref{Gamma} to see that 
\begin{equation}\label{context} \xi^{\epsilon}(z^\epsilon_{\gamma})=\begin{cases}\sum\limits_{j=0}^{N-1} \sum\limits_{m_1\in \binom{Y}{\frac{\delta+N-1}2-j}}
(-1)^{j+N}
\sum\limits_{J=\epsilon}^{\frac{\delta+N-1}2}(-1)^{J}
\binom{J-1-j}{J-\epsilon}\binom{\frac{\delta+N-1}2-j}{J-j} \\ 
\hskip.3in\times [D(\mu)](m_1^*)\otimes m_1(\gamma)\in V_{
\delta+N-1-2j,j}\subseteq \operatorname{Tot}({\mathbb B})_{N}.\end{cases}
\end{equation}
We simplify
$$\textstyle E= \sum\limits_{J=\epsilon}^{\frac{\delta+N-1}2}(-1)^{J}
\binom{J-1-j}{J-\epsilon}\binom{\frac{\delta+N-1}2-j}{J-j}$$ in the context of (\ref{context}).
Let $k=J-\epsilon$. The expression is
$$E=(-1)^{\epsilon}\textstyle \sum\limits_{k=0}^{\frac{\delta+N-1}2-\epsilon}(-1)^{k}
\binom{k+\epsilon-1-j}{k}\binom{\frac{\delta+N-1}2-j}{k+\epsilon-j}.$$
We need only consider $0\le\frac{\delta+N-1}2-j$ (otherwise, the monomial $m_1$ of (\ref{context}) has negative degree and  (\ref{context}) is zero); so
$$\binom{\frac{\delta+N-1}2-j}{k+\epsilon-j}=
\binom{\frac{\delta+N-1}2-j}
{\frac{\delta+N-1}2-\epsilon-k}.
$$It is always true for all integers $a$ and $b$ that $\binom ba=(-1)^a\binom{a-b-1}{a}$; so in particular, 
$$(-1)^{k}
\binom{k+\epsilon-1-j}{k}= \binom {j-\epsilon}k,$$and 
\begin{align}E&=(-1)^{\epsilon}\textstyle \sum\limits_{k=0}^{\frac{\delta+N-1}2-\epsilon}\binom {j-\epsilon}k
\binom{\frac{\delta+N-1}2-j}
{\frac{\delta+N-1}2-\epsilon-k}\notag\\
&=(-1)^{\epsilon}\textstyle \sum\limits_{k\in\mathbb Z}\binom {j-\epsilon}k
\binom{\frac{\delta+N-1}2-j}
{\frac{\delta+N-1}2-\epsilon-k}
\notag\end{align}
Now use the identity 
$$\sum\limits_{k\in \mathbb Z}\binom{b}{c-k}\binom ak=\binom{a+b}c,$$see, for example, \cite[Obs.~1.2.i]{K95},
with $a= j-\epsilon$, $b=\frac{\delta+N-1}2-j$, and $c=\frac{\delta+N-1}2-\epsilon$ 
to learn that  
$$E
=(-1)^{\epsilon}
\binom{\frac{\delta+N-1}2-\epsilon}{\frac{\delta+N-1}2-\epsilon}.
$$The hypotheses that $N$ is positive and $\delta+N-1$ is even force 
\begin{equation}\label{non-neg}\textstyle\frac{\delta+N-1}2-\epsilon\end{equation} to be non-negative. 
(If  $\epsilon=\frac{\delta-1}2$, then (\ref{non-neg}) is at least $\frac N2$. If $\epsilon=\frac{\delta}2$, then  (\ref{non-neg}) is at least $\frac{N-1}2$. 
If  $\epsilon=\frac{\delta+1}2$, then 
$2\le N$ and  (\ref{non-neg}) is at least $\frac{N-2}2$.) Thus $E=(-1)^{\epsilon}$; and
\begin{align} &\xi^{\epsilon}(z^\epsilon_{\gamma})\notag \\=&(-1)^{N+\epsilon}\sum\limits_{j=0}^{N-1} \sum\limits_{m_1\in \binom{Y}{\frac{\delta+N-1}2-j}}
(-1)^{j}
[D(\mu)](m_1^*)\otimes m_1(\gamma)\in V_{
\delta+N-1-2j,j}\subseteq \operatorname{Tot}({\mathbb B})_{N}\notag \\=&(-1)^{N+\epsilon}\zeta_{\gamma}.\notag
\end{align}This completes the proof of (\ref{resol.a}).

\medskip\noindent(\ref{resol.b}) Assertion (\ref{resol.a}), together with Lemmas~\ref{87.15} and \ref{87.15-bot}, shows that $\xi^{\epsilon}$ induces an isomorphism $\operatorname{H}_N(\operatorname{Tot}({\mathbb T^\epsilon}))\to \operatorname{H}_N(\operatorname{Tot}({\mathbb B}))$ for all positive $N$. The complex ${\mathbb L^\epsilon}$ is the mapping cone of $\xi^{\epsilon}$; so, the long exact sequence of homology associated to a mapping cone of complexes guarantees that $\operatorname{H}_N({\mathbb L^\epsilon})=0$ for all positive $N$. In other words, ${\mathbb L^\epsilon}$ is a resolution of $\operatorname{H}_0({\mathbb L^\epsilon})$. Apply Lemma~\ref{real-10.8} to conclude that ${\mathbb M^\epsilon}$ is also a resolution of $\operatorname{H}_0({\mathbb M^\epsilon})=\operatorname{H}_0({\mathbb L^\epsilon})$. 
\end{proof}

The following remark does not use any of the machinery of the paper and could appear almost anywhere. It provides the connection between the assertions in   Theorem~\ref{main} about resolutions and projective dimension and the assertions about grade unmixedness. The unmixedness assertions (\ref{PS}) and (\ref{main.new-f}) of Theorem~\ref{main}  
play a critical role in \cite{KPU-BA}.  
\begin{remark}\label{CMPS} Adopt Data~{\rm{\ref{87data6}}} and 
{\rm{\ref{unm-data}}}.  Assume that  
Hypothesis~{\rm\ref{87hyp6}.\ref{87hyp6-a}}
 is in effect. Let $\mathfrak p \in V(I_\goth g(\Psi)) \setminus V(I_1(\tau))$. Then the ideals $I_\goth g(\Psi)_{\mathfrak p}$  and $(I_\goth g(\Psi)+\goth C) _{\mathfrak p}$ are equal and are perfect of grade $\delta$.
\end{remark}
\begin{proof} It follows immediately from the definition of $\Psi$, as given in Data~{\rm\ref{87data6}}, that
$$\tau\circ \Psi=0;$$
therefore,  
$\tau_{\mathfrak p} \circ \Psi_{\mathfrak p}=0$ and $\operatorname{im}  \Psi_{\mathfrak p} \subset {\rm ker} \, \tau_{\mathfrak p}$. On the other hand, the choice of $\mathfrak p$ yields that the map $\tau_{\mathfrak p}: F_{\mathfrak p} \longrightarrow R_{\mathfrak p}$ is surjective and hence splits. Thus $\operatorname{ker}  \tau_{\mathfrak p}$ is a free summand of $F_{\mathfrak p}$ of rank $\goth f-1$. Therefore we may replace the target $F_{\mathfrak p}$ of the map $ \Psi_{\mathfrak p} $ by $\operatorname{ker} \tau_{\mathfrak p} \simeq R_{\mathfrak p}^{\goth f-1}$ without changing $I_\goth g(\Psi_{\mathfrak p})=I_\goth g(\Psi)_{\mathfrak p}$. By the choice of $\mathfrak p$ the ideal $I_\goth g(\Psi)_{\mathfrak p}$ is proper and has grade at least $\delta=(\goth f-1)-\goth g+1$, therefore $I_\goth g(\Psi)_{\mathfrak p}$ is perfect of grade $\delta$. 

If $\delta$ is even, then $\goth C=0$.
If $\delta$ is odd, then the inclusion $\goth C_{\mathfrak p}\subseteq I_\goth g(\Psi)_{\mathfrak p}$ is an immediate consequence of the fact $I_1(\tau)\cdot \goth C\subseteq I_\goth g(\Psi)$; see (\ref{much-promised}).
\end{proof}

\begin{proof3} (\ref{main-cx}) Remark~\ref{rmk5} explains why Lemma~\ref{MOC}, together with Remark~{\rm\ref{num-info-Med}.\ref{num-info-Med.i}}, shows that ${\mathbb M^\epsilon}$ is a complex.

\medskip\noindent(\ref{main.a},\ref{main.new-b})  Apply Lemma~\ref{resol} to conclude that ${\mathbb M^\epsilon}$ is a resolution. It is clear that $$\operatorname{H}_0(\mathbb M^\epsilon)=\begin{cases} \overline{R}, &\text{if $\epsilon$ is $\frac \delta2$ or $\frac {\delta+1}2$, and}\\
\widetilde{R},&\text{if $\epsilon=\frac{\delta-1}2$}.\end{cases}$$ Of course, if $\delta$ is even, then $\lceil \frac{\delta-1}2\rceil=\lceil \frac{\delta}2\rceil$ and $\overline{R}=\widetilde{R}$. 
\qed\end{proof3}

\bigskip In Remark~\ref{promise} we promised to explain why the  hypothesis $\operatorname{grade} I_\goth g(\Psi)\le \delta$, which appears in Theorem~\ref{main}.\ref{main.new-e.iii}, corresponds to quotients $\widetilde{R}=R/(\goth C+I_\goth g(\Psi))$ which are of no interest. In Observation~\ref{6.rmk}, we show that if
$\operatorname{grade} I_\goth g(\Psi)\le \delta$ fails (in the context of Theorem~\ref{main}.\ref{main.new-e.iii}), then $\widetilde {R}=0$. We notice that Theorem~\ref{main}.\ref{main.new-b} remains correct and meaningful in this situation: $\mathbb M^\epsilon$ is a resolution of $0$; in other words,  $\mathbb M^\epsilon$ is split exact.
\begin{observation}\label{6.rmk}  Adopt Data~{\rm{\ref{87data6}}} and 
{\rm{\ref{unm-data}}} with $\goth g=1$ and $\delta$ odd.  
Assume that Hypothesis~{\rm\ref{87hyp6}} are in effect and that $I_1(\tau)$ is a proper ideal of $R$. Then the following statements are equivalent:
\begin{enumerate}[\rm(a)]\item\label{6.rmk.a} $\goth C+I_\goth g(\Psi)$ is a proper ideal of $R$,\item\label{6.rmk.b} $\operatorname{grade} I_\goth g(\Psi)\le \delta$, and \item\label{6.rmk.c} $I_{\goth g}(\Psi)\neq I_1(\tau)$.\end{enumerate}
\end{observation}
\begin{proof}
Assume (\ref{6.rmk.a}) holds. Then Hypothesis~\ref{87hyp6}.\ref{87hyp6-a}, (\ref{above}), and Theorem~\ref{main}.\ref{main.new-b} together with (\ref{last}) give
$$\delta\le \operatorname{grade} I_\goth g(\Psi)\le\operatorname{grade} (\goth C+I_\goth g(\Psi))\le \operatorname{pd}\widetilde{R}\le \goth f-1=\delta;$$ hence $\operatorname{grade} I_{\goth g}(\Psi)=\delta$ and (\ref{6.rmk.b}) holds.  At this point,
$$\operatorname{grade} I_\goth g(\Psi)= \delta=\goth f-1<\goth f=\operatorname{grade} I_1(\tau)$$ and (\ref{6.rmk.c}) holds.

Assume (a) does not hold. Then, according to (\ref{much-promised}),
$$R=\goth C+I_\goth g(\Psi)\subseteq I_\goth g(\Psi):I_1(\tau),$$ so 
$I_1(\tau)\subseteq I_\goth g(\Psi)$. The other inclusion always holds (because of the definition of $\Psi$ in Data~\ref{87data6}); hence $I_\goth g(\Psi)=I_1(\tau)$ and (\ref{6.rmk.c}) does not hold. At this point, Hypothesis~\ref{87hyp6}.\ref{87hyp6-b} gives
$$\delta <\goth f\le \operatorname{grade} I_1(\tau)=\operatorname{grade} I_{\goth g}(\Psi),$$ and (\ref{6.rmk.b}) does not hold.
\end{proof}

\begin{proof3.5}
\medskip\noindent(\ref{main.b}) The length of ${\mathbb M^\epsilon}$, which is given in Remark~{\rm\ref{num-info-Med}.\ref{num-info-Med.a}},  is an upper bound on the projective dimension of $\operatorname{H}_0(\mathbb M^\epsilon)$. On the other hand, the last non-zero map of $\mathbb M^\epsilon$ is recorded in 
Observation~\ref{last-map}. We show that in each case there is at least one 
basis elements $u$ in ${\mathbb M^\epsilon}_{N_{\max}}$ and a proper ideal $J$ of $R$ with  with $$d_{N_{\max}}(u)\in J\cdot {\mathbb M^\epsilon}_{N_{\max}-1};$$and therefore, there is no shorter projective resolution of $\operatorname{H}_0(\mathbb M^\epsilon)$.
The last maps (\ref{4.6.2'}), (\ref{4.6.4'}), (\ref{4.6.5'}), and (\ref{4.6.6'}) all have the desired property because of the hypothesis that $I_1(\tau)$ is a proper ideal of $R$. 
Consider the last map (\ref{4.6.3'}) with the additional hypothesis that $\operatorname{grade} I_\goth g(\Psi)\le \delta$.
Observation~\ref{6.rmk} shows that $\goth C+I_{\goth g}(\Psi)$ is a proper ideal of $R$. The image of (\ref{4.6.3'}) is contained in 
$(\goth C+I_\goth g(\Psi)){\mathbb M^\epsilon}_{\goth f-2}$, since, according to (\ref{4.6.1}), 
$$ \xi(X_1^{\delta-1})=(-1)^{\epsilon-1}[D(\mu)](X_1^{(\frac{\delta+1}2)})\otimes X_1^{(\frac{\delta-1}2)}\in \goth C\cdot V^{\mathbb B}_{\goth f,\epsilon-1};$$ see (\ref{2.3.3.5}) for the final inclusion.
Now consider the last map (\ref{4.6.1'}). We treat the cases $2\le \goth g$ and $\goth g=1$ separately.
If $2\le \goth g$, then  (\ref{more**}) shows that the image of (\ref{4.6.1'}) is contained in  $I_1(\Psi)\cdot {\mathbb M^\epsilon}_0$; and this is sufficient because $I_1(\Psi)\subseteq I_1(\tau)$.
If $\goth g=1$, then we must impose the additional hypothesis that $\operatorname{grade} I_\goth g(\Psi)\le 
\delta$. Once the hypothesis is imposed, then Observation~\ref{6.rmk} again shows that $\goth C+I_\goth g(\Psi)$ is a proper ideal of $R$; and this is sufficient because the image of  (\ref{4.6.1'}) is $(\goth C+I_g(\Psi)){\mathbb M^\epsilon}_0$.

\medskip\noindent(\ref{PS},\ref{main.new-e}) 
Let $M$ be the $R$-module $\overline{R}$, if $\delta$ is even; or else $\widetilde{R}$, if $\delta$ is odd. Fix  ${\mathfrak p} 
\in \operatorname{Ass}_RM$. Assertion~(\ref{main.b})  shows that 
$$\operatorname{pd}_RM\le \goth f-1<\operatorname{grade} I_1(\tau)$$ and so the Auslander-Buchsbaum formula
$$\operatorname{depth} R_{\mathfrak p}=\operatorname{pd}_{R_{\mathfrak p}}M_{\mathfrak p}+ \operatorname{depth} M_{\mathfrak p}$$ yields that $\operatorname{depth} R_{\mathfrak p}=\operatorname{pd}_{R_{\mathfrak p}}M_{\mathfrak p} <\operatorname{grade} I_1(\tau)$; hence,  $I_1(\tau)$ cannot be contained in $\mathfrak p$.  
 Remark \ref{CMPS} now shows that the  ideals $I_\goth g(\Psi)_{\mathfrak p}$
and $(I_\goth g(\Psi)+\goth C)_{\mathfrak p}$
of $R_{\mathfrak p}$ are  perfect of grade $\delta$. It follows that 
$\operatorname{depth} R_{\mathfrak p}=\delta$.

\medskip\noindent(\ref{main.new-f}) If $\delta$ is even, then we proved in (\ref{PS}) that $I_{\goth g}(\Psi)$ is grade unmixed; so the present assertions are obvious. Let $\delta$ be odd. In light of (\ref{much-promised}) and \ref{gunm}.\ref{17.15.b} we have 
$$\goth C+I_{{\goth g}}(\Psi)\quad\subseteq \quad I_{{\goth g}}(\Psi):I_{1}(\tau)\quad\subseteq  \quad I_{{\goth g}}(\Psi):I_{1}(\tau)^\infty\quad\subseteq  \quad I_{{\goth g}}(\Psi)^{\rm{gunm}}.$$
On the other hand, assertion (\ref{main.new-e}), together with (\ref{much-promised}), shows that    $K=\goth C+I_{{\goth g}}(\Psi)$ is a grade unmixed ideal  with $I_{{\goth g}}(\Psi)\subseteq K$ and $$\operatorname{grade} K=\operatorname{grade} I_{{\goth g}}(\Psi)<\operatorname{grade} (I_{{\goth g}}(\Psi):K).$$ So, according to \ref{gunm}, once again,  $K=I_{{\goth g}}(\Psi)^{\text{gunm}}$.

\medskip\noindent(\ref{main.new-g}) This assertion follows immediately from 
(\ref{main.new-b}) and (\ref{main.new-f}). 
 \qed \end{proof3.5}

\begin{proof4} 
Fix $\epsilon$ equal to either $\frac \delta2$ or $\frac{\delta+1}2$.    Remark~\ref{num-info-Med}.\ref{num-info-Med.e} explains why there is a quasi-isomorphism from ${\mathbb M}^{\epsilon}$ to a $\goth g$-linear resolution when the component $R_0$ (of  homogeneous elements  of degree zero) in $R$ is a field. This completes the proof of (\ref{main.c.iii}). \qed\end{proof4}  

\begin{remark}\label{CM}  The proper ideal $I$ in a Noetherian ring $R$ is perfect if
$\operatorname{pd}_R\frac RI\le \operatorname{grade} I$. (See \ref{perfect} for a brief discussion of perfect ideals including their connection to the Cohen-Macaulay quotient rings.) It follows that, when all of the hypotheses of Theorem~\ref{main} are in effect, then \begin{enumerate}[\rm(a)]\item\label{CM.a}
$I_\goth g(\Psi)$ is a perfect ideal if and only if $\goth g=1$ and $\delta$ is even, and
\item
$I_\goth g(\Psi)^{\text{gunm}}$ is a perfect ideal if and only if\begin{enumerate}[\rm(i)]\item\label{CM.bi} 
 $I_\goth g(\Psi)$ is a perfect ideal, or \item\label{CM.bii} $\goth g=1$, $\delta$ is odd, and $\operatorname{grade} I_{\goth g}(\Psi)\le \delta$,  or \item\label{CM.biii} $\goth g=2$ and $\delta$ is odd, or
\item\label{CM.biv} $2\le g$ and $\delta=1$.\end{enumerate}\end{enumerate}
The ideals of (\ref{CM.a}) (and (\ref{CM.bi})) are called Huneke-Ulrich type two almost complete intersections; they had previously been resolved in \cite{K95}; see also (\ref{aci}). The ideals of (\ref{CM.bii}) are called Huneke-Ulrich deviation two Gorenstein ideals; they had previously been resolved in \cite{K86,S,K92}; see also, (\ref{gor}). 
The ideals of (\ref{CM.biii}) are discussed in (\ref{perfect-gens}). The ideals of (\ref{CM.biv}) are studied in \cite{J,M} (see also (\ref{**})); these ideals are principal.
 Most of the ideals $I_\goth g(\Psi)$ and $I_\goth g(\Psi)^{\text{gunm}}$ that are resolved in Theorem~\ref{main} are not perfect. None of these non-perfect ideals had been resolved previously.
\end{remark}

\section{The $h$-vector of $\operatorname{H}_0(\mathbb M^\epsilon)$.}\label{h-vect}
In this section we prove item (\ref{'23.19.c.ii}) of Theorem~\ref{main.c}. The proof is contained in Claims~\ref{.XNov6} and \ref{.X'Hope2}; see, in particular, Remark~\ref{.DONE}.
Two  identities (Lemmas \ref{.X'23.8} and \ref{.X'Cj25.1}) involving binomial coefficients are used in this proof. The proof of these identities is self-contained and may be found at the end of the section. We often consider binomial coefficients with negative parameters; see, for example, \cite[1.1 and 1.2]{K95} for a list  of the elementary properties of these objects.

We begin by establishing the first form of  the simplified Hilbert numerator $\operatorname{hn}_{\operatorname{H}_0(\mathbb M^\epsilon)}(s)$, as given in Theorem~\ref{main.c}.\ref{'23.19.c.ii}. Simultaneously, we calculate $\operatorname{hn}_{\operatorname{H}_0(\mathbb M^\epsilon)}(1)$ as given in  Theorem~\ref{main.c}.\ref{1.3.ii.2}.

\begin{claim}\label{.XNov6} Adopt the notation and hypotheses of Theorem~{\rm\ref{main.c}.\ref{'23.19.c.ii}}.
Then $$\operatorname{hn}_{\operatorname{H}_0(\mathbb M^\epsilon)}(s)=
\begin{cases}
\phantom{+}(1-s)^{\goth g}\sum\limits_{j\le \epsilon-1}
(-1)^{\delta+1}\binom{\goth g+j-1}{j}s^{2j+2\goth g-\goth f}
\\
+\sum\limits_{\ell=0}^{\goth g-2}\binom{\ell+\goth f-\goth g-1}{\ell}s^\ell\\
+\sum\limits_{\ell=0}^{\goth f-\goth g-2}(-1)^{\ell+\delta}\binom{\goth g+\ell-1}\ell s^{\ell+2\goth g-\goth f}
\end{cases}$$
and 
$$\operatorname{hn}_{\operatorname{H}_0(\mathbb M^\epsilon)}(1)=  
\sum\limits_{i=0}^{\lfloor \delta/2\rfloor}\binom{\goth f-2-2i}{\delta-2i},
$$
which is equal  to 
$$\begin{cases}
\text{the number of monomials of even degree at most $\delta$ in $\goth g-1$ variables,} &\text{if $\delta$ is even, or}\\
\text{the number of monomials of odd degree at most $\delta$ in $\goth g-1$ variables}, &\text{if $\delta$ is odd}.\\
\end{cases}$$
\end{claim}
\begin{proof}
If $0\to F_p\to \dots\to F_1\to F_0\to M\to 0$ is a homogeneous resolution of the graded $R$-module $M$ by finitely generated free $R$-modules, with $F_i=\bigoplus_{j=1}^{\operatorname{rank} F_i}R(-\beta_{i,j})$, then the Hilbert Series of $M$ is
$\operatorname{HS}_M(s)=\operatorname{HS}_R(s)\cdot \operatorname{HN}_M(s)$, for 
$$\operatorname{HN}_M(s)=\sum_{i=0}^p(-1)^i\sum_{j=1}^{\operatorname{rank} F_i} s^{\beta_{i,j}}.$$
Use the resolution of $\operatorname{H}_0(\mathbb M^\epsilon)$  given in Definition~\ref{doodles'}.\ref{doodles'.a} and Remark~\ref{num-info-Med}.\ref{num-info-Med.c} to see that
\begingroup\allowdisplaybreaks\begin{align}\operatorname{HN}_{\operatorname{H}_0(\mathbb M^\epsilon)}(s)=&1+\begin{cases}
\phantom{+}
\sum\limits_{\gf{j\le \epsilon -1}{\delta\le i+j}}(-1)^{i-\delta+1}
\operatorname{rank} V_{i,j}^{{\mathbb B}}s^{i+2j+2\goth g-\goth f}\\
+
\sum\limits_{\gf{\epsilon\le j}{i+j\le \delta-1}}(-1)^{i-\delta}
\operatorname{rank} V_{i,j}^{{\mathbb T}}s^{i+2j+2\goth g-\goth f}
\end{cases}
\notag\\
=&1+\begin{cases}
\phantom{+}
\sum\limits_{\gf{j\le \epsilon-1}{\delta\le i+j}}(-1)^{i-\delta+1}
\binom {\goth f} i
\binom{\goth g+j-1}{j}
s^{i+2j+2\goth g-\goth f}\\
+
\sum\limits_{\gf{\epsilon\le j}{i+j\le \delta-1}}(-1)^{i-\delta}
\binom{\goth f}i\binom{\goth g+j-1}{j}s^{i+2j+2\goth g-\goth f}.
\end{cases}
\label{.X'Nov.3}\end{align}
\endgroup
Add 
\begin{equation}\label{.X'Nov.3a}
\sum\limits_{\left\{i,j\left|\gf{j\le \epsilon-1}{i+j<\delta}\right.\right\}}(-1)^{i-\delta+1}
\binom {\goth f} i
\binom{\goth g+j-1}{j}
s^{i+2j+2\goth g-\goth f}
\end{equation} to the top  summand on the far right side of (\ref{.X'Nov.3}) and subtract (\ref{.X'Nov.3a}) from the bottom  summand on the far right side of (\ref{.X'Nov.3a}) to see that 
\begin{equation}\label{.X'Nov.3b}\operatorname{HN}_{\operatorname{H}_0(\mathbb M^\epsilon)}(s)=1+\begin{cases}
\phantom{+} 
\sum\limits_{\{i,j|j\le \epsilon -1\}}(-1)^{i-\delta+1}
\binom {\goth f} i
\binom{\goth g+j-1}{j}
s^{i+2j+2\goth g-\goth f}
\\
+
\sum\limits_{\{i,j|i+j\le \delta -1\}}(-1)^{i-\delta}
\binom{\goth f}i\binom{\goth g+j-1}{j}s^{i+2j+2\goth g-\goth f}.
\end{cases}\end{equation}
It is clear that the top summand on the far right side of (\ref{.X'Nov.3b}) is 
$$\textstyle
(1-s)^{\goth f}(-1)^{\delta+1}\Big[ \sum\limits_{\{j|j\le \epsilon-1\}}
\binom{\goth g+j-1}{j}s^{2j}\Big]s^{2\goth g-\goth f}
;$$
 and therefore, 
\begin{equation}\label{.X'Nov.3c}\operatorname{HN}_{\operatorname{H}_0(\mathbb M^\epsilon)}(s)=1+\begin{cases}
\phantom{+} 
(1-s)^{\goth f}
\sum\limits_{\{j|j\le \epsilon-1\}}(-1)^{\delta+1}
\binom{\goth g+j-1}{j}s^{2j+2\goth g-\goth f}
\\
+
\sum\limits_{\{i,j|i+j\le \delta -1\}}(-1)^{i-\delta}
\binom{\goth f}i\binom{\goth g+j-1}{j}s^{i+2j+2\goth g-\goth f}.
\end{cases}\end{equation}
We next prove that \begin{equation}\label{.X'Hope1}\begin{array}{ll}
&1
+
\sum\limits_{i+j\le \delta -1}(-1)^{i-\delta}
\binom{\goth f}i\binom{\goth g+j-1}{j}s^{i+2j+2\goth g-\goth f}
\\
=&(1-s)^{\goth f-\goth g}\Big(  
\sum\limits_{\ell=0}^{\goth g-2}\binom{\ell+\goth f-\goth g-1}{\ell}s^\ell
+\sum\limits_{\ell=0}^{\goth f-\goth g-2}(-1)^{\ell+\delta}\binom{\goth g+\ell-1}\ell
s^{\ell+2\goth g-\goth f}\Big).\end{array} \end{equation}
 Once (\ref{.X'Hope1}) 
is accomplished, then we will have shown that
\begin{equation}\label{.Xalmost-there}\operatorname{HN}_{\operatorname{H}_0(\mathbb M^\epsilon)}(s)=(1-s)^{\goth f-\goth g} \begin{cases}
\phantom{+} 
(1-s)^{\goth g}
\sum\limits_{\{j|j\le \epsilon-1\}}
(-1)^{\delta+1}\binom{\goth g+j-1}{j}s^{2j+2\goth g-\goth f}
\\
+\sum\limits_{\ell=0}^{\goth g-2}\binom{\ell+\goth f-\goth g-1}{\ell}s^\ell\\
+\sum\limits_{\ell=0}^{\goth f-\goth g-2}(-1)^{\ell+\delta}\binom{\goth g+\ell-1}\ell
s^{\ell+2\goth g-\goth f}.
\end{cases}\end{equation}

Observe that  \begingroup\allowdisplaybreaks\begin{align}{\rm (\ref{.X'Hope1})}&\iff 1=
\begin{cases}  
\phantom{+}(1-s)^{\goth f-\goth g}\sum\limits_{\ell=0}^{\goth g-2}\binom{\ell+\goth f-\goth g-1}{\ell}s^\ell\\
+(1-s)^{\goth f-\goth g}\sum\limits_{\ell=0}^{\goth f-\goth g-2}(-1)^{\ell+\delta}\binom{\goth g+\ell-1}\ell s^{\ell+2\goth g-\goth f}\\
-\sum\limits_{i+j\le \delta -1}(-1)^{i-\delta}
\binom{\goth f}i\binom{\goth g+j-1}{j}s^{i+2j+2\goth g-\goth f}\end{cases}\notag\\
&\iff 1=
\begin{cases}  
\phantom{+}\sum\limits_{k}\ \sum\limits_{\ell\le \goth g-2}(-1)^k\binom{\goth f-\goth g}k
\binom{\ell+\goth f-\goth g-1}{\ell}s^{k+\ell}\\
+\sum\limits_{k}\sum\limits_{\ell\le \goth f-\goth g-2}(-1)^{k+\ell+\delta}\binom{\goth f-\goth g}k \binom{\goth g+\ell-1}\ell s^{k+\ell+2\goth g-\goth f}\\
+\sum\limits_{k+\ell\le \goth f-\goth g -1}(-1)^{k+\delta+1}
\binom{\goth f}k\binom{\goth g+\ell-1}{\ell}s^{k+2\ell+2\goth g-\goth f}\end{cases}\notag\\
&\iff
1=
\begin{cases}  
\phantom{+}\sum\limits_{0\le L\le \goth f-2}\ \sum\limits_{\ell\le \goth g-2}(-1)^{L-\ell}\binom{\goth f-\goth g}{L-\ell}
\binom{\ell+\goth f-\goth g-1}{\ell}s^{L}\\
+\sum\limits_{2\goth g-\goth f\le L\le \goth f-2}\ \sum\limits_{\ell\le \goth f-\goth g-2}(-1)^{L+\goth g}\binom{\goth f-\goth g}{L-2\goth g+\goth f-\ell} \binom{\goth g+\ell-1}\ell s^{L}\\
+\sum\limits_{2\goth g-\goth f\le L\le \goth f-2}\ \sum\limits_{
L-\goth g +1\le \ell}
(-1)^{L+\goth g+1}
\binom{\goth f}{L-2\ell-2\goth g+\goth f}\binom{\goth g+\ell-1}{\ell}s^{L}\end{cases}\notag\\
&\iff
1=
\begin{cases}  
\phantom{+}\sum\limits_{L\le \goth f-2}\ \sum\limits_{\ell\le \goth g-2}(-1)^{L-\ell}\binom{\goth f-\goth g}{L-\ell}
\binom{\ell+\goth f-\goth g-1}{\ell}s^{L}\\
+\sum\limits_{L\le f-2}\ \sum\limits_{\ell\le \goth f-\goth g-2}(-1)^{L+\goth g}\binom{\goth f-\goth g}{L-2\goth g+\goth f-\ell} \binom{\goth g+\ell-1}\ell s^{L}\\
+\sum\limits_{L\le f-2}\ \sum\limits_{
L-\goth g +1\le \ell}
(-1)^{L+\goth g+1}
\binom{\goth f}{L-2\ell-2\goth g+\goth f}\binom{\goth g+\ell-1}{\ell}s^{L}.\end{cases}\notag\end{align}\endgroup
Look at one coefficient at a time,  multiply both sides of the equation by $(-1)^L$, and recall the meaning of $\chi(S)$ from (\ref{chi}), in order to see that 
$${\rm (\ref{.X'Hope1})}\iff
\chi(L=0)=\begin{cases} 
\phantom{+}\sum\limits_{\ell\le \goth g-2}(-1)^{\ell}\binom{\goth f-\goth g}{L-\ell}
\binom{\ell+\goth f-\goth g-1}{\ell}\\
+\sum\limits_{\ell\le \goth f-\goth g-2}(-1)^{\goth g}\binom{\goth f-\goth g}{\goth g-L+\ell} \binom{\goth g+\ell-1}\ell \\
+\sum\limits_{L-\goth g +1\le \ell}
(-1)^{\goth g+1}
\binom{\goth f}{2\ell+2\goth g-L}\binom{\goth g+\ell-1}{\ell},\end{cases}$$
for each  integer $L$, with $L\le \goth f-2$. Observe that
\begingroup\allowdisplaybreaks\begin{align}
&\textstyle\sum\limits_{\ell\le \goth g-2}(-1)^{\ell}\binom{\goth f-\goth g}{L-\ell}
\binom{\ell+\goth f-\goth g-1}{\ell}
\notag\\
=&\textstyle\sum\limits_{\ell\le \goth g-2}(-1)^{\ell}
\binom{\ell+\goth f-\goth g-1}{\ell}\binom{\goth f-\goth g}{\goth f-\goth g-L+\ell},&&\text{because $0\le \goth f-\goth g$,}\notag\\
=&\textstyle
(-1)^{\goth g+L+\goth f}\sum\limits_{m\le \goth f-L-2}(-1)^{m}
\binom{L-1+m}{\goth g+L-\goth f+m}\binom{\goth f-\goth g}{m},
&&\text{where $m=\goth f-\goth g-L+\ell$,}
\notag\\
=&\textstyle(-1)^{\goth g+L+\goth f}(-1)^{\goth f-L}\binom{L-1}{\goth g-2}\binom{\goth f-2}{\goth f-L-2},&&\hskip-4pt\begin{array}{l}\text{by  (\ref{.X'23.9}) with $J=\goth f-L-1$,}\\
\text{$G=\goth g-1$ and $F=\goth f-2$,}\end{array}\notag\\
=&\textstyle(-1)^{\goth g}\binom{\goth f-2}{L}\binom{L-1}{\goth g-2},&&\text{because $0\le \goth f-2$.}\notag \end{align}\endgroup
Formula (\ref{.X'23.9}) requires that $1\le L$. On the other hand, it is easy to see that 
$$\textstyle\sum\limits_{\ell\le \goth g-2}(-1)^{\ell}\binom{\goth f-\goth g}{L-\ell}
\binom{\ell+\goth f-\goth g-1}{\ell}= \textstyle(-1)^{\goth g}\binom{\goth f-2}{L}\binom{L-1}{\goth g-2},$$ whenever $L\le 0$.
 It follows that
$${\rm (\ref{.X'Hope1})}\iff
\chi(L=0)=\begin{cases} 
\phantom{+}(-1)^{\goth g}\binom{\goth f-2}{L}\binom{L-1}{\goth g-2}\\
+\sum\limits_{\ell\le \goth f-\goth g-2}(-1)^{\goth g}\binom{\goth f-\goth g}{\goth g-L+\ell} \binom{\goth g+\ell-1}\ell \\
+\sum\limits_{L-\goth g +1\le \ell}
(-1)^{\goth g+1}
\binom{\goth f}{2\ell+2\goth g-L}\binom{\goth g+\ell-1}{\ell},\end{cases}$$ for each integer $L$ with $L\le \goth f-2$.
According to Lemma~\ref{.X'Cj25.1}, the right side of the previous display holds. It follows that (\ref{.X'Hope1}) and (\ref{.Xalmost-there}) also hold. We finish the proof of Claim~\ref{.XNov6} by calculating the value of
\begin{equation}\label{.*}\frac{\operatorname{HN}_{\operatorname{H}_0(\mathbb M^\epsilon)}(s)}{(1-s)^{\goth f-\goth g}}=\begin{cases}
\phantom{+} 
(1-s)^{\goth g}\sum\limits_{\{j|j\le \epsilon-1\}}
(-1)^{\delta+1}\binom{\goth g+j-1}{j}s^{2j+2\goth g-\goth  f}
\\
+\sum\limits_{\ell=0}^{\goth g-2}\binom{\ell+\goth f-\goth g-1}{\ell}s^\ell\\
+\sum\limits_{\ell=0}^{\goth f-\goth g-2}(-1)^{\ell+\delta}\binom{\goth g+\ell-1}\ell s^{\ell+2\goth g-\goth f}
\end{cases}\end{equation} at $s=1$ and observing that this value is not zero. The
value is
\begingroup \allowdisplaybreaks \begin{align}&\textstyle\sum\limits_{\ell=0}^{\goth g-2}\binom{\ell+\goth f-\goth g-1}{\ell}
+\sum\limits_{\ell=0}^{\goth f-\goth g-2}(-1)^{\ell+\delta}\binom{\goth g+\ell-1}\ell
\notag\\=&\textstyle\binom{\goth f-2}{\goth g-2}+\sum\limits_{\ell=0}^{\goth f-\goth g-2}(-1)^{\ell+\delta}\binom{\goth g+\ell-1}\ell\notag\\
=&\textstyle\sum\limits_{\ell=0}^{\goth f-\goth g}(-1)^{\ell+\delta}\binom{\goth g+\ell-1}\ell,&&\textstyle\text{because $\binom{\goth f-2}{\goth g-2}=-\binom{\goth f-2}{\goth f-\goth g-1}+\binom{\goth f-1}{\goth f-\goth g}$,}\notag\\
=&\textstyle\sum\limits_{k=0}^{\delta}(-1)^{k}\binom{\goth f-1-k}{\delta-k},&&\textstyle\text{where $k=\delta-\ell$,}\notag\\
=&\textstyle\sum\limits_{0\le k}(-1)^{k}\binom{\goth f-1-k}{\delta-k},&&\textstyle\text{because $\binom{\goth f-1-k}{\delta-k}=0$ for $\delta<k$,}\notag\\
=&\textstyle\sum\limits_{0\le i}\Big[\binom{\goth f-1-2i}{\delta-2i}-\binom{\goth f-1-(2i+1)}{\delta-(2i+1)}\Big],&&\textstyle\text{write $k=2i$ or $k=2i+1$,}\notag\\
=&\textstyle\sum\limits_{0\le i}\binom{\goth f-2-2i}{\delta-2i}\notag\\
=&\textstyle\sum\limits_{i=0}^{\lfloor \delta/2\rfloor}\binom{\goth f-2-2i}{\delta-2i},\notag
\end{align}\endgroup
which equals 
$$\begin{cases}
\text{the number of monomials of even degree at most $\delta$ in $\goth g-1$ variables,} &\text{if $\delta$ is even, or}\\
\text{the number of monomials of odd degree at most $\delta$ in $\goth g-1$ variables}, &\text{if $\delta$ is odd}.\\
\end{cases}$$
This completes the proof of Claim~\ref{.XNov6}.
\end{proof}

\begin{remark}\label{6.3}This remark is a continuation of Remark~\ref{rmk}.b. If $\goth g=1$ and $\goth f$ is even, then there are no monomials of odd degree in $0$ variables. On the other hand, one can easily verify that the right side of (\ref{.*}) is $1-s$.
\end{remark}

\begin{claim}\label{.X'Hope2} The polynomial $\operatorname{hn}_{\operatorname{H}_0(\mathbb M^\epsilon)}(s)$, as calculated in Claim~{\rm\ref{.XNov6}}, is also equal to
$$\textstyle \sum\limits_{\ell=0}^{\goth g-1}\binom{\ell+\goth f-\goth g-1}\ell s^\ell-\chi(\epsilon=\frac{\delta-1}2)\binom{\goth g+\epsilon-1}{\epsilon
}s^{\goth g-1}+\sum\limits_{\ell=\goth g}^{q(\goth g,\goth f)} 
\sum\limits_{j\le
\epsilon
-1}(-1)^{\ell+\goth g+1}\binom{\goth g}{\ell-2\goth g+\goth f-2j}
\binom{\goth g+j-1}{j}s^{\ell},$$ for $$q(\goth g,\goth f)=\begin{cases}
2\goth g-3,&\text{if $\epsilon=\frac{\delta-1}2$,}\\
2\goth g-2,&\text{if $\epsilon=\frac{\delta}2$, and}\\2\goth g-1,&\text{if $\epsilon=\frac{\delta+1}2$}.\end{cases}$$
\end{claim}
\begin{remark}\label{.DONE}Once Claim~\ref{.X'Hope2} is established, then the proof of Theorem~\ref{main.c}.\ref{'23.19.c.ii} is complete. Indeed, the first form of $\operatorname{hn}_{\operatorname{H}_0(\mathbb M^\epsilon)}(s)$ is established in Claim~\ref{.XNov6}, the second form in Claim~\ref{.X'Hope2},
 Theorem~\ref{main.c}.\ref{1.3.ii.1} is a routine statement which is included to set the meaning of the notation in the reader's mind,   Theorem~\ref{main.c}.\ref{1.3.ii.1+} maybe read directly from Claim~\ref{.X'Hope2}, and Theorem~\ref{main.c}.\ref{1.3.ii.2} is established in Claim~\ref{.XNov6}.
\end{remark}

\begin{proof}
Expand $(1-s)^\goth g$ and observe that
$$  \textstyle\begin{array}{ll}
&(1-s)^\goth g\sum\limits_{j\le \epsilon-1}(-1)^{\delta+1}
\binom{\goth g+j-1}{j}s^{2j+2\goth g-\goth f}\\
=&\sum\limits_{\ell=2\goth g-\goth f}^{q(\goth g,\goth f)} 
\sum\limits_{j\le\epsilon -1}(-1)^{\ell+\goth g+1}\binom{\goth g}{\ell-2\goth g+\goth f-2j}
\binom{\goth g+j-1}{j}s^{\ell}.\end{array}$$
The difference between the two formulations for $\operatorname{hn}_{\operatorname{H}_0(\mathbb M^\epsilon)}(s)$ is 
\begingroup\allowdisplaybreaks \begin{align} 
&\begin{cases} \phantom{+}\text{the formulation of Claim~\ref{.XNov6}} \\\hline- \text{the formulation of Claim~\ref{.X'Hope2}}\end{cases}\notag\\
=&\begin{cases}  
\phantom{+}\sum\limits_{\ell=2\goth g-\goth f}^{q(\goth g,\goth f)} 
\sum\limits_{j\le\epsilon-1}(-1)^{\ell+\goth g+1}\binom{\goth g}{\ell-2\goth g+\goth f-2j}
\binom{\goth g+j-1}{j}s^{\ell}\\
+\sum\limits_{\ell=0}^{\goth g-2}\binom{\ell+\goth f-\goth g-1}{\ell}s^\ell\\
+\sum\limits_{\ell=2\goth g-\goth f}^{\goth g-2}(-1)^{\ell+\goth g}\binom{\ell-\goth g+\goth f-1}{\ell-2\goth g+\goth f} s^{\ell}\\\hline
-\sum\limits_{\ell=0}^{\goth g-1}\binom{\ell+\goth f-\goth g-1}\ell s^\ell\\
+\chi(\epsilon=\frac{\delta-1}2)\binom{\goth g+\epsilon-1}{\epsilon
}s^{\goth g-1}\\
-\sum\limits_{\ell=\goth g}^{q(\goth g,\goth f)} 
\sum\limits_{j\le
\epsilon
-1}(-1)^{\ell+\goth g+1}\binom{\goth g}{\ell-2\goth g+\goth f-2j}
\binom{\goth g+j-1}{j}s^{\ell}\\
\end{cases}
\notag\\
=&\begin{cases}  
\phantom{+}\sum\limits_{\ell=2\goth g-\goth f}^{\goth g-1} 
\sum\limits_{j\le \epsilon-1}(-1)^{\ell+\goth g+1}\binom{\goth g}{\ell-2\goth g+\goth f-2j}
\binom{\goth g+j-1}{j}s^{\ell}\\
+\chi(\epsilon=\frac{\delta-1}2)\binom{\goth g+\epsilon-1}{\epsilon
}s^{\goth g-1}\\
+\sum\limits_{\ell=2\goth g-\goth f}^{\goth g-2}(-1)^{\ell+\goth g}\binom{\ell-\goth g+\goth f-1}{\ell-2\goth g+\goth f} s^{\ell} \\
-\sum\limits_{\ell=\goth g-1}\binom{\ell+\goth f-\goth g-1}\ell s^\ell.
\end{cases}\label{.Xtag-me}
\end{align}\endgroup
The  bottom two summands of (\ref{.Xtag-me}) combine to form
$$\textstyle \sum\limits_{\ell=2\goth g-\goth f}^{\goth g-1}(-1)^{\ell+\goth g} \binom{\ell-\goth g+\goth f-1}{\ell-2\goth g+\goth f} s^{\ell}.$$Keep in mind that $\epsilon$ is equal to either $\frac{\delta-1}2$ or $\lceil\frac \delta 2\rceil$. We split the top summand of (\ref{.Xtag-me}) into two pieces  and see that (\ref{.Xtag-me}) is
\begin{align}
=&\begin{cases}  
\phantom{+}\sum\limits_{\ell=2\goth g-\goth f}^{\goth g-1} 
\sum\limits_{j\le \lceil\frac \delta 2\rceil-1}(-1)^{\ell+\goth g+1}\binom{\goth g}{\ell-2\goth g+\goth f-2j}
\binom{\goth g+j-1}{j}s^{\ell}\\
-\chi(\epsilon=\frac{\delta-1}2)\sum\limits_{\ell=2\goth g-\goth f}^{\goth g-1} 
\sum\limits_{j= \epsilon}(-1)^{\ell+\goth g+1}\binom{\goth g}{\ell-2\goth g+\goth f-2j}
\binom{\goth g+j-1}{j}s^{\ell}\\
+\chi(\epsilon=\frac{\delta-1}2)\binom{\goth g+\epsilon-1}{\epsilon
}s^{\goth g-1}\\
+\sum\limits_{\ell=2\goth g-\goth f}^{\goth g-1}(-1)^{\ell+\goth g}\binom{\ell-\goth g+\goth f-1}{\ell-2\goth g+\goth f} s^{\ell}. \\
\end{cases}\label{.Xtag-me2}
\end{align}
The second and third summands of (\ref{.Xtag-me2}) add to zero because in the second summand if $\ell\le \goth g-2$, then $\ell-2\goth g+\goth f-2j\le -1$ and $\binom{\goth g}{\ell-2\goth g+\goth f-2j}=0$. Thus,  the difference between the two formulations of  
for $\operatorname{hn}_{\operatorname{H}_0(\mathbb M^\epsilon)}(s)$ is 
 \begin{equation}\label{.Xlast-eq}\textstyle \sum\limits_{\ell=2\goth g-\goth f}^{\goth g-1} (-1)^{\ell+\goth g}
\Big[-\sum\limits_{j\le\lceil\frac{\delta}2\rceil-1}\binom{\goth g}{\ell-2\goth g+\goth f-2j}
\binom{\goth g+j-1}{j}+  \binom{\ell-\goth g+\goth f-1}{\ell-2\goth g+\goth f}\Big]s^{\ell}.\end{equation}
The constraint $j\le \lceil\frac{\delta}2\rceil-1$ in (\ref{.Xlast-eq}) is redundant because if $\lceil\frac{\delta}2\rceil \le j$, then 
$$\ell-2\goth g+\goth f-2j\le (\goth g-1)-2\goth g+\goth f-(\goth f-\goth g)=-1$$ and $\binom{\goth g}{\ell-2\goth g+\goth f-2j}=0$. It follows that
$$\textstyle \sum\limits_{j\le\lceil\frac{\delta}2\rceil-1}\binom{\goth g}{\ell-2\goth g+\goth f-2j}\binom{\goth g+j-1}{j}=\sum\limits_{j\in \mathbb Z}\binom{\goth g}{\ell-2\goth g+\goth f-2j}\binom{\goth g+j-1}{j};$$ and this number 
is the coefficient of $x^{\ell-2\goth g+\goth f}$ in
$$(1+x)^\goth g\frac{1}{(1-x^2)^\goth g}=\frac{1}{(1-x)^\goth g}=\sum\limits_{j=0}^{\infty}\binom{\goth g+j-1}{j}x^j.$$ This coefficient is $\binom{\ell-\goth g+f-1}{\ell-2\goth g+f}$; the quantity in (\ref{.Xlast-eq}) is zero; and the proof is complete.
\end{proof}

\bigskip We  used the identities (\ref{.X'23.9})  and (\ref{.X25.1gts})  in the proof of Claims \ref{.XNov6} and \ref{.X'Hope2}. We now prove these identities. Recall that we often consider binomial coefficients with negative parameters and that  \cite[1.1 and 1.2]{K95} contains  a list  of the elementary properties of these objects. In particular, we often use  \cite[1.2.g]{K95}, which says that if $a$, $b$, and $c$ are integers, then
\begin{equation}\label{.X'1.2.g}\textstyle 0\le a\implies\sum\limits_{k\in \mathbb Z}(-1)^k\binom{b+k}{c+k}\binom ak=(-1)^a\binom b{a+c}.\end{equation}

\begin{lemma}\label{.X'23.8}If $F$, $G$, and $J$ are integers with $0\le G\le F$ and $0\le J\le F$, then
\begin{align}
\label{.X'23.8.1}\textstyle (-1)^J\binom{F-J}{G-1}\binom{F}{J-1}&\textstyle=\sum\limits_{m=J}^{F-G+1}
(-1)^m\binom{F-J+m}{G-J+m}\binom{F-G+1}{m}\\
 \intertext{and}
\textstyle (-1)^{J-1}\binom{F-J}{G-1}\binom{F}{J-1}&\textstyle =\sum\limits_{m=0}^{J-1}
(-1)^m\binom{F-J+m}{G-J+m}\binom{F-G+1}{m}. \label{.X'23.9}
\end{align}
\end{lemma}

\begin{proof} We first prove (\ref{.X'23.8.1}). Let $L(F,G,J)$ and $R(F,G,J)$ represent the left and right sides of (\ref{.X'23.8.1}), respectively. In other words, 

$$\begin{array}{ll}
L(F,G,J)=(-1)^J\binom{F-J}{G-1}\binom{F}{J-1}\ &\text{and}\\
R(F,G,J)=\sum\limits_{m=J}^{F-G+1}(-1)^m
\binom{F-J+m}{G-J+m}\binom{F-G+1}{m}.\end{array}$$
The proof is by induction. 

First consider $J=0$ and $0\le G\le F$. 
Observe that  $L(F,G,0)=0$ and apply (\ref{.X'1.2.g})
to see that $$R(F,G,0)=\textstyle \sum\limits_{m=0}^{F-G+1}(-1)^m
\binom{F+m}{G+m}\binom{F-G+1}{m}=(-1)^{F-G+1}\binom{F}{F+1}=0.$$
Thus, $L(F,G,0)=R(F,G,0)$ for all $F$ and $G$ with $0\le G\le F$.

Next consider $G=0$ and $0\le J\le F$. Observe that $L(F,0,J)=0$. Apply (\ref{.X'1.2.g}) once again to see that 
$$\textstyle R(F,0,J)=\sum\limits_{m=J}^{F+1}(-1)^m
\binom{F-J+m}{m-J}\binom{F+1}{m}=(-1)^{F+1}\binom{F-J}{F+1-J}=0.$$ Thus,
$L(F,0,J)=R(F,0,J)$ for all $F$ and $J$ with $0\le J\le F$.

Also, notice that if $F+1<G+J$, then $L(F,G,J)$ and $R(F,G,J)$ are both zero. Indeed, $0\le F-J<G-1$ forces $\binom{F-J}{G-1}=0$; hence $L(F,G,J)=0$. Also, $F-G+1<J$; so the sum involved in  $R(F,G,J)$ is empty and $R(F,G,J)=0$.

Now we consider the main case: $1\le J$, $1\le G$, and $G+J\le F+1$. Observe that $$\begin{array}{llll}
R(F,G,J)&=& R(F-1,G-1,J-1)+(-1)^J\binom{F-1}{G-1}\binom{F-G+1}{J-1}\vspace{5pt}\\
&=&L(F-1,G-1,J-1)+(-1)^J\binom{F-1}{G-1}\binom{F-G+1}{J-1},&\quad\text{by induction},\vspace{5pt}\\ 
&=&(-1)^{J-1}\binom{F-J}{G-2}\binom{F-1}{J-2}+(-1)^J\binom{F-1}{G-1}\binom{F-G+1}{J-1}.\end{array}$$
One may now quickly verify that $L(F,G,1)=R(F,G,1)$ for all $F$ and $G$ with $0\le G\le F$ and $L(F,1,J)=R(F,1,J)$ for all $F$ and $J$ with $0\le J\le F$. 
 Henceforth, we assume that 
\begin{equation}\label{.X'hyp}2\le J,\quad  2\le G,\quad  \text{and} \quad G+J\le F+1.\end{equation}
The ambient hypotheses, together with hypotheses (\ref{.X'hyp}), ensure that
$$\begin{array}{l}0\le G-2\le F-J,\quad 0\le J-2\le F-1,\quad 0\le G-1\le F-1,\quad 0\le J-1\le F-G+1,\\
0\le G-1\le F-J,\quad 0\le J-1\le F
\end{array}
$$ and therefore, each of the relevant binomial coefficients $\binom ab$ is equal to $\frac{a!}{b!(b-a)!}$. In order to show that $R(F,G,J)=L(F,G,J)$ it suffices to show 
$$\textstyle (-1)^{J-1}\binom{F-J}{G-2}\binom{F-1}{J-2}+(-1)^J\binom{F-1}{G-1}\binom{F-G+1}{J-1}=(-1)^J\binom{F-J}{G-1}\binom{F}{J-1}.$$
In other words, it suffices to show 
\begin{equation}\label{.X'STS}\textstyle \binom{F-1}{G-1}\binom{F-G+1}{J-1}=\binom{F-J}{G-2}\binom{F-1}{J-2}+\binom{F-J}{G-1}\binom{F}{J-1}.\end{equation}
The right side of (\ref{.X'STS}) is
$$\begin{array}{ll}&
\frac{(F-1)!(F-J)!}{(G-1)!(F-J-G+2)!(J-1)!(F+1-J)!}\Big((G-1)(J-1)+F(F-J-G+2)\Big)\vspace{5pt}\\
=&\frac{(F-1)!(F-J)!}{(G-1)!(F-J-G+2)!(J-1)!(F+1-J)!}(F-J+1)(F-G+1)\vspace{5pt}\\
=&\binom{F-1}{G-1}\binom{F-G+1}{J-1},
\end{array}$$which is the left side of (\ref{.X'STS}). This completes the proof of (\ref{.X'23.8.1}).

 Apply  (\ref{.X'1.2.g}) to see that
\begin{equation}\label{.black-dot}\textstyle \sum\limits_{m=0}^{F-G+1}
(-1)^m\binom{F-J+m}{G-J+m}\binom{F-G+1}{m}=(-1)^{F-G+1}\binom{F-J}{F-J+1}=0.\end{equation} (The right-most equality in (\ref{.black-dot}) holds  since $0\le F-J< F-J+1$.) The formula (\ref{.X'23.9}) follows readily because, if $F-G+1<J$, then the right side of (\ref {.X'23.9}) is equal to the left side of (\ref{.black-dot}), and, if $J\le F-G+1$, then (\ref {.X'23.9}) is obtained by subtracting (\ref{.X'23.8.1}) from (\ref{.black-dot}).
\end{proof}

\begin{lemma}\label{.X'Cj25.1} If $\goth f$, $\goth g$, and $L$ are integers with $1\le \goth g\le \goth f$, then 
\begin{equation}\label{.X25.1gts}(-1)^\goth g\chi(L=0)\chi(2\le f)=\begin{cases} 
+\binom{\goth f-2}{L}\binom{L-1}{\goth g-2}\\
+\sum\limits_{\ell\le \goth f-\goth g-2}\binom{\goth f-\goth g}{\goth g-L+\ell} \binom{\goth g+\ell-1}\ell \\
-\sum\limits_{L-\goth g +1\le \ell}
\binom{\goth f}{2\ell+2\goth g-L}\binom{\goth g+\ell-1}{\ell}.\end{cases}\end{equation}
\end{lemma}

\begin{proof} The proof is by induction. The base cases are treated in Claims~\ref{.Xclaim0}, \ref{.Xclaim1}, \ref{.Xclaim2}, and \ref{.Xclaim3}.
\begin{claim}\label{.Xclaim0}Lemma~{\rm\ref{.X'Cj25.1}} holds if $\goth g=1$.\end{claim}
\begin{proof}
Let $\theta(L,\goth f)$ be the right side of (\ref{.X25.1gts}), after $\goth g$ has been set equal to $1$. In other words, $$\theta(L,\goth f)=\textstyle\sum\limits_{0\le \ell\le \goth f-3}\binom{\goth f-1}{
1-L+\ell} 
-\sum\limits_{\max\{L,0\}\le \ell}
\binom{\goth f}{
2-L+2\ell}.$$ Observe that $\theta(L,1)=0$ for all $L$. 
We complete the proof of Claim~\ref{.Xclaim0} by showing that $\theta(L,\goth f)=-\chi(L=0)$ for all $\goth f$ with $2\le \goth f$.

A short calculation gives
\begin{equation}\label{.28.9gts}\textstyle \theta(L,\goth f+1)=\theta(L,\goth f)+\theta(L+1,\goth f)+[\chi(2\le \goth f)-\chi(0\le L)] \binom{\goth f}{
L+1},\end{equation} for all $\goth f$ and $L$. Apply (\ref{.28.9gts}), with $\goth f=1$, to conclude that $\theta(L,2)=-\chi(L=0)$. 
 For $2\le \goth f$, apply induction and (\ref{.28.9gts}) to see that 
$$\begin{array}{lll} \theta(L,\goth f+1)&=&-\chi(L=0)-\chi(L=-1)+[1-\chi(0\le L)] \binom{\goth f}{
L+1}\\&=&-\chi(L=0)-\chi(L=-1)+\chi(L=-1)\\&=&-\chi(L=0).\end{array}$$
The proof of Claim~\ref{.Xclaim0} is complete. \end{proof}

\begin{claim}\label{.Xclaim1}Lemma~{\rm\ref{.X'Cj25.1}} holds if $L\le \goth g-2$.\end{claim}
\begin{proof}In light of Claim~\ref{.Xclaim0}, we may assume that $2\le \goth g$. Let
$$\begin{array}{llll}A&=&\binom{\goth f-2}{L}\binom{L-1}{\goth g-2},\\
B&=&\sum\limits_{\ell\le \goth f-\goth g-2}\binom{\goth f-\goth g}{\goth g-L+\ell} \binom{\goth g+\ell-1}\ell,& \text{and} \\
C&=&\sum\limits_{L-\goth g +1\le \ell}
\binom{\goth f}{2\ell+2\goth g-L}\binom{\goth g+\ell-1}{\ell}.
\end{array}$$
Observe first that
$$A=\left\{\begin{array}{lll} 0,\quad&\text{if $L\le -1$}, &\text{because $\binom{\goth f-2}{L}=0$,}\\(-1)^\goth g,\quad&\text{if $L=0$,}\\0,\quad&\text{if $1\le L\le \goth g-2$,}&\text{because $\binom{L-1}{\goth g-2}=0$};\end{array}\right.$$ thus, 
$$(-1)^\goth g\chi(L=0)=A, \quad\text{when $L\le \goth g-2$}.$$It remains to show that 
\begin{equation}\label{.Xrts}L\le \goth g-2 \implies B=C.\end{equation}
Notice that if $L\le \goth g-2$, then the constraint  $\ell\le \goth f-\goth g-2$ is redundant in $B$ and the constraint $L-\goth g+1\le \ell$ is redundant in $C$. Indeed, if $\goth f-\goth g-1\le \ell  $, then 
$$\goth f-\goth g+1= 2+(\goth f-\goth g-1)\le (\goth g-L)+\ell$$ and   $\binom{\goth f-\goth g}{\goth g-L+\ell}=0$. In a similar manner, if $\ell\le L-\goth g$, then 
$$\ell\le L-\goth g\le -2$$ and $\binom{\goth g+\ell-1}{\ell}=0$.
When $L\le \goth g-2$, then 
$$\textstyle B=\sum\limits_{\ell\in \mathbb Z}\binom{\goth f-\goth g}{\goth g-L+\ell} \binom{\goth g+\ell-1}\ell
=\sum\limits_{\ell\in \mathbb Z}\binom{\goth f-\goth g}{\goth f-2\goth g+L-\ell} \binom{\goth g+\ell-1}\ell$$
is equal to the coefficient of $x^{\goth f-2\goth g+L}$ in $(1+x)^{\goth f-\goth g}\cdot \frac 1{(1-x)^\goth g}$; and 
$$\textstyle C=\sum\limits_{\ell\in \mathbb Z}
\binom{\goth f}{2\ell+2\goth g-L}\binom{\goth g+\ell-1}{\ell}=\sum\limits_{\ell\in \mathbb Z}
\binom{\goth f}{\goth f-2\goth g+L-2\ell}\binom{\goth g+\ell-1}{\ell}$$ is equal to the coefficient of $x^{\goth f-2\goth g+L}$ in $(1+x)^{\goth f}\cdot \frac 1{(1-x^2)^\goth g}$. The two coefficients are equal; (\ref{.Xrts}) has been established, and the proof of Claim~\ref{.Xclaim1} is complete.
\end{proof}

\begin{claim}\label{.Xclaim2}Lemma~{\rm\ref{.X'Cj25.1}} holds if $\goth g=\goth f$.\end{claim}
\begin{proof}In light of Claims~\ref{.Xclaim0} and \ref{.Xclaim1}, it suffices to assume that
$1\le \goth g-1\le L$. We  show that $A+B+C=0$ for
$$\begin{array}{ll}
A=\binom{\goth g-2}{L}\binom{L-1}{\goth g-2},\\
B=\sum\limits_{\ell\le -2}\binom{0}{\goth g-L+\ell} \binom{\goth g+\ell-1}\ell,&\text{and} \\
C=-\sum\limits_{L-\goth g +1\le \ell}
\binom{\goth g}{2\ell+2\goth g-L}\binom{\goth g+\ell-1}{\ell}.\end{array}$$
The fact that  $0\le\goth g-2<L$, ensures that  $\binom{\goth g-2}L=0$; and therefore,   $A=0$. The sum $B$ is zero because the parameter $\ell$ in the binomial coefficient $\binom{\goth g+\ell-1}\ell$ is negative. Suppose  that the parameter $\ell$  makes a non-zero contribution to $C$. Then  $0\le \ell$ because of the binomial coefficient $\binom{\goth g+\ell-1}{\ell}$. It follows that  $\ell\le 2\ell$. On the other hand, $\ell$ also satisfies $L-\goth g+1\le \ell$ and $0\le L-\goth g-2\ell$. Thus, 
$$L-\goth g+1\le \ell\le 2\ell\le L-\goth g,$$which is impossible, and  $C$ is zero. This completes the proof of Claim~\ref{.Xclaim2}.
\end{proof}

\begin{claim}\label{.Xclaim3}Lemma~{\rm\ref{.X'Cj25.1}} holds if $\goth f-1\le L$.\end{claim}

\begin{proof}The hypothesis $1\le \goth g\le \goth f$, combined with Claim~\ref{.Xclaim0}, allows us to assume that $2\le \goth f$. In this case, each of the constituent pieces of (\ref{.X25.1gts}) is zero.
It is clear that $\chi(L=0)$ and $\binom{\goth f-2}L$ are both zero. The summand 
$\sum\limits_{\ell\le \goth f-\goth g-2}\binom{\goth f-\goth g}{\goth g-L+\ell} \binom{\goth g+\ell-1}\ell$ is zero because, $$\ell\le \goth f-\goth g-2\implies \goth g-L+\ell\le \goth f-2-L\le -1\implies\binom{\goth f-\goth g}{\goth g-L+\ell}=0.$$
The summand $\sum\limits_{L-\goth g +1\le \ell}
\binom{\goth f}{2\ell+2\goth g-L}\binom{\goth g+\ell-1}{\ell}$ is zero because, if $L-\goth g+1\le \ell$, then 
$$\goth f+1\le (L+1)+1\le (\ell+\goth g)+(\ell+\goth g-L) =2\ell+2\goth g-L$$
and $\binom{\goth f}{2\ell+2\goth g-L}=0$. This completes the proof of Claim~\ref{.Xclaim3}. \end{proof}

We now carry out the induction step in the proof of Lemma~\ref{.X'Cj25.1}. We assume that (\ref{.X25.1gts}) holds at $(\goth f,\goth g,L)$ and $(\goth f,\goth g,L+1)$ and we prove that (\ref{.X25.1gts}) holds at $(\goth f+1,\goth g,L)$. In light of  Claims \ref{.Xclaim0}, \ref{.Xclaim1}, \ref{.Xclaim2}, and \ref{.Xclaim3}, we may assume that
\begin{equation}\label{.Xconstraints}
2\le \goth g\le \goth f\quad \text{and}\quad \goth g-1\le L\le \goth f-1.\end{equation}
Let $$\textstyle \theta(\goth f,\goth g,L)=
\sum\limits_{\ell\le \goth f-\goth g-2}\binom{\goth f-\goth g}{\goth g-L+\ell} \binom{\goth g+\ell-1}\ell \\
-\sum\limits_{L-\goth g +1\le \ell}
\binom{\goth f}{2\ell+2\goth g-L}\binom{\goth g+\ell-1}{\ell}.
$$To complete the proof, we must prove that
\begin{equation}\label{.Xgoal25.1}\textstyle \binom{\goth f-1}L\binom{L-1}{\goth g-2}+\theta(\goth f+1,\goth g,L)=0.\end{equation}
Write
$\theta(\goth f+1,\goth g,L)=A+B$ for
$$\textstyle    
A=\sum\limits_{\ell\le \goth f-\goth g-1}\binom{\goth f+1-\goth g}{\goth g-L+\ell} \binom{\goth g+\ell-1}\ell\quad \text{and}\quad
B=-\sum\limits_{L-\goth g +1\le \ell}
\binom{\goth f+1}{2\ell+2\goth g-L}\binom{\goth g+\ell-1}{\ell}.
$$
Observe that $A=A_1+A_2$ with $$\textstyle A_1=\sum\limits_{\ell\le \goth f-\goth g-2}\binom{\goth f+1-\goth g}{\goth g-L+\ell} \binom{\goth g+\ell-1}\ell\quad\text{and}\quad A_2=\sum\limits_{\ell= \goth f-\goth g-1}\binom{\goth f+1-\goth g}{\goth g-L+\ell} \binom{\goth g+\ell-1}\ell
=\binom{\goth f+1-\goth g}{\goth f-L-1} \binom{\goth f-2}{\goth f-\goth g-1}.
$$Use Pascal's identity to write $A_1=A_1'+A_1''$ and $B=B_1+B_2$ with
$$\textstyle 
\begin{array}{lllllll}A_1'&=&\sum\limits_{\ell\le \goth f-\goth g-2}\binom{\goth f-\goth g}{\goth g-L+\ell} \binom{\goth g+\ell-1}\ell,&&
A_1''&=&\sum\limits_{\ell\le \goth f-\goth g-2}\binom{\goth f-\goth g}{\goth g-L+\ell-1} \binom{\goth g+\ell-1}\ell,\\
\textstyle B_1&=&-\sum\limits_{L-\goth g +1\le \ell}
\binom{\goth f}{2\ell+2\goth g-L}\binom{\goth g+\ell-1}{\ell},&\text{ and }& B_2&=&
-\sum\limits_{L-\goth g +1\le \ell}
\binom{\goth f}{2\ell+2\goth g-L-1}\binom{\goth g+\ell-1}{\ell}.\end{array}$$
Separate $B_2=B_2'+ B_2''$ with

$$\textstyle 
\begin{array}{lllllll}
B_2'&=&
-\sum\limits_{L-\goth g +2\le \ell}
\binom{\goth f}{2\ell+2\goth g-L-1}\binom{\goth g+\ell-1}{\ell}&\text{ and }&
 B_2''&=&-
\sum\limits_{L-\goth g +1= \ell}
\binom{\goth f}{2\ell+2\goth g-L-1}\binom{\goth g+\ell-1}{\ell}\\
&&&&&=&-\binom{\goth f}{L+1}\binom{L}{L-\goth g +1}.\end{array}
$$
Thus, we apply induction to see that \begin{align}&\textstyle\text{the left side of (\ref{.Xgoal25.1})}\vspace{5pt}\notag\\
=&\textstyle\binom{\goth f-1}L\binom{L-1}{\goth g-2}+\theta(\goth f+1,\goth g,L)\vspace{5pt}\notag\\
=&\textstyle\binom{\goth f-1}L\binom{L-1}{\goth g-2}+(A_1'+A_1''+A_2)+(B_1+B_2'+B_2'')\vspace{5pt}\notag\\
=&\textstyle\binom{\goth f-1}L\binom{L-1}{\goth g-2}+(A_1'+B_1)+(A_1''+B_2')+A_2+B_2''\vspace{5pt}\notag\\
=&\textstyle\binom{\goth f-1}L\binom{L-1}{\goth g-2}+\theta(\goth f,\goth g,L)+\theta(\goth f,\goth g,L+1)+A_2+B_2''\vspace{5pt}\notag\\
=&\textstyle\Big(\binom{\goth f-1}L\binom{L-1}{\goth g-2}-\binom{\goth f-2}{L}\binom{L-1}{\goth g-2}\Big)-\binom{\goth f-2}{L+1}\binom{L}{\goth g-2}+\binom{\goth f+1-\goth g}{\goth f-L-1} \binom{\goth f-2}{\goth f-\goth g-1}-\binom{\goth f}{L+1}\binom{L}{L-\goth g +1}\vspace{5pt}\notag\\
=&\textstyle\binom{\goth f-2}{L-1}\binom{L-1}{\goth g-2}-\binom{\goth f-2}{L+1}\binom{L}{\goth g-2}+\binom{\goth f+1-\goth g}{\goth f-L-1} \binom{\goth f-2}{\goth f-\goth g-1}-\binom{\goth f}{L+1}\binom{L}{L-\goth g +1}
\label{.Xcrit-sum}\end{align}
If $L=\goth f-1$ or $L=\goth f-2$, then $\binom{\goth f-2}{L+1}=0$ and (\ref{.Xcrit-sum}) is equal to
$$\textstyle\binom{\goth f-2}{\goth f-2}\binom{\goth f-2}{\goth g-2}
+\binom{\goth f+1-\goth g}{0} \binom{\goth f-2}{\goth f-\goth g-1}-\binom{\goth f}{\goth f}
\binom{\goth f-1}{\goth f-\goth g}=0$$or
$$\textstyle \binom{\goth f-2}{\goth f-3}\binom{\goth f-3}{\goth g-2}+
(\goth f+1-\goth g) \binom{\goth f-2}{\goth g-1}-\goth f\binom{\goth f-2}{\goth g-1}=0,$$respectively. 
If $\goth g=\goth f$, then the constraint $\goth g-1\le L\le \goth f-1$ from (\ref{.Xconstraints}) forces $L=\goth g-1$; and therefore, 
(\ref{.Xcrit-sum}) is equal to 
$$\textstyle\binom{\goth g-2}{\goth g-2}\binom{\goth g-2}{\goth g-2}-\binom{\goth g-2}{\goth g}\binom{\goth g-1}{\goth g-2}+\binom{1}{0} \binom{\goth g-2}{-1}-\binom{\goth g}{\goth g}\binom{\goth g-1}{0}=0.$$
We now evaluate (\ref{.Xcrit-sum}) in the remaining cases. That is, we assume
$L\le \goth f-3$, $g\le \goth f-1$, and the constraints of (\ref{.Xconstraints}) are in effect. In this case,  every binomial coefficient $\binom ab$ which appears in (\ref{.Xcrit-sum}) satisfies $0\le b\le a$ and may be evaluated as $\binom ab=\frac{a!}{b!(a-b)!}$. It follows that (\ref{.Xcrit-sum}) is equal to
$$\textstyle \frac{(\goth f-2)!}{(\goth f-L-1)!(L+1)(\goth g-1)!(L-\goth g+2)!}\begin{cases}\phantom{+}(L+1)(\goth g-1)(L-\goth g+2)\\
-(\goth f-L-2)(\goth f-L-1)(\goth g-1)\\+(L+1)(\goth f-\goth g)(\goth f-\goth g+1)\\-\goth f(\goth f-1)(L-\goth g+2),\end{cases}$$ and this is zero. The proof of Lemma~\ref{.X'Cj25.1} is complete.
\end{proof}

Related techniques are used in \cite{KPU-HS} to give an  explicit formula for the Hilbert Series of an algebra defined by a linearly presented, standard graded, residual intersection of a grade three Gorenstein ideal.

\section{Application to blowup algebras.}\label{BA}

Our original motivation for this paper was to find the defining equations for linearly presented height three Gorenstein ideals. We conclude the paper by summarizing the relevant results of the present paper in the language of \cite{KPU-BA}. The integers  $\goth f$ and $\goth g$ from the present paper become $n$ and $d$ in \cite{KPU-BA}.

\begin{data}\label{data8}
 Let $k$ be a field and let $n$ and $d$ be positive integers. 
Let  $\varphi$ be an $n\times n$ alternating matrix with linear entries from the polynomial ring $R=k[x_1,\dots,x_d]$ and  $B$ be the $d\times n$ matrix with linear entries from $T=k[T_1,\dots,T_n]$ such that  the matrix equation
\begin{equation}\label{JD}[T_1,\dots,T_n]\cdot \varphi=[x_1,\dots,x_d]\cdot B\end{equation}
holds. \end{data}

\begin{definition}\label{C-of-phi} Adopt Data~\ref{data8}. We define a homogeneous ideal $C(\varphi)$ in   $T$. Consider the $(n+d)\times(n+d)$ alternating matrix $$\goth B=\bmatrix \varphi&-B^{\rm t}\\B&0\endbmatrix,$$with entries in the polynomial ring $k[x_1,\dots,x_d,T_1,\dots,T_n]$.
Let $\goth B'$ be the matrix which is obtained from $\goth B$ by deleting the final row and column. View the Pfaffian $\Pf(\goth B')$ of $\goth B'$ as a polynomial in $T[x_1,\dots,x_d]$ and  
let $C(\varphi)=c_{T}(\Pf(\goth B'))$ be
 the content of $\Pf(\goth B')$.\end{definition}
\begin{Remarks}Retain the notation of Definition~\ref{C-of-phi}.
\begin{enumerate}[\rm(a)]
\item If $n+d$ is even then $C(\varphi)=0$.
\item The ideal $C(\varphi)$ of $T$ is generated by  homogeneous forms  of degree $d-1$. 
\end{enumerate}
  \end{Remarks}

\begin{theorem}\label{unmpart}
Adopt   Data~{\rm \ref{data8}} with $3\le d<n$ and $n$ odd. If $n-d\le\grade I_d(B)$, then the following statements hold.
\begin{enumerate}[\rm(a)]
\item\label{8.3.a} The unmixed part of $I_d(B)$ is equal to
$I_d(B)^{\text{\rm unm}}=I_d(B)+C(\varphi)$.
\item\label{8.3.b} The multiplicity of 
$T/I_d(B)^{\text{\rm unm}}$ is 
$$\sum\limits_{i=0}^{\lfloor (n-d)/2\rfloor}\binom{n-2-2i}{n-d-2i}.$$

\item\label{8.3.c} The ring $T/I_d(B)^{\text{\rm unm}}$ is Cohen-Macaulay on the punctured spectrum and $$\operatorname{depth} (T/I_d(B)^{\text{\rm unm}})=\begin{cases}
1,&\text{if $d$ is odd},\\
2,&\text{if $d$ is even and $d+3\le n$, and} \\
n-1,&\text{if $d+1=n$}.
\end{cases}$$
\item\label{8.3.d} If $d$ is odd, then the  ideal  $I_d(B)$ is unmixed;
furthermore, $I_d(B)$
 has a linear minimal resolution and $\operatorname{reg}(T/I_d(B))=d$.
\item\label{8.3.e} The Hilbert series of $T/I_d(B)^{\text{\rm unm}}$ is 
$$\textstyle \frac{\sum\limits_{\ell=0}^{d-2}\binom{\ell+n-d-1}{\ell}z^\ell + \sum\limits_{\ell=0}^{n-d-2}(-1)^{\ell+d+1}\binom{d+\ell-1}\ell z^{2d-n+\ell}}{(1-z)^d} +(-1)^d\sum\limits_{j\le \lceil\frac{n-d-3}{2}\rceil}
\binom{d+j-1}{j}z^{2j+2d-n}  .$$

\end{enumerate}\end{theorem}
\begin{Remark}The conclusions of assertion (\ref{8.3.d}) do not hold when $d$ is even: the ideal $I_d(B)$ is mixed and the minimal resolution of $I_d(B)^{\rm unm}$ is not linear; see for example, (\ref{notlinear}). 
\end{Remark}
\begin{proof} 
We begin working in the polynomial ring $S=k[x_1,\dots,x_d,T_1,\dots,T_n]$. 
Write $\varphi$ as $\sum_{\ell=1}^dx_\ell\varphi_\ell$, where $\varphi_1,\dots,\varphi_d$ are $n\times n$ alternating matrices with entries from the field $k$.
Observe that the equation (\ref{JD}) gives rise to
\begin{align}\label{8.1}[x_1,\dots,x_d]B&=[T_1,\dots,T_n]\varphi=[T_1,\dots,T_n](x_1\varphi_1+\dots+x_d\varphi_d)
\\&=[x_1,\dots,x_d]\left[\begin{matrix} [T_1,\dots,T_n]\varphi_1\\\hline \vdots\\\hline [T_1,\dots,T_n]\varphi_d\end{matrix}\right].\notag\end{align}
The entries of  $B$ and the entries of the final matrix in the previous display  
are linear forms from the polynomial ring $T=k[T_1,\dots,T_n]$; and therefore the equation
(\ref{8.1}) can occur only if
\begin{equation}\label{B}B= \left[\begin{matrix} [T_1,\dots,T_n]\varphi_1\\\hline \vdots\\\hline [T_1,\dots,T_n]\varphi_d\end{matrix}\right].\end{equation}
 In other words, the entry of $B$ in row $a$ and column $b$ is
\begin{equation}\label{Bab}B_{a,b}=\sum_i T_i(\varphi_a)_{i,b}\end{equation} for all $a$ and $b$.

Now we work in the polynomial ring $T=k[T_1,\dots,T_n]$ and 
we convert  the data of \ref{data8} into Data~\ref{87data6}.
Let $F$ and $G$ be free $T$-modules with bases  $e_1,\dots,e_n$ and $h_1,\dots,h_d$, respectively; and let $e_1^*,\dots,e_n^*$ and $h_1^*,\dots,h_d^*$ be the corresponding dual bases for $F^*$ and $G^*$, respectively. Define  $$\tau:F\to T\quad\text{and}\quad \mu: G^*\to \textstyle\bigwedge^2F$$ by
$$\tau(e_i)=T_i\quad \text{and}\quad \mu(h_\ell^*)=\sum_{1\le i<j\le n}(\varphi_{\ell})_{i,j} \,e_i\wedge e_j.$$
Notice that 
\begin{align*}\Psi(h_\ell^*)&=\tau(\mu(h_\ell^*))=\sum_{1\le i<j\le n}(\varphi_{\ell})_{i,j} \tau(e_i\wedge e_j)=\sum_{1\le i<j\le n}(\varphi_{\ell})_{i,j} (T_i e_j-T_je_i)\\
 &=\sum_{1\le i,j\le n}(\varphi_{\ell})_{i,j} T_i e_j=\sum_{j=1}^n B_{\ell,j}e_j,\end{align*}
where the final equality is due to (\ref{Bab}).
Therefore, $B^{\rm t}$ is the matrix for $\Psi$, with respect to the chosen bases. It follows that $I_d(\Psi)=I_d(B)$ and both hypothesis of \ref{87hyp6} are satisfied and we may apply the results of the present paper. Assertion (\ref{8.3.b}) is Theorem~\ref{main.c}.\ref{1.3.ii.2}. The fact that $T/(I_d(B))^{\rm umn}$ is Cohen-Macaulay on the punctured spectrum  is contained in Remark~\ref{CMPS}.  
Theorem~\ref{main}.\ref{main.b} gives the projective dimension  of $T/(I_d(B))^{\rm umn}$ as a $T$-module; and hence the depth of $T/(I_d(B))^{\rm umn}$ by way of the Auslander-Buchsbaum formula. 
Both assertions in  (\ref{8.3.c}) have been established.
Assertion~(\ref{8.3.d}) is Theorem~\ref{main}.\ref{PS} and Theorem~\ref{main.c}.\ref{main.c.iii}; and assertion~(\ref{8.3.e}) is Theorem~\ref{main.c}.\ref{'23.19.c.ii}.  

Now we prove (\ref{8.3.a}). Theorem~\ref{main}.\ref{main.new-f} shows that $I_d(B)^{\rm unm}=I_d(B)+\goth C$, for $\goth C$ from (\ref{unm-data-b}). We prove that the $T$-ideals $\goth C$ and $C(\varphi)$ are equal. If $Z$ is a free $T$-module of rank one and $z\in Z$, then let $o(z)$ represent the annihilator of $Z/(z)$. 
According to (\ref{2.3.3.5}), the $T$-ideal $\goth C$ of  (\ref{unm-data-b}) is 
generated by 
\begin{equation}\label{goth-C}\begin{array}{l}\left\{o\Big((\mu(h_1^*))^{(u_1)}\wedge \cdots\wedge (\mu(h_d^*))^{(u_d)}\wedge({\textstyle\bigwedge}^{d-1}\Psi)(h_1^*\wedge \dots \wedge h_{d-1}^*)\Big)\right\}\\  
\text{where $u_1+\dots +u_d
={\textstyle \frac{n-d+1}2}$ and  $1\le u_d$}.
\end{array}
\end{equation}
(Actually (\ref{2.3.3.5}) treats $1$ as the distinguished index and we have treated $d$ as the distinguished index. This switch makes no difference, only a sign is changed and even the sign change disappears as soon as we look at the order ideal.)

The homomorphism $\Psi: G^*\to F$, with $\Psi(h^*_\ell)=\sum_{j=1}^nB_{\ell,j}e_j$, for $1\le \ell\le d$, is naturally the same as the element
\begin{equation}\label{naturally}\sum_{\gf{1\le j\le n}{1\le \ell\le d}}B_{\ell,j}e_j \otimes h_\ell\in F\otimes G,\end{equation}which we call $\underline{\Psi}$.
Let $G'$ represent the submodule $\bigoplus_{\ell=1}^{d-1} T h_\ell'$ of $G$,
 $G^{\prime*}$ represent the submodule $\bigoplus_{\ell=1}^{d-1} T h_\ell^{\prime*}$ of $G^*$,  
$\Psi':G^{\prime*}\to F$ represent the restriction of $\Psi:G^*\to F$ to $G^{\prime*}$, and $\underline{\Psi'}$ be the element of $F\otimes G'$ which corresponds to $\Psi':G^{\prime*}\to F$. Notice that $\frak B'$ is the alternating matrix which is associated to the element 
$$\Phi=\sum_{\ell=1}^dx_\ell\mu(h_\ell^*)+\underline{\Psi'}+0\in \textstyle S\otimes_T (\bigwedge^2F\oplus (F\otimes G')\oplus \bigwedge^2G')=S\otimes_T\bigwedge^2(F\oplus G'). $$
So $\Phi^{(\frac{n+d-1}2)}=\Pf(\mathfrak B')\cdot \omega$,
where $\omega$ is a  basis element of $S\otimes_T\textstyle\bigwedge^{n+d-1}(F\oplus G')$. In other words,
\begin{equation}\label{may6-1}o(\Phi^{(\frac{n+d-1}2)})=(\Pf(\mathfrak B')).\end{equation}

The $r$-th divided power of $\Phi$ is given by
\begingroup\allowdisplaybreaks\begin{align}\Phi^{(r)}=\Big(\sum_{\ell=1}^dx_\ell\mu(h_\ell^*)+\underline{\Psi'}\Big)^{(r)}
&=\sum\limits_{i=0}^r \Big[\Big(\sum_{\ell=1}^dx_\ell\mu(h_\ell^*)\Big)^{(i)} \wedge (\underline{\Psi'})^{(r-i)}\Big]\notag\\&\in S\otimes_T\sum\limits_{i=0}^r\Big[\textstyle\bigwedge^{2i}F\wedge \bigwedge^{r-i}(F\otimes G')\Big];\notag\\ 
\intertext{hence, rank considerations yield}  \Phi^{(\frac{n+d-1}2)}=\Big(\sum_{\ell=1}^dx_\ell\mu(h_\ell^*)+\underline{\Psi'}\Big)^{(\frac{n+d-1}2)}
&= \Big(\sum_{\ell=1}^dx_\ell\mu(h_\ell^*)\Big)^{(\frac{n-d+1}2)} \wedge 
(\underline{\Psi'})^{(d-1)}\label{may6-2}\\&\in \textstyle S\otimes_T\bigwedge^{n-d+1}F\wedge \bigwedge^{d-1}
(F\otimes G')\notag\\&\textstyle=S\otimes_T\bigwedge^{n+d-1}(F\oplus G').\notag
\end{align}\endgroup
Combine (\ref{may6-1}) and (\ref{may6-2}), use the rules of divided powers,  and use the fact that if $M$ is a square matrix, then
$$\Pf\left(\begin{array}{c|c}0&M\\\hline-M^{\rm t}&0\end{array}\right)=\det M,$$ to see that
\begin{align*}
&(\Pf(\mathfrak B'))\\=&o\Big[\Big(\sum_{\ell=1}^dx_\ell\mu(h_\ell^*)\Big)^{(\frac{n-d+1}2)} \wedge (\underline{\Psi'})^{(d-1)}\Big]\\
=&o\Big[\sum\limits_{u_1+\dots+u_d=\frac{n-d+1}2}x_1^{u_1}\cdots x_d^{u_d}
(\mu(h_1^*))^{(u_1)}\wedge \cdots\wedge  (\mu(h_d^*))^{(u_d)}\wedge \textstyle(\bigwedge^{d-1}\Psi')(h_1^*\wedge \cdots \wedge h_{d-1}^*)\Big].
\end{align*}
If $u_d=0$, then \begin{align*}&(\mu(h_1^*))^{(u_1)}\wedge \cdots\wedge  (\mu(h_{d-1}^*))^{(u_{d-1})}\wedge \textstyle(\bigwedge^{d-1}\Psi')(h_1^*\wedge \cdots \wedge h_{d-1}^*)\\
=&\tau\Big((\mu(h_1^*))^{(u_1+1)}\wedge (\mu(h_2^*))^{(u_2)}\wedge\cdots\wedge  (\mu(h_{d-1}^*))^{(u_{d-1})}\wedge \textstyle(\bigwedge^{d-2}\Psi')(h_2^*\wedge \cdots \wedge h_{d-1}^*)\Big)\\ \in &\tau(\textstyle\bigwedge^{n+1}F)=0.\end{align*}
So, $(\Pf(\mathfrak B'))$ is equal to
$$o\Big[\sum\limits_{\gf{u_1+\dots+u_d=\frac{n-d+1}2}{1\le u_d}}x_1^{u_1}\cdots x_d^{u_d}
(\mu(h_1^*))^{(u_1)}\wedge \cdots\wedge  (\mu(h_d^*))^{(u_d)}\wedge \textstyle(\bigwedge^{d-1}\Psi')(h_1^*\wedge \cdots \wedge h_{d-1}^*)\Big].$$
The ideal $C(\varphi)$ of Definition~\ref{C-of-phi} is the content of $\Pf(\mathfrak B')$, when $\Pf(\mathfrak B')$ is viewed as an element of $T[x_1,\dots,x_d]$; hence, 
$C(\varphi)$ is the $T$-ideal  generated by
$$\left\{ o\left(
(\mu(h_1^*))^{(u_1)}\wedge \cdots\wedge  (\mu(h_d^*))^{(u_d)}\wedge \textstyle(\bigwedge^{d-1}\Psi')(h_1^*\wedge \cdots \wedge h_{d-1}^*)\right)\left|\begin{array}{l} \sum_{\ell=1}^d u_\ell=\frac{n-d+1}2\\1\le u_d\end{array}\right.\right\}.$$
A quick comparison with (\ref{goth-C}) shows that $\goth C=C(\varphi)$. The proof is complete.
\end{proof}

\end{document}